\documentclass{article}
\usepackage[english]{babel}
\usepackage{amsmath}
\usepackage{amsthm}
\usepackage{amssymb}
\usepackage{amscd, enumerate}
\usepackage[latin1]{inputenc}
\usepackage{pstricks}
\usepackage{graphicx}
\usepackage[all]{xy}

\newtheorem{teor}{Theorem}[section]
\newtheorem{defi}{Definition}
\newtheorem{lema}[teor]{Lemma}
\newtheorem{prop}[teor]{Proposition}
\newtheorem{cor}[teor]{Corollary}
\newtheorem{rem}[teor]{Remark}

\newtheorem{notat}[teor]{Notation}

\begin{document}

\pagestyle{plain}

\begin{center}

\large{\bf{The Hochschild Cohomology ring of preprojective algebras
of type $\mathbb{L}_n$}}

\vspace{1cm}

Estefanía Andreu Juan

\vspace{0.5cm}

Department of Mathematics, University of Murcia

Campus de Espinardo

30100 Murcia

Spain

Phone number: +34 868884181

Fax number: +34 868884182

\ E-mail address: eaj1@um.es

\vspace{1cm}
\end{center}

\begin{center} \textbf{ABSTRACT}

\end{center}

We compute the Hochschild Cohomology of a finite-dimensional
preprojective algebra of gene\-ra\-lized Dynkin type $\mathbb{L}_n$
over a field of characteristic different from 2 .  In particular, we
describe the ring structure of the Hochschild Cohomology ring under
the Yoneda product by giving an explicit presentation by generators
and relations.

\vspace{0.4cm}


\vspace{0.4cm}

\noindent Classification Code: 16Gxx ; Representation theory of
rings and algebras.

\section{Introduction}

Given a nonoriented finite graph $\Delta$, with $\Delta_0$ and
$\Delta_1$ as sets of vertices and edges, respectively, the
\emph{preprojective algebra} of (type) $\Delta$, denoted $P(\Delta
)$, is the algebra given by quiver and relations as follows. The
quiver $Q:=Q_\Delta$ of $P(\Delta )$ has the same vertices and the
same loops as $\Delta$. Then, for each edge
$i\frac{\hspace{0.2cm}}{\hspace{0.7cm}}j$ in $\Delta$ which is not a
loop, $Q$ will have two opposite arrows $a:i\rightarrow j$ and
$\bar{a}:j\rightarrow i$. Convening that $\bar{a}=a$ whenever $a$ is
a loop, the algebra is subject to as many relations as vertices in
$\Delta$, namely, one relation $\sum_{i(a)=i}a\bar{a}=0$ per each
$i\in Q_0=\Delta_0$ (here $i(a)$ denotes the initial vertex of $a$).

The notion of preprojective algebra first appeared in the late 70s
in the work of Gelfand and Ponomarev \cite{GP} to study the
representation theory of a finite quiver without oriented cycles.
They found their first applications in classification problems of
algebras of finite type (\cite{DR1}, \cite{DR2}) and have been
linked to universal enveloping algebras and cluster algebras
(\cite{GLS1}, \cite{GLS2}). They also occur in very diverse parts of
mathematics. For instance, they play a special role in Lusztig's
perverse sheaf approach to quantum groups (\cite{L1}, \cite{L2}) and
have been used to tackle differential geometry problems \cite{K} or
to study non-commutative deformations of Kleinian singularities
\cite{CBH}.

It is well known that $P(\Delta)$ is finite-dimensional if and only
if $\Delta$ is a disjoint union of generelized Dynkin graphs,
$\mathbb{A}_n,$ $\mathbb{D}_n,$ $ \mathbb{E}_6,$ $ \mathbb{E}_7,$ $
\mathbb{E}_8$ or $\mathbb{L}_n$, where $\mathbb{L}_n$ is the graph:

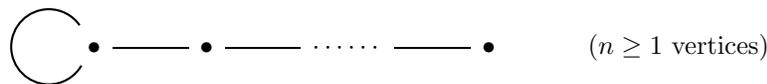
\vspace{0.5cm}

\begin{center}

\psarc{-}(1,0.1){0.5}{35}{330} \hspace{1.4cm} $\bullet$
$\frac{\hspace{0.2cm}}{\hspace{1cm}}$ $\bullet$
$\frac{\hspace{0.2cm}}{\hspace{1cm}}$ $\cdots \cdots$
$\frac{\hspace{0.2cm}}{\hspace{1cm}}$ $\bullet$ \hspace{1cm} ($n\geq
1$ vertices)

\end{center}

\vspace{0.5cm}

The aim of this paper is to determine the  structure of the
classical and stable Hochschild cohomology rings of a preprojective
algebra of type $\mathbb{L}_n$ over an algebraically closed field
$K$, provided that $Char(K)$ is different from 2, and the structure
as graded modules over them of the classical and stable Hochschild
homology. For preprojective algebras of Dynkin type, the  structure
of the Hochschild cohomology ring is known for type $\mathbb{A}_n$
in arbitrary characteristic (\cite{ES1}, \cite{ES2}) and, in the
case of a field of characteristic zero, for types $\mathbb{D}_n$ and
$\mathbb{E}$ \cite{C-H.E}.

An important common feature of the preprojective algebras of
generalized Dynkin type is that, except for $\Delta=\mathbb{A}_1$,
$P(\Delta)$ is ($\Omega-$)periodic of period at most 6 (thus
self-injective) where $\Omega$ is the Heller's syzygy operator (see
\cite{RS} and \cite{ES3} for the Dynkin cases and \cite{BES} for the
case $\mathbb{L}_n$). The multiplicative structure of the Hochschild
cohomology ring $HH^*(\Lambda )$ for a self-injective finite
dimensional algebra $\Lambda$ is of great interest  in connection
with the study of varieties of modules  and with questions about its
relationship with the Yoneda algebra of $\Lambda$. This is the
graded algebra $E(\Lambda )=Ext_\Lambda^*(\Lambda /J,\Lambda /J)$,
where $J=J(\Lambda )$ denotes the Jacobson radical of $\Lambda$.
Indeed, with inspiration from modular representation theory of
finite groups, where the theory of varieties of modules had been
developed by Carlson (\cite{C1}, \cite{C2}), Benson (\cite{B}) and
others, Snashall and Solberg (\cite{SnSo}, see also \cite{EHSST})
started the study of varieties of modules over arbitrary finite
dimensional algebras, replacing the group cohomology ring
$HH^*(G,K)$ by the Hochschild cohomology ring $HH^*(\Lambda )$ of
the considered algebra $\Lambda$. For the new theory to be
satisfactory it is generally required that $\Lambda$ is
self-injective and it is necessary that the algebra $HH^*(\Lambda )$
satisfies some finite generation conditions, which are always
satisfied when $\Lambda$ is periodic. On the other side, there is a
canonical homomorphism of graded algebras $\varphi :HH^*(\Lambda
)\longrightarrow E(\Lambda )$ whose image is contained in the graded
center $Z^{gr}(E(\Lambda ))$ of $E(\Lambda )$ \cite{SnSo}. It is
known that, in case $\Lambda$ is a Koszul algebra, one has
$Im(\varphi )=Z^{gr}(E(\Lambda ))$ \cite{BGSS}, but little else is
known on the inclusion $Im(\varphi )\subseteq Z^{gr}(E(\Lambda ))$
for general algebras. Related with this question, there is an
intriguing open problem (\cite{GSS}) which asks wether
$\Omega$-periodicity of $\Lambda /J$ as a $\Lambda$-module implies
that $\Lambda$ is a periodic algebra.

The above questions suggest that finding patterns of behaviour of
the homogeneous elements of $HH^*(\Lambda )$ with respect to the
Yoneda product, in particular cases where the multiplicative
structure of $HH^*(\Lambda )$ is computable,  can help to give some
hints on how to tackle them.

The two main results of the paper are the following, from which all
the desired structures (classical and stable Hochschild homology and
cohomology) are described (see Corollaries \ref{structure of
homology as graded module} and \ref{cor:presentation:stable
cohomology ring}).

\begin{teor}\label{HH^* with Char not 2 not dividing 2n+1} Let $\Lambda
=P(\mathbb{L}_n)$ and suppose that $Char(K)\neq 2$ and $Char(K)$
does not divide $2n+1$. Then $HH^*(\Lambda)$ is the graded
commutative algebra given by

\begin{enumerate}[a)]

    \item \underline{Generators}: $x_0, x_1, \dots x_n, y, z_1, \dots, z_n, \gamma,
    h$ with degrees $deg(x_i)=0$, $deg(y)= 1$, $deg(z_j)= 2$, $deg(\gamma)=
    4$ and $deg(h)=6$.

    \item \underline{Relations}:

    \begin{enumerate}[i)]

        \item $x_i\xi =0$, for each $i=1, \dots , n$ and each
        generator $\xi$.

        \item $x_0^n= y^2= x_0 z_j=0$ $(j=1, \dots, n)$

        \item $z_jz_k= (-1)^{k-j+1}(2j-1)(n-k+1)x_0^{n-1}\gamma$ for
        $1\leq j\leq k\leq n$.

        \item $z_j\gamma= (-1)^{j}(n-j+1)x_0^{n-1}h$ $(j=1, \dots, n)$

        \item $\gamma^2= z_1h$

    \end{enumerate}

\end{enumerate}

\end{teor}

\begin{teor}\label{HH^* with char 2n+1} Let $\Lambda=
P(\mathbb{L}_n)$ and suppose that $Char(K)$ divides $2n+1$. Then
$HH^*(\Lambda)$ is the graded commutative algebra given by

\begin{enumerate}[a)]

    \item \underline{Generators}: $x_0, x_1, \dots, x_n, y, z_1, \dots, z_n,
    t_1, t_2, \dots t_{n-1}, \gamma, h$ with degrees $deg(x_i)=0$,
    $deg(y)=1$, $deg(z_j)=2$, $deg(t_k)=3$, $deg(\gamma)=4$ and
    $deg(h)=6$.

    \item \underline{Relations}:

    \begin{enumerate}[i)]

        \item $x_i\xi=0$ for each $i=1, \dots, n$ and each generator
        $\xi$.

        \item $x_0^n = y^2 = x_0z_j= x_0t_i= yt_i= t_it_k=0$ $(j=1, \dots, n
        \hspace{0.3cm}i,k=1, \dots n-1)$

        \item $z_jz_k=(-1)^{k-j+1}(2j-1)(n-k+1)x_0^{n-1}\gamma$ for
        $1\leq j\leq k\leq n$.

        \item $z_j\gamma= (-1)^j(n-j+1)x_0^{n-1}h$

        \item $\gamma^2= z_1h$

        \item $yz_j=(-1)^{j-1}(2j-1)yz_1$

        \item $z_kt_j=\delta_{jk}x_0^{n-1}y\gamma$ $(k=1, \dots, n \hspace{0.4cm} j=1\dots, n-1)$

        \item $t_j\gamma = \delta_{1j}x_0^{n-1}yh$ $(j=1, \dots,
        n)$.

    \end{enumerate}

\end{enumerate}

\end{teor}

\begin{rem} \rm
We recently learnt about the preprint \cite{Eu}, where the
multiplicative and the Batalin-Vilkovisky structure of $HH^*(\Lambda
)$ is calculated in characteristic zero (actually over the complex
numbers). We do not look at the Gerstenhabber bracket in this paper.
On what concerns the multiplicative structure,  the techniques used
in our paper are  valid for all characteristics $\neq 2$ and detect
a subtle difference of behaviour between the cases when $char(K)$
divides or not $2n+1$, where $n$ is the number of vertices. All
through the paper, we will make frequent comments on the
similarities and disimilarities of our results and those  in
\cite{Eu}.
\end{rem}

The organization of the paper is as follows. In Section (2) we
recall some general facts concerning self-injective algebras which
are needed through paper. Special emphasis is put on the behaviour
of the classical and stable Hochschild homology and cohomology of
these algebras, revisiting some results of Eu and Schedler \cite{ES}
concerning the Calabi-Yau Frobenius algebras.  We also show that the
stable Hochschild cohomology ring is a localization of the classical
one for periodic algebras (see Proposition \ref{localization}). We
introduce the concept of dualizable basis and give conditions for
its existence (Lemma \ref{lema:basis for graded-Schurian algebras}).
In section (3) we introduce the preprojective algebra of type
$\mathbb{L}_n$. In \cite{Eu}, the term 'of type T' instead of 'of
type L' is used and different relations are used to present the
algebra. We prove that the algebra has a dualizable basis and is
hence symmetric.  We then give a concrete cochain complex which
computes the Hochschild cohomology (Proposition \ref{cohomology
complex}. The graded Frobenius condition of
$\underline{HH}^*(\Lambda )$ follows from the symmetric and periodic
condition of $\Lambda$. Excepting this result all other results  in
that section are characteristic-free and will be applied in a
forthcoming paper to tackle the case of characteristic $2$. In
Section (4) we explicitly caculate the dimensions of the Hochschild
cohomology and homology spaces, and also those of the cyclic
homology spaces in characteristic zero (Theorem \ref{dimension of
HH^i} and Corollary \ref{cor:cyclic homology spaces}).
  Actually, by identifying the structure of each $HH^i(\Lambda
)$ as a module over $Z(\Lambda )=HH^0(\Lambda )$, we give a
canonical basis of each $HH^i(\Lambda )$ (cf. Proposition
\ref{K-basis for HH^*Lambda1}). The final section (5) studies the
multiplication in $HH^*(\Lambda )$, giving auxiliary results leading
to  the proof of the two main theorems, from which we also derive a
presentation $\underline{HH}^*(\Lambda )$ by generators an relations
(Corollary \ref{cor:presentation:stable cohomology ring}).

\section{Preliminaries}
We will fix an algebraically closed field $K$ all through the paper.
No condition on its characteristic is required in this section, all
through which $\Lambda$ will be a finite dimensional algebra of the
form $\Lambda =KQ/I$, where $Q$ is a finite quiver, $KQ$ is its path
algebra and $I$ is an admissible ideal, i.e., $I$ is generated by
set of linear combinations of paths of length $\geq 2$, called
relations, and $I$ contains all paths of length $m$, for some $m\geq
2$. We shall denote by $Q_0$ and $Q_1$ the sets of vertices and
arrows of $Q$, respectively, and $i(a)$ and $t(a)$ will denote the
origin and the end of a given $a\in Q_1$. For each $i\in Q_0$, we
will denote by $e_i$ the associated idempotent element of $\Lambda$.

After section \ref{subsection.dualizable basis}  we shall
concentrate on the preprojective algebra of type $\mathbb{L}_n$ but,
for the moment, we need some  preliminaries in this general context.
 All through the paper unadorned tensor products are considered over the field $K$. The word 'module' will mean 'left module' unless
otherwise stated and we shall denote by $_\Lambda
\text{Mod}_\Lambda$ the category of $\Lambda -\Lambda -$bimodules
(abbreviated $\Lambda$-bimodules). If
$\Lambda^e:=\Lambda\otimes\Lambda^{op}$ denotes the enveloping
algebra of $\Lambda$ and $M$ is a $\Lambda$-bimodule, then we will
view $M$ either as a left $\Lambda^e$-module, with multiplication
$(a\otimes b^{o})m= amb$, or as a right $\Lambda^e$-module, with
multiplication $m(a\otimes b^{o})=bma$, for all $a,b\in \Lambda$ and
all $m\in M$. In this way, we identify $_{\Lambda^e}Mod=_\Lambda
Mod_\Lambda =Mod_{\Lambda^e}$. Whenever $M$ and $N$ are
$\Lambda$-bimodules, we shall denote by $Hom_{\Lambda^e}(M,N)$ the
corresponding space of morphisms.

\subsection{Equivalences of categories induced by automorphisms}

 It is well-known that, given any automorphism
$\sigma\in Aut(\Lambda )$, each $\Lambda$-module $M$ admits a
twisted version $_{\sigma}M$, where the underlying vector space is
$M$ and the multiplication by elements of $\Lambda$ is given by
$a\cdot m=\sigma (a)m$, for all $a\in\Lambda$ and $m\in M$. It is
also well-known that the assignment $M\rightsquigarrow _\sigma M$
underlines and equivalence of categories (which acts as the identity
on morphisms) $_\Lambda Mod\stackrel{\cong}{\longrightarrow}_\Lambda
Mod$ with a quasi-inverse taking $M\rightsquigarrow
_{\sigma^{-1}}M$.

Suppose now that  $\sigma, \tau\in Aut(\Lambda)$. Then we get an
automorphism of the enveloping algebra,  $\sigma \otimes \tau^{op}:
\Lambda \otimes \Lambda^{op}\longrightarrow \Lambda \otimes
\Lambda^{op}$, which takes  $(a\otimes b^{o}\rightsquigarrow
\sigma(a)\otimes \tau(b)^{o}$. If $M$ is a $\Lambda$-bimodule, which
we view as a  left $\Lambda^e$-module, the previous paragraph gives
a new left $\Lambda^e$-module $_{\sigma \otimes \tau^{op}}M$. In the
usual way, we interpret it as a $\Lambda$-bimodule $
_{\sigma}M_{\tau}$, and then the multiplications by elements of
$\Lambda$ are given by $a\cdot m\cdot b= \sigma(a)m\tau(b)$. In
particular, the assignment $M\rightsquigarrow _\sigma M_\tau$
underlies an equivalence of categories $_\Lambda
Mod_\Lambda\stackrel{\cong}{\longrightarrow}_\Lambda Mod_\Lambda$.

We are specially interested in the case of the previous paragraph
when $\sigma =1_{\Lambda}$, and denote by  $F_\tau :_\Lambda
\text{Mod}_\Lambda\stackrel{\cong}{\longrightarrow}_\Lambda
\text{Mod}_\Lambda$ the equivalence taking $M\rightsquigarrow
_1M_\tau$.  We will need an alternative description of the
selfequivalence of categories induced by $F_\tau$ on the full
subcategory $_\Lambda Proj_\Lambda$ of $_\Lambda \text{Mod}_\Lambda$
consisting of the projective $\Lambda$-bimodules. We still denote by
$F_\tau:_\Lambda
Proj_\Lambda\stackrel{\cong}{\longrightarrow}_\Lambda Proj_\Lambda$
the mentioned selfequivalence.

\begin{lema} \label{lema.equivalencia en bimodulos proyectivos}
Let $\tau\in Aut(\Lambda )$ be an automorphism which fixes the
vertices and consider the  $K$-linear functor $G_\tau :_\Lambda
Proj_\Lambda\longrightarrow_\Lambda Proj_\Lambda$ identified by the
following data:

\begin{enumerate}
\item $G_\tau (P)=P$, for each projective $\Lambda$-bimodule $P$
\item $G_\tau$ preserves coproducts
\item If $f:\Lambda e_i\otimes e_j\Lambda\longrightarrow\Lambda e_k\otimes
e_l\Lambda$ is a morphism in $_\Lambda Proj_\Lambda$, then $f_\tau
:=G_\tau (f)$ is the only morphism of $\Lambda$-bimodules $f_\tau
:\Lambda e_i\otimes e_j\Lambda\longrightarrow\Lambda e_k\otimes
e_l\Lambda$ taking $e_i\otimes e_j\rightsquigarrow\sum_{1\leq r\leq
m}a_r\otimes\tau^{-1}(b_r)$, provided that $f(e_i\otimes
e_j)=\sum_{1\leq k\leq r}a_r\otimes b_r)$.
\end{enumerate}
Then $G_\tau$ is naturally isomorphic to the selfequivalence $F_\tau
=_1(-)_\tau :_\Lambda Proj_\Lambda\longrightarrow_\Lambda
Proj_\Lambda$.
\end{lema}
\begin{proof}
It is clear that the given conditions determine a unique $K$-linear
functor $G_\tau :_\Lambda Proj_\Lambda\longrightarrow_\Lambda
Proj_\Lambda$ since each projective $\Lambda$-bimodule is a
coproduct of bimodules of the form $\Lambda e_i\otimes e_j\Lambda$.
In order to give the desired natural isomorphism $\psi
:G_\tau\stackrel{\cong}{\longrightarrow}F_\tau$ it will be enough to
define it on the indecomposable  projective $\Lambda$-bimodules
$P=\Lambda e_i\otimes e_j\Lambda$. Indeed, for such a $P$, we define
$\psi_P:G_\tau (P)=P\longrightarrow _1P_\tau =F_\tau (P)$ by the
rule $\psi_P(a\otimes b)=a\otimes\tau (b)$. It is clear that
$\psi_P$ is an isomorphism of $\Lambda$-bimodules. Finally, it is
routinary to see that if $f:P=\Lambda e_i\otimes
e_j\Lambda\longrightarrow Q=\Lambda e_k\otimes e_l\Lambda$ is a
morphism between indecomposable objects of $_\Lambda Proj_\Lambda$,
then

$$F_\tau (f)\circ\psi_P=f\circ\psi_P=\psi_Q\circ f_\tau
=\psi_Q\circ G_\tau (f),$$

which shows that the $\psi_P$ define a natural isomorphism $\psi
:G_\tau\stackrel{\cong}{\longrightarrow}F_\tau $ as desired.
\end{proof}

\subsection{The Yoneda product of extensions}

For the convenience of the reader we recall the definition of
$HH^*(\Lambda)$ and of the Yoneda product. By classical theory of
derived functors, for each pair $M,N$ of $\Lambda$-modules, one can
compute the $K$-vector space $Ext_\Lambda^n(M,N)$ of $n$-extensions
as the $n$-th cohomology space of the complex $Hom_\Lambda
(P^\bullet ,N)$, where

$$P^\bullet :...P^{-n-1}\stackrel{d^{-n-1}}{\longrightarrow}P^{-n}\longrightarrow...\longrightarrow
P^{-1}\stackrel{d^{-1}}{\longrightarrow}P^0\twoheadrightarrow
M\rightarrow 0$$ is a projective resolution of $M$.

Suppose that $L,M,N$ are $\Lambda$-modules, that $P^\bullet$ and
$Q^\bullet$ are projective resolutions of $L$ and $M$, respectively,
and that $m,n$ are natural numbers. If $\delta\in
Ext_\Lambda^n(L,M)$ and $\epsilon\in Ext_\Lambda^m(M,N)$, then we
can choose a  $\tilde{\delta}\in Hom_\Lambda (P^{-n},M)$, belonging
to the kernel of the transpose map $(d^{-n-1})^*:Hom_\Lambda
(P^{-n},M )\longrightarrow Hom_\Lambda (P^{-n-1},M )$ of the
differential $d^{-n-1}:P^{-n-1}\longrightarrow P^{-n}$ of
$P^\bullet$, which represents $\delta$. Similarly we can choose an
$\tilde{\epsilon}\in Hom_\Lambda (Q_m,N)$ which represents
$\epsilon$. Due to the projectivity of the $P^i$, there is a
non-unique sequence of morphisms of $\Lambda$-modules
$\delta^{-k}:P^{-n-k}\longrightarrow Q^{-k}$ ($k=0,1,...,$) making
commute the following diagram

\vspace{0.3cm}

$$\xymatrix{\cdots \ar[r] & P^{-n-k} \ar[r] \ar[d]^{\delta^{-k}} & \cdots \ar[r] &
P^{-n-1}\ar[r] \ar[d]^{\delta^{-1}} & P^{-n}  \ar[r]
\ar[d]^{\delta^0} \ar[dr]^{\widetilde{\delta}}& \cdots \ar[r] & P^0
 \ar[r] & L \ar[r] & 0\\
\cdots \ar[r] & Q^{-k} \ar[r] & \cdots \ar[r] & Q^{-1} \ar[r] & Q^0
\ar[r] & M \ar[r] & 0 & &}
$$

\vspace{0.3cm}

Then the composition
$\tilde{\epsilon}\circ\delta^{-m}:P^{-m-n}\longrightarrow N$ is in
the kernel of $(d^{-m-n})^*$ and, thus, it represents an element of
$Ext_\Lambda^{m+n}(L,N)$. This element is denoted by
$\epsilon\delta$ and does not depend on the choices made. It is
called the \emph{Yoneda product} of $\epsilon$ and $\delta$. It is
well-known that the map

\begin{center}
$Ext_\Lambda^m(M,N)\times Ext_\Lambda^n(L,N)\longrightarrow
Ext_\Lambda^{m+n}(L,N)$ \hspace*{0.5cm} ($(\epsilon ,\delta
)\rightsquigarrow \epsilon\delta$)
\end{center}
is $K$-bilinear.

When $M=N$ in the above setting, the vector space
$Ext_\Lambda^*(M,M)=\oplus_{i\geq 0}Ext_\Lambda ^i(M,M)$ inherites a
structure of graded $K$-algebra, where the multiplication of
homogeneous elements is the Yoneda product. In this paper we are
specifically interested in a particular case of this situation.
Namely, when we replace $\Lambda$ by its enveloping algebra
$\Lambda^e$ and replace $M$ by $\Lambda$, viewed as
$\Lambda^e$-module (i.e. as a $\Lambda$-bimodule). Then
$HH^i(\Lambda ):=Ext_{\Lambda^e}^i(\Lambda ,\Lambda )$ is called the
\emph{i-th Hochschild cohomology space}, for each $i\geq 0$. The
corresponding graded algebra $Ext^*_{\Lambda^e}(\Lambda ,\Lambda )$
is denoted by $HH^*(\Lambda )$ and called the \emph{Hochschild
cohomology ring (or algebra)} of $\Lambda$. By a celebrated result
of Gerstenhaber (\cite{G}), we know that $HH^*(\Lambda )$ is graded
commutative. That is if $\epsilon\in HH^i(\Lambda )$ and $\delta\in
HH^j(\Lambda )$ are homogeneous elements then $\epsilon\delta
=(-1)^{ij}\delta\epsilon$.

\subsection{Some facts about self-injective algebras}
In this paragraph we assume $\Lambda$ to be self-injective. It is
well-known (see section 2 in \cite{BES} and \cite{HZ}) that there is
an automorphism $\eta\in Aut(\Lambda )$, called the \emph{Nakayama
automorphism} and uniquely determined up to inner automorphism, such
that $D(\Lambda)$ is isomorphic to the twisted bimodule
$_1\Lambda_\eta$. The automorphism $\eta$ is also identified by the
fact that the Nakayama functor

$$ DHom_\Lambda (-,\Lambda ):_\Lambda Mod\longrightarrow _\Lambda Mod$$
is naturally isomorphic to the self-equivalence of $_\Lambda Mod$
which takes $M\rightsquigarrow _{\eta^{-1}}M$.

The automorphism $\eta$ can be chosen to permute the vertices of $Q$
and the corresponding permutation $\nu$ of $Q_0$ is called the
\emph{Nakayama permutation}. This permutation is identified by the
fact that $Soc(e_i\Lambda )$ is the simple right module $S_{\nu
(i)}$ associated to the vertex $\nu (i)$, equivalently, by the fact
that $Soc(\Lambda )\cap e_i\Lambda e_{\nu (i)}\neq 0$.

Any isomorphism
$f:_1\Lambda_\eta\stackrel{\cong}{\longrightarrow}D(\Lambda )$
yields a nondegenerate, but not necessarily symmetric, bilinear form
 $(-,-):\Lambda\times\Lambda\longrightarrow K$ defined by
 $(a,b)=f(b)(a)$. Then one gets
 $(ac,b)=(a,cb)$ for all $a,b,c\in\Lambda$. A bilinear
 form $(-,-):\Lambda\times\Lambda\longrightarrow K$ such that
 $(ac,b)=(a,cb)$, for all $a,b,c\in\Lambda$, will be called a \emph{Nakayama  form}.
It always comes from a Nakayama automorphism of $\Lambda$ in the
just described way. To see that, note that
$bf(1)=f(b)=f(1)\eta^{-1}(b)$, and so

$$(a,b)=f(b)(a)=(bf(1))(a)=(f(1)\eta^{-1}(b))(a)=f(1)(\eta^{-1}(b)a)=(af(1))(\eta^{-1}(b))=f(a)(\eta^{-1}(b))=(\eta^{-1}(b),a),$$
for all $a,b\in\Lambda$. This tells us that we recuperate $\eta$
from $(-,-)$ by the rule that $(a,b)=(\eta^{-1}(b),a)$, using the
nondegeneracy condition on $(-,-)$.

 Given any basis $B$ of
 $\Lambda$, one  obtains a (right) dual basis $B^*=\{b^*:$ $b\in B\}$
 identified by the property that $(b,c^*)=\delta_{bc}$, for all $b,c\in
 B$.

 The following result shows that essentially all Nakayama forms can be
 constructed from suitable bases of $\Lambda$ in particular way.

 \begin{prop} \label{prop.Nakayama form from basis}
 Let $\Lambda$ be a selfinjective algebra, let $(-,-):\Lambda\times\Lambda\longrightarrow
 K$ be a bilinear form and consider the following assertions:

 \begin{enumerate}
 \item $(-,-)$ is a Nakayama form
 \item There is a basis $B=\bigcup_{i,j\in Q_0}e_iBe_j$ of $\Lambda$
 containing the vertices and a basis $\{w_i:$ $i\in Q_0\}$ of $\text{Soc}(\Lambda
 )$ such that $(x,y)=\sum_{i\in
 Q_0}\lambda_i$ for all $x,y\in\Lambda$, where $\lambda_i$ is the coefficient of $w_i$ in the expression of
$xy$ as a linear combination of the elements of
  $B$.
 \end{enumerate}
 Then $2)\Longrightarrow 1)$ and, in case $(e_i,e_i)=0$ for all $i\in
 Q_0$, the inverse implication is also true.
 \end{prop}
 \begin{proof}
$2)\Longrightarrow 1)$ Fix a basis $B$ as in assertion 2, where,
without loss of generality, we assume that $w_i=e_iw_ie_{\nu (i)}$,
for all $i\in Q_0$
 , where $\nu$ is the Nakayama permutation. Then we
 clearly have $(xa,y)=(x,ay)$, for all $a,x,y\in\Lambda$. Suppose
 now that $x\in\Lambda$ is an element such that $(x,y)=0$, for all
 $y\in\Lambda$. Then one also has $(e_ix,y)=0$, for all
 $y\in\Lambda$, so that, in order to prove the (left) nondegeneracy
 of $(-,-)$, we can assume that $x=e_ix$ for a (unique) $i\in Q_0$.
 Since $Soc(\Lambda )$ intersects nontrivially any nonzero (left or
 right) ideal, in case $x\neq 0$, we have $0\neq x\Lambda\cap Soc(\Lambda )\subseteq e_iSoc(\Lambda )=e_iSoc(\Lambda )e_{\nu
 (i)}$. Then there is $y\in\Lambda$ such that $xy=\lambda w_i$, with
 $0\neq \lambda\in K$.
 This  implies that $(x,y)=\lambda\neq 0$, which is a contradiction.
 Therefore $(-,-)$ is left nondegenerate and right
 nondegeneracy follows dually.

$1)\Longrightarrow 2)$ We know that the map
$f:\Lambda\longrightarrow D(\Lambda
 )$, given by $f(b)=(-,b):a\rightsquigarrow (a,b)$, is an
 isomorphism of left $\Lambda$-modules. Moreover, we have
 $(a,b)=[af(b)](1)$. If now $b\in Soc(\Lambda )$ and $a\in J(\Lambda
 )$, then $f(b)\in Soc(_\Lambda D(\Lambda ))$ and so $af(b)=0$, and hence $(a,b)=0$. It
 follows that $J(\Lambda )\subseteq ^\perp Soc(\Lambda )$ and a
 comparison of dimensions, using the nondegeneracy of $(-,-)$,  gives that $J(\Lambda )=^\perp Soc(\Lambda
 )$. By a  'symmetric' argument, one gets that  $Soc(\Lambda )^\perp =J(\Lambda
 )$. It follows that $(-,-)$ defines by restriction nondegenerate
 bilinear forms

 \begin{center}
 $(-,-):KQ_0\times Soc(\Lambda )\longrightarrow K$

 $(-,-):Soc(\Lambda )\times KQ_0\longrightarrow K$
 \end{center}
 If we denote by $\{e_i^*:$ $i\in Q_0\}$ and $\{^*e_i :$ $i\in Q_0\}$
 the right and left dual basis of $Q_0\equiv\{e_i:$ $i\in Q_0\}$,
 respectively, with respect to the these nondegenerate forms
 $w_i=e_i$. The equality $(a,b)=(\eta^{-1}(b),a)$, where $\eta$ is
 the Nakayama automorphism, implies that $e^*_i=^*e_{\nu (i)}$, for
 all $i\in Q_0$. Therefore, up to the ordering of its elements, the
 left and the right
 dual bases of $Q_0$ coincide. We put in the sequel $w_i=e_i^*$, for
 each $i\in Q_0$, and $B'=\{w_i:$ $i\in Q_0\}$ is the basis of $Soc(\Lambda
 )$ that we choose. We next put $W:=KQ_0\oplus Soc(\Lambda )$ and
 claim that $\Lambda =W\oplus W^\perp$. Due to the nondegeneracy of $(-,-)$, it  will be enough to
 prove that $W\cap W^\perp =0$. Indeed if $v\in W\cap W^\perp$ and we
 write it as $v=u+w$, with $u\in KQ_0$ and $w\in Soc(\Lambda )$,
 then we have $0=(u+w,w_i)=(u,w_i)$, for all $i\in Q_0$. The
 nondegeneracy of $(-,-):KQ_0\times Soc(\Lambda )\longrightarrow K$
 gives that $u=0$, and hence $u=w$.
 But then $w\in W^\perp\subseteq (KQ_0)^\perp$, which contradicts the
 nondegeneracy of $(-,-):KQ_0\times Soc(\Lambda )\longrightarrow K$,
 and thus settles our claim.

 Our desired basis $B$ will be the union (in this order) of  the basis $Q_0\equiv \{e_i:$ $i\in
 Q_0\}$ of $KQ_0$,  of any basis of $W^\perp$
 contained in $\bigcup_{i,j\in Q_0}e_i\Lambda e_j$ and of the basis $\{w_i:$ $i\in
 Q_0\}$ of $Soc(\Lambda )$.

It only remains to prove that if $x,y\in B$, then $(x,y)$ is as in
the statement. In case $x=e_i$, we have that $(x,y)=(e_i,y)=0$
unless $y=w_i$, and clearly $(x,y)$ is the sum of the coefficients
of the $w_j$ in $xy$. In case $x\not\in Q_0$, then we have
$(x,y)=(e_i,xy)$ for $(-,-)$ is a Nakayama form. But if
$xy=\sum_{b\in B}\lambda_bb$, with $\lambda_b\in K$, then
$(e_i,xy)=\sum_{b\in B}\lambda_b(e_i,b)$ and the result follows from
the case $x=e_i$ already studied.
\end{proof}

\begin{defi} \label{defi:associated Nak form and dualizable} \rm
Let $B$ be a basis of $\Lambda$
 such that $B=\bigcup_{i,j\in Q_0}e_iBe_j$, $B$  contains the vertices  of $Q$ and  contains a basis $\{w_i:$ $i\in Q_0\}$
 of $Soc(\Lambda
 )$, with $w_i\in e_i\Lambda$ for each $i\in Q_0$. The Nakayama
 form $(-,-):\Lambda\times\Lambda\longrightarrow K$ given by the above
 proposition
is called the \emph{Nakayama form associated to $B$}.

The basis $B$ will be called \emph{dualizable} when its associated
Nakayama form is symmetric.
\end{defi}

The following is a useful property:

\begin{lema} \label{Nakayama aut}Let $\Lambda$ be a self-injective algebra with $\eta$ as
 Nakayama automorphism. The following assertions hold:

\begin{enumerate}

    \item $\Lambda^e= \Lambda \otimes \Lambda^{op}$ is a
    self-injective algebra with Nakayama automorphism $\eta\otimes
    (\eta^{-1})^o$.

    \item $M^*:= Hom_{\Lambda^e}(M, \Lambda^e)$ is isomorphic to $_{\eta}D(M)_{\eta^{-1}}$, for each $\Lambda$-bimodule M, and the
    isomorphism is natural on $M$.

    \item $Hom_{\Lambda^e}(\Lambda ,\Lambda^e)\cong _1\Lambda_{\eta^{-1}}$.

\end{enumerate}

\end{lema}

\begin{proof}

1) It is well-known that the tensor product of self-injective
algebras is again self-injective. Moreover, if $A$ and $B$ are
self-injective algebras, then the map

$$<-,->: (A\otimes B)\times (A\otimes B)\longrightarrow K$$ given by
$<a\otimes b, a'\otimes b'>= (a,a')(b,b')$ is a Nakayama form for
$A\otimes B$, from which it easily follows that $\eta_{A}\otimes
\eta_{B}$ is a Nakayama automorphism for $A\otimes B$.

We just need to check now that $(\eta^{-1})^o$ is a Nakayama
automorphism for $\Lambda^{op}$. But note that we have an obvious
Nakayama form $<-,->:\Lambda^{op}\times\Lambda^{op}\longrightarrow
K$ defined by the rule $<x^o,y^o>=(y,x)$, where the second member of
the equality is given by a fixed Nakayama form $(-,-)$ of $\Lambda$.
Then we get:

 $$<\eta(b)^o,
a^o>= (a, \eta(b))=(\eta^{-1}(\eta(b)),a)= (b,a)=<a^o,b^o>$$ from
which it follows that the assignment $b^o\rightsquigarrow
\eta^{-1}(b)^o$ is a Nakayama automorphism for $\Lambda^{op}$.

 We have that $DHom_{\Lambda^e}(-, \Lambda^e)$ is the Nakayama
functor $\Lambda^e-mod\longrightarrow \Lambda^e-mod$. It follows
that we have natural isomorphisms
$DHom_{\Lambda^e}(-,\Lambda^e)\cong
_{[{\eta\otimes(\eta^{-1})^o]}^{-1}}(-)= _{\eta^{-1}\otimes
\eta^o}(-)\cong _{\eta^{-1}}(-)_{\eta}$ of functors $_\Lambda
Mod_\Lambda\longrightarrow _\Lambda Mod_\Lambda$. We then get
$Hom_{\Lambda^e}(M, \Lambda^e)\cong D(_{\eta^{-1}}M_{\eta})\cong
D(_{\eta^{-1}\otimes \eta^0}M)= D(M)_{\eta^{-1}\otimes\eta^0}$ (it
is easy to see that if $A$ is an algebra and $\sigma\in Aut(A)$,
then $D(_{\sigma}N)\cong D(N)_{\sigma}$ for each left $A$-module
$N$).

But the right structure of $\Lambda^e$-module on a
$\Lambda$-bimodule $X$ is given by $x(a\otimes b^0)= bxa$. It
follows that $D(M)_{\eta^{-1}\otimes \eta^0}=
_{\eta}D(M)_{\eta^{-1}}$.

 By assertion $2$, we have $Hom_{\Lambda^e}(\Lambda, \Lambda^e)
\cong _{\eta}D(\Lambda)_{\eta^{-1}}$. But $D(\Lambda)\cong
_1\Lambda_{\eta}$ and then we have
$_{\eta}D(\Lambda)_{\eta^{-1}}\cong _{\eta}\Lambda_1\cong
_1\Lambda_{\eta^{-1}}$.

\end{proof}

We look now at the case when $I$ is a homogeneous ideal of $KQ$
 with respect the length grading of $KQ$. In such case we get an induced grading on $\Lambda$
 in the obvious way. We shall call it the \emph{natural grading}.
 The following lemma gives a handy criterion for a basis to be
 dualizable.

\begin{lema} \label{lema:basis for graded-Schurian algebras}
Let $\Lambda= \frac{KQ}{I}$ be a graded self-injective algebra such
that its Nakayama permutation is the identity and $dim(e_i\Lambda_n
e_j)\leq 1$, for all $i,j\in Q_0$ and all natural numbers $n$. Let
$B$ be a basis of $\Lambda$ consisting of paths and negative paths
and containing the vertices and a basis  of $Soc(\Lambda)$. If
$(-,-):\Lambda\times\Lambda\longrightarrow K$ is the Nakayama form
associated to $B$, then the following assertions are equivalent:

\begin{enumerate}

    \item $a^*a= \omega_{t(a)}$, for all $a\in Q_1$.

    \item $b^{**}=b$, for each $b\in B$.

    \item $(-,-)$ is symmetric, i.e., $B$ is a dualizable basis.

\end{enumerate}

\end{lema}

\begin{proof}

Let $b\in e_iBe_j$ any element of the basis. We claim that there is
a unique $\widetilde{b}\in B$ such that $(b,\widetilde{b})\neq 0$.
Clearly we have $(b,c)=0$ when $c\not\in e_jBe_i$. If $\omega_i\in
e_iBe_i$ is the unique element of $e_iSoc(\Lambda)$ in $B$, then we
also have $(b,c)=0$ whenever $c\in e_jBe_i$ but $deg(b) + deg(c)\neq
deg(\omega_i)$. Therefore $(b,c)\neq 0$ implies that $c\in
e_jB_re_i$ and $r=deg(\omega_i)- deg(b)$. Since $dim(e_j\Lambda_r
e_i)\leq 1$ and $(-,-)$ is non-degenerate our claim is settled by
choosing $\widetilde{b}$ to be the unique element in $e_jBe_i$, with
$r= deg(\omega_i)-deg(b)$.

The last paragraph implies that the assignment
$b\rightsquigarrow\widetilde{b}$ gives an involutive bijction
$B\longrightarrow B$ and that
$b^*=(b,\widetilde{b})^{-1}\widetilde{b}$, for each $b\in B$. In
particular, if $B^{**}$ is the (right) dual basis of $B^*$ then
$b^{**}= \lambda_b b$, for some $\lambda_b \in K^*$ which can be
explicitly calculated. Namely, we have

$$1=(b^*, b^{**})= ((b,\widetilde{b})^{-1}\widetilde{b}, \lambda_b b)=
\lambda_b (b, \widetilde{b})^{-1}(\widetilde{b},b)$$ and so
$\lambda_b=(b,\widetilde{b})(\widetilde{b},b)^{-1}$. It follows that
$b^{**}=b$ if and only if $(b,\widetilde{b})=(\widetilde{b},b)$.
From this the equivalence of assertions $2$ and $3$ is immediate.

Note that our hypotheses guarantee that the nonzero homogeneous
elements of $\Lambda$ are precisely the scalar multiples of the
elements of $B$. We denote by $H$ the set of those nonzero
homogeneous elements. Then an alternative description of $b^*$ is
that it is the unique element of $H$ such that $bb^*=\omega_{i(b)}$.
We can then extend $(-)^*$ to a bijective map $(-)^*:
H\longrightarrow H$, so that $h^*$ is the unique element of $H$ such
that $hh^*= \omega_{i(h)}$. It is then clear that $(\lambda h)^*=
\lambda^{-1}h^*$, for all $h\in H$ and $\lambda\in K^*$.

$1)\Longrightarrow 2)$ Note that if $h_1,h_2 \in H$ are such that
$h_1h_2(h_1h_2)^*= \omega_{i(h_1)}$ implies that $h_2(h_1h_2)^*=
h_1^*$.

We next prove that if $a\in Q_1$ and $h\in H$ are such that $ah\neq
0$, then $(ah)^*a=h^*$. We proceed by induction on $deg(h)$, the
case $deg(h)=0$ being a direct consequence of the hypothesis. Since
$h$ is a scalar multiple of an element of $B$, we can assume without
loss of generality that $h$ is a path in $Q$, say, $h=\alpha_1\cdots
\alpha_r$. Then we have

$$h[(ah)^*a]= \alpha_1\cdots \alpha_r(a\alpha_1\cdots \alpha_r)^*a=
\alpha_1\cdots \alpha_{r-1}[\alpha_r(a\alpha_1\cdots
\alpha_{r-1}\alpha_r)^*]a=$$ $$\alpha_1\cdots
\alpha_{r-1}(a\alpha_1\cdots \alpha_{r-1})^*a$$

By the induction hypothesis, the last term is equal to
$\alpha_1\cdots \alpha_{r-1}(\alpha_1\cdots \alpha_{r-1})^*=
\omega_{i(b)}$. It follows that $(ab)^*a= h^*$.

We finally prove by induction on $deg(h)$ that $h^*h=\omega_{t(h)}$,
for all $h\in H$, which implies that $h^{**}=h$, for all $h\in H$
and hence ends the proof. The cases of $deg(h)=0,1$ are clear. So we
assume that $deg(h)>1$ and, again, assume that $h=\alpha_1\cdots
\alpha_r$ is a path $(r>1)$. Then

$$h^*h= [\alpha_1(\alpha_2\cdots \alpha_r)]^*\alpha_1\alpha_2 \cdots \alpha_r=
(\alpha_2\cdots \alpha_r)^*\alpha_2 \cdots \alpha_r$$ and, by the
induction hypothesis, the last term is equal to
$\omega_{t(\alpha_r)}= \omega_{t(h)}$.

\end{proof}

\subsection{Stable and Absolute Hochschild (Co)Homology of
Self-injective Algebras}

In this section we briefly recall some results from \cite{ES} which
will be nedeed through the paper. All through the section $A$ is a
basic finite dimensional self-injective algebra.

\begin{defi} \label{Complete projective resolution} \rm Let  $M$ be any $A$-module. A \emph{complete projective resolution} of $M$
is an acyclic complex of projective $A$-modules

$$P: \cdots \longrightarrow P^{-2}\stackrel{d^{-2}}{\longrightarrow}P^{-1}
\stackrel{d^{-1}}{\longrightarrow}P^0
\stackrel{d^0}{\longrightarrow} P^1\stackrel{d^1}{\longrightarrow}
P^2\longrightarrow \cdots$$ such that $Z^1:= Ker(d^1)=M$. Such a
complete resolution is called \emph{minimal} in case the induced
morphism $P^n\longrightarrow Z^{n+1}:=Im(d^n)$ is a projective
cover, for each $n\in \mathbb{Z}$.

\end{defi}

\begin{prop}Let  $M$ be an $A$-module. A complete projective resolution of $A$ is
unique, up to isomorphism in the homotopy category $\mathcal{H}A$. A
minimal complete projective resolution is unique, up to isomorphism
in the category $\mathcal{C}A$ of (cochain) complexes of
$A$-modules.

\end{prop}

\begin{proof}The first assertion ia a consequence of a more general
fact (see, e.g. \cite{Kr} Proposition 7.2 and Example 7.16]) which
states that the assignment $P\rightsquigarrow Z^1(P)= Ker(d^1)$
gives an equivalence of triangulated categories
$\mathcal{H}_{ac}(A-Inj)\cong A-\underline{Mod}$, where
$\mathcal{H}_{ac}(A-Inj)$ is the full subcategory of $\mathcal{H}A$
consisting of acyclic complexes of injective (=projective)
$A$-modules.

The final assertion of the proposition is a direct consequence of
the uniqueness of the projective cover up to isomorphism.

\end{proof}
\begin{notat} \rm
 Under the hypothesis of definition \ref{Complete projective
resolution}, let $M, N$ be left $A$-modules and $X$ be a right
$A$-module, and let $P$ be a complete projective resolution of $M$.
For each $i\in\mathbb{Z}$, we put

\begin{enumerate}

    \item $\underline{Ext}_A^i(M,N)= H^i(Hom_A(P,N))$

    \item $\underline{Tor}_i^A(X,M)= H^{-i}(X\otimes_A P)$,
    \end{enumerate}
where $H^i(-)$ denotes the $i$-th homology space of the given
complex.

    We call $\underline{Ext}_A^i(-,-)$ and
    $\underline{Tor}_i^A(-,-)$ the stable $Ext$ and the stable
    $Tor$, respectively. Their definition does not depend on the
    complete resolution $P$ that we choose.
\end{notat}

We clearly have $\underline{Ext}_A^i(M,N)= Ext_{A}^i(M,N)$ and
$\underline{Tor}_i^A(X,M)= Tor_i^A(X,M)$, for all $i>0$. In
particular, we have canonical isomorphisms of graded $K$-vector
spaces

$$Ext_A^*(M,N)= \oplus_{i\geq 0}Ext_A^i(M,N)\stackrel{\lambda_{M,N}}
{\longrightarrow}\oplus_{i\in \mathbb{Z}}\underline{Ext}_A^i(M,N)=:
\underline{Ext}_A^*(M,N)$$ and

$$\underline{Tor}_*^A(X,M)= \oplus_{i\in \mathbb{Z}}\underline{Tor}_i^A(X,M)
\stackrel{\mu_{X,M}}{\longrightarrow}\oplus_{i\geq 0}Tor_i^A(X,M)=
Tor_*^A(X,M)$$ where $Ker(\lambda_{M,N})$ and $Coker(\mu_{M,N})$ are
concentrated in degree 0. Actually, $Ker(\lambda_{M,N})=
\mathcal{P}(M,N)= \{f\in Hom_A(M,N)= Ext_A^0(M,N): \text{ f factors
through a projective A-module}\}$ and $Coker(\mu_{X,N})$ is
isomorphic to the image of the morphism $1_x\otimes j_M: X\otimes M
\longrightarrow X\otimes _AP^1$, where $j: M\longrightarrow P^1$ is
the injective envelope of $M$. The following fact for $Ext$ is
well-known. Finally, note that
$\underline{Ext}_A^i(M,N)\cong\underline{Hom}_A(\Omega_A^i(M),N)$
for all $i\in\mathbb{Z}$. In particular, for $M=N$ we get  a
structure of graded algebra on $\underline{Ext}_A^*(M,M)$ induced
from that of
$\oplus_{i\in\mathbb{Z}}\underline{Hom}_A(\Omega_A^i(M),M)$, which
is defined by the rule

\begin{center}
$g\cdot f=g\circ\Omega_A^i(f)$,
\end{center}
whenever $f\in \underline{Hom}_A(\Omega_A^j(M),M)$ and
$g\in\underline{Hom}_A(\Omega_A^i(M),M)$. In particular, the
multiplication on $\underline{Ext}_A^*(M,M)$ extends the Yoneda
product defined in subsection 2.2. Then the next result follows in a
straightforward way.

\begin{prop} \label{graded module Ext(M,N)}

Let $M,N$ be left $A$-modules. The space $\underline{Ext}_A^*(M,M)$
has a canonical structure of graded algebra over which
$\underline{Ext}_A^*(M,N)$ is a graded right module. Moreover the
map

$$\lambda_{M,M}: Ext_A^*(M,M)\longrightarrow
\underline{Ext}_A^*(M,M)$$ is a homomorphism of graded algebras and
the following diagram is commutative, where the horizontal arrows
are the multiplication maps:

$$\xymatrix{Ext_A^*(M,N)\otimes Ext_A^*(M,M)\ar[r]^{\vspace{0.7 cm}\hspace{1.2 cm}{Yoneda}} \ar[d]^{\lambda_{M,N}\otimes \lambda_{M,M}}& Ext_A^*(M,N) \ar[d]^{\lambda_{M,N}} \\
\underline{Ext}_A^*(M,N)\otimes \underline{Ext}_A^*(M,M) \ar[r] &
\underline{Ext}_A^*(M,N)}
$$

\end{prop}

Consider now $\underline{Tor}_*^A(X,M)$ as a graded $K$-vector
space, but taking as $n$-homogeneous component precisely
$\underline{Tor}_{-n}^A(X,M)$. For that reason we shall write
$\underline{Tor}_{- *}^A(X,M)$. If now $P$ is a fixed minimal
complete projective resolution of $M$, then $P$ is canonically
differential graded (dg) $A$-module, i.e., an object in
$\mathcal{C}_{dg}A$ using the terminology of \cite{K2}. Then $B=
End_{\mathcal{C}_{dg}A}(P)$ is a dg algebra which acts on
$X\otimes_A P$ in the obvious way, making $X\otimes_AP$ into a dg
left $B$-module. As a consequence, $H^*(X\otimes_A P)=
Tor_{-*}^A(X,M)$ is a graded left module over the cohomology algebra
$H^*B= \oplus_{n\in \mathbb{Z}}Hom_{\mathcal{H}A}(P, P[n])$, where
$\mathcal{H}A$ is the homotopy category of $A$ (see \cite{K2}). But,
due to the fact that the $1$-cocycle functor gives an equivalence of
triangulated categories $Z^1:
\mathcal{H}_{ac}(Inj-A)\stackrel{\sim}{\longrightarrow}
A-\underline{Mod}$, we have canonical isomorphisms of graded
algebras $H^*B\cong \oplus_{n\in \mathbb{Z}}\underline{Hom}_A(M,
\Omega_A^{-n}M)\cong \underline{Ext}_A^*(M,M)$. It then follows the
desired structure of $Tor_{-*}^A(X,M)$ as a graded left
$\underline{Ext}_A^*(M,M)$-module.

If we take now the nonnegatively graded subalgebra
$\underline{Ext}_A^{\geq 0}(M,M):=\oplus_{n\geq
0}\underline{Ext}_A^n(M,M)$ of $\underline{Ext}_A^*(M,M)$, then
$\underline{Tor}_{->0}(X,M)= \oplus_{j<0}\underline{Tor}_j(X,M)$ is
a graded $\underline{Ext}^{\geq 0}(M,M)$-submodule of
$\underline{Tor}_{-*}^A(X,M)$ and the quotient
$\frac{\underline{Tor}_{-*}^A(X,M)}{\underline{Tor}_{-(>0)}^A(X,M)}$,
which is isomorphic to $\underline{Tor}_{-(\leq 0)}^A(X,M)$ as a
graded $K$-vector space, is a graded left $\underline{Ext}_A^{\geq
0}(M,M)$-module. That is, $\oplus_{i\geq 0}\underline{Tor}_i^A(X,M)$
has a canonical structure of graded left $\underline{Ext}_A^{\geq
0}(M,M)$-module, where $\underline{Tor}_i^A(X,M)$ is the component
of degree $-i$, for all $i\geq 0$. Since we have a surjective
morphism of graded algebras $Ext_A^*(M,M)\twoheadrightarrow
\underline{Ext}_A^{\geq 0}(M,M)$ we get a structure of graded left
$Ext_A^*(M,M)$-module on $\underline{Tor}_{-(\leq 0)}^A(X,M)$ given
by restriction of scalars.

We can now provide $Tor_{-*}^A(X,M)$ (that is just $Tor_*^A(X,M)$,
but with $Tor_i^A(X,M)$ in degree $-i$, for all $i\geq 0$) with a
structure of graded left $Ext_A^*(M,M)$-module of which
$\underline{Tor}_{-(\leq 0)}^A(X,M)$ is a graded submodule. Indeed,
the product $Ext_A^i(M,M)\cdot Tor_{-(-j)}^A(X,M)= Ext_A^i(M,M)\cdot
\underline{Tor}_{-(-j)}^A(X,M)$ is clear when $j\neq 0$. For $j=0$
we put

\begin{center}
$Ext_A^i(M,M)\cdot Tor_0^A(X,M)=0$ if $i>0$,
\end{center}

and for $i=0$ the multiplication is identified by the following
diagram,

$$\xymatrix{Ext_A^0(M,M)\times Tor_0^A(X,M)\ar[r]
\ar[d]^{\shortparallel} & Tor_0^A(X,M) \ar[d]^{\shortparallel}
\\ End_A(M)\times (X\otimes_{A}M) \ar[r] & X\otimes_{A}M},$$
where the bottom horizontal arrow is the canonical map $(f,x\otimes
m)\rightsquigarrow x\otimes f(m)$.

These comments prove the following correspondent of proposition
\ref{graded module Ext(M,N)} for $Tor$.

\begin{prop} \label{graded module Tor(X,M)}
Let $X$ and $M$ be a right and a left $A$-modules, respectively.
Then $\underline{Tor}_{-*}^A(X,M)$ (resp. $Tor_{-*}^A((X,M)$) has a
canonical structure of graded left $\underline{Ext}_A^*(M,M)$-
(resp. $Ext_A^*(M,M)$-)module. Moreover, the following diagram is
commutative

\begin{picture}(340,230)(-50,-180)

\put(-60,10){\bf{$Ext_A^*(M,M)\times \underline{Tor}_{-*}^A(X,M)$}}

\put(120,10){\bf{$\underline{Ext}_A^*(M,M)\times
\underline{Tor}_{-^*}^A(X,M)$}}

\put(299,10){\bf{$\underline{Tor}_{-*}^A(X,M) $}}

\put(-60,-50){\bf{$Ext_A^*(M,M)\times Tor_{-*}^A(X,M)$}}

\put(292,-50){\bf{$Tor_{-*}^A(X,M)$}}

\put(73,16){\bf{\small $\lambda_{M,M}\times 1$}} \put(258,16){\small
{mult.}} \put(5,-15){\small ${1\times \mu_{X,M}}$}
\put(332,-15){\small ${\mu_{X,M}}$} \put(150,-57){\small{ {mult.}}}

\put(72,13){\vector(1,0){45}} \put(250,13){\vector(1,0){45}}
\put(327,2){\vector(0,-1){35}} \put(2,2){\vector(0,-1){35}}
\put(72,-47){\vector(1,0){215}}
\end{picture}

\end{prop}

We are specially interested in the particular case of the two
previous propositions  in which $A= \Lambda^e= \Lambda \otimes
\Lambda^{op}$, for a self-injective algebra $\Lambda$ and $M=
\Lambda$ viewed as $\Lambda^e$-module. In that case, we put
$\underline{HH}^n(\Lambda, N)=
\underline{Ext}_{\Lambda^e}^n(\Lambda,N)$ and
$\underline{HH}_n(\Lambda, N)=
\underline{Tor}_n^{\Lambda^e}(\Lambda, N)$ and call then the $n$-th
stable Hochschild cohomology and $n$-th stable Hochschild homology
space of $\Lambda$ with coefficients in $N$. Putting
$\underline{HH}^*(\Lambda,N)= \oplus_{n\in
\mathbb{Z}}\underline{HH}^n(\Lambda,N)$, $\underline{HH}^*(\Lambda)=
\underline{HH}^*(\Lambda, \Lambda)$, $\underline{HH}_*(\Lambda,N)=
\oplus_{n\in \mathbb{Z}}\underline{HH}_n(\Lambda,N)$ and
$\underline{HH}_*(\Lambda)= \underline{HH}_*(\Lambda, \Lambda)$, we
have the following straightforward consequence of Propositions
\ref{graded module Ext(M,N)} and \ref{graded module Tor(X,M)}. The
graded commutativity of $HH^*(\Lambda )$ was showed by Gerstenhaber
(\cite{G}) and that of  $\underline{HH}^*(\Lambda )$ can be found in
\cite{ES}.

\begin{cor}In the situation above, $\underline{HH}^*(\Lambda)$ (resp.
$HH^*(\Lambda)$) has a canonical structure of graded commutative
algebra over which $\underline{HH}^*(\Lambda, N)$ (resp.
$HH^*(\Lambda,N)$) is a graded right module and
$\underline{HH}_{-*}(\Lambda, N)$ (resp. $HH_{-*}(\Lambda,N)$) is a
graded left module. Moreover, the graded algebra structure on
$HH^*(\Lambda)$ and the graded module structures on
$HH^*(\Lambda,N)$ and $HH_{-*}(\Lambda,N)$ are determined by its
stable correspondent, except in degree zero.
\end{cor}

\begin{rem} \rm
If $B=\oplus_{i\in\mathbb{Z}}B_i$ is a graded commutative algebra,
then any graded left $B$-module $V=\oplus_{i\in\mathbb{Z}}V_i$ may
be viewed as a graded right $B$-module by defining
$vb=(-1)^{deg(b)deg(v)}bv$, for all homogeneous elements $b\in B$
and $v\in V$.  In particular, we shall view in this way
$HH_{-*}(\Lambda,M)$ as graded right
$\underline{HH}^*(\Lambda)$-module, for each $\Lambda$-bimodule $M$.
We proceed similarly with $\underline{HH}_{-*}(\Lambda ,M)$ over
$\underline{HH}^*(\Lambda )$.
\end{rem}

Note that if $Q$ and $M$ are a projective and an arbitrary
$\Lambda$-bimodule, then $D(Q\otimes_{\Lambda^e}M)\cong
Hom_{\Lambda^e}(Q, D(M))$ by adjunction. If now $P$ is the complete
minimal  projective resolution of $\Lambda$ as a bimodule
(equivalently, as a right $\Lambda^e$-module), we have an
isomorphism of complexes $D(P\otimes_{\Lambda^e}M)\cong
Hom_{\Lambda^ e}(P,D(M))=Hom_{\mathcal{C}_{dg}\Lambda^e}(P,D(M))$,
convening that $D(T)^i=D(T^{-i})$ for each complex (or each graded
vector space) $T$ and each $i\in\mathbb{Z}$. It is routinary to see
that the last isomorphism preserves the structures of right dg
modules over the dg algebra
$B:=End_{\mathcal{C}_{dg}\Lambda^e}(P,P)$. It then follows easily:

\begin{rem} \label{Stable duality} \rm If $\Lambda$ is a finite dimensional self-injective
algebra, then

\begin{enumerate}

    \item $\underline{HH}_{-*}(\Lambda,M)\cong D(\underline{HH}^*(\Lambda,
    D(M)))$ as graded $\underline{HH}^*(\Lambda)$-modules.

    \item $HH_{-*}(\Lambda, M)\cong D(HH^*(\Lambda,D(M)))$ as graded
    $HH^*(\Lambda)$-modules.

\end{enumerate}

\end{rem}

\subsection{The Calabi-Yau Property}

If $\Lambda$ is a self-injective algebra, then $_\Lambda mod$ is a
Frobenius abelian category (i.e., it has enough projectives and
injectives and the projective objects coincide with the injective
ones). As a consequence, $\Lambda-\underline{mod}$ is a triangulated
category with $\Omega_{\Lambda}^{-1}:
_A\underline{mod}\longrightarrow _\Lambda\underline{mod}$ as its
suspension functor (cf. \cite{H}, Chapter 1).

Recall (see \cite{K3} and \cite{Ko}) that a $Hom$ finite
triangulated $K$-category $\mathcal{T}$ with suspension functor
$\sum : \mathcal{T}\longrightarrow \mathcal{T}$ is called Calabi-Yau
when there is a natural number $n$ such $\sum^n$ is a Serre functor
(i.e. $DHom_{\mathcal{T}}(X,-)$) and $Hom_{\mathcal{T}}(-, \sum^nX)$
are naturally isomorphic as cohomological functors
$\mathcal{T}^{op}\longrightarrow K-mod$). In such a case, the
smallest natural number $m$ such that $\sum^m$ is a Serre functor is
called the Calabi-Yau dimension \emph{CY-dimension} for short) of
$\mathcal{T}$.

In case $\Lambda$ is a self-injective algebra, Auslander formula
(see \cite{ARS}, Chapter IV, Section 4) says that one has a natural
isomorphism $D\underline{Hom}_\Lambda (X,-)\cong Ext_\Lambda
^1(-,\tau X)$, where $\tau: _\Lambda\underline{mod}\longrightarrow
_\Lambda\underline{mod}$ is the Auslander-Reiten (AR) translation.
Moreover $\tau= \Omega^2\mathcal{N}$, where $\mathcal{N}= D
Hom_{\Lambda}(-,\Lambda)\cong D(\Lambda)\otimes _{\Lambda}-:
_{\Lambda}mod \longrightarrow _{\Lambda}mod$ is the Nakayama functor
(see \cite{ARS}). As a consequence, as shown in \cite{ESk},
$_A\underline{mod}$ is $m$-CY if and only if $m$ is the smallest
natural number such that $\Omega_\Lambda^{-m-1}\cong
\mathcal{N}\cong _{\eta^{-1}}(-)$ as triangulated functors
$_{\Lambda}\underline{mod}\longrightarrow
_{\Lambda}\underline{mod}$, where $\eta$ is the Nakayama
automorphism.

In \cite{ES} an algebra is called \emph{Calabi-Yau Frobenius} of
dimension $m$ if it is self-injective and $m$ is the smallest
natural number such that $\Omega_{\Lambda^e}^{m+1}(\Lambda)\cong
Hom_{\Lambda^e}(\Lambda, \Lambda^e)$. Due to lemma \ref{Nakayama
aut}, we know that this $m$ is the smallest natural number such that
$\Omega_{\Lambda^e}^{m+1}(\Lambda)\cong _1\Lambda_{\eta^{-1}}$ in
$\Lambda^e-\underline{mod}$.

\begin{cor}If $\Lambda$ is a Calabi-Yau Frobenius algebra of
dimension $m$, then $\Lambda-\underline{mod}$ is $d$-Calabi-Yau, for
some natural number $d\leq m$. In case $\Lambda$ is symmetric, $d+1$
is a divisor of $m+1$.

\end{cor}

\begin{proof} For each $k\geq 0$ the functors

$$\Omega_{\Lambda^e}^k(\Lambda)\otimes_{\Lambda}-: \Lambda-\underline{mod}
\longrightarrow \Lambda-\underline{mod}$$ and

$$\Omega_{\Lambda}^k: \Lambda-\underline{mod}\longrightarrow
\Lambda-\underline{mod}$$ are naturally isomorphic as triangulated
functors. If $\Omega_{\Lambda^e}^{m+1}(\Lambda)\cong _1
\Lambda_{\eta^{-1}}$, then $\Omega_{\Lambda}^{m+1}\cong
_1\Lambda_{\eta^{-1}}\otimes_{\Lambda}-\cong _{\eta}(-)$ as
triangulated functors $\Lambda-\underline{mod}\longrightarrow
\Lambda-\underline{mod}$. By taking quasi-inverse functors, we then
have $\Omega_{\Lambda}^{-m-1}\cong _{\eta^{-1}}(-)$ and, by
\cite{ESk}, $\Lambda-\underline{mod}$ is $d$-CY for some natural
number $d\leq m$.

In case $\Lambda$ is symmetric (i.e. $\eta$ is inner), the CY
dimension of $\Lambda-\underline{mod}$ is $d$ when $d+1$ is the
order of $\Omega_{\Lambda}$ in the stable Picard group of $\Lambda$,
i.e., in the group of  natural isoclasses of triangulated
self-equivalences $\Lambda-\underline{mod}\longrightarrow
\Lambda-\underline{mod}$. The fact that
$\Omega_{\Lambda}^{-m-1}\cong _{\eta^{-1}}(-)\cong
id_{\Lambda-\underline{mod}}$ implies that $d+1$ divides $m+1$.

\end{proof}

\begin{rem} \rm To the best of our knowledge, it is not known whether the
Frobenius CY dimension of $\Lambda$ coincides with the CY dimension
of $\Lambda-\underline{mod}$.
\end{rem}

Recall that if $B=\oplus_{n\in \mathbb{Z}}B_n$ is a graded algebra,
then a graded (left or right) $B$-module $V= \oplus_{n\in
\mathbb{Z}}V_n$ is said to be locally finite dimensional graded left
(resp. right) $B$-modules is denoted by $B$-lfgr (lfgr-$B$). We have
the canonical duality $D: B-lfgr\longrightarrow lefgr B$, which is
inverse of itself. Slightly diverting from the terminology of
\cite{ES}, we say that the graded algebra $B$ is graded Frobenius in
case the category $B$-lfgr is a Frobenius category, which is
equivalent to say that the injective and the projective objects
coincide in $B$-lfgr. Clearly, a graded Frobenius algebra in the
sense of \cite{ES} is graded Frobenius in our sense.

\begin{teor}\label{Eu-Schedler}(Eu-Schedler) Let $\Lambda$ be a Calabi-Yau Frobenius
algebra of dimension $m$ and let $M$ be any $\Lambda$-bimodule.
There are isomorphisms of graded right
$\underline{HH}^*(\Lambda)$-modules:

\begin{enumerate}

    \item $\underline{HH}_{-*}(\Lambda, M)[-m]\cong \underline{HH}^*(\Lambda,M)$

    \item $\underline{HH}^*(\Lambda,M)\cong D(\underline{HH}^*(\Lambda,D(M)))[-m]=
    D(\underline{HH}^*(\Lambda,D(M))[m])$

    \item $\underline{HH}^*(\Lambda)\cong D(\underline{HH}^*(\Lambda))[-2m-1]=
    D(\underline{HH}^*(\Lambda)[2m+1])$

\end{enumerate}

In particular $\underline{HH}^*(\Lambda)$ is a graded Frobenius
algebra.

\end{teor}

\begin{proof} Fix a complete minimal projective resolution $P$ of
$\Lambda$ as $\Lambda$-bimodule. We put

$$P^*: \cdots \longrightarrow Hom_{\Lambda^e}(P^1,\Lambda^e)\longrightarrow Hom_{\Lambda^e}
(P^0,\Lambda^e)\longrightarrow
Hom_{\Lambda^e}(P^{-1},\Lambda^e)\longrightarrow \cdots$$ and

$$D(P): \cdots \longrightarrow D(P^1)\longrightarrow D(P^0)\longrightarrow D(P^{-1})
\longrightarrow \cdots$$

\begin{enumerate}

    \item $P[-m-1]$ is a minimal complete projective resolution of
    $\Omega_{\Lambda^e}^{m+1}(\Lambda)$ and we have $Z^0(P^*)=Hom_{\Lambda^e}(\Lambda, \Lambda^e)\cong _1\Lambda_{\eta^{-1}}\cong
    \Omega_{\Lambda^e}^{m+1}(\Lambda)$ (see Lemma \ref{Nakayama aut}) due to the self-injective
    condition of $\Lambda^e$. Therefore $P^*[-1]$ is also a minimal
    complete projective resolution of
    $\Omega_{\Lambda^e}^{m+1}(\Lambda)$. The uniqueness of the
    minimal complete projective resolution gives that $P[-m-1]\cong
    P^*[-1]$, hence $P\cong P^*[m]$ in the category $\mathcal{C}(\Lambda^e)$ of  complexes of $\Lambda$-bimodules. That gives
    isomorphisms of complexes of $K$-vector spaces $Hom_{\Lambda^e}(P,M)\cong Hom_{\Lambda^e}(P^*[m],M)
    \cong Hom_{\Lambda^e}(P^*,M)[-m]\cong
    P\otimes_{\Lambda^e}M[-m]$, the last one due to the fact that
    there is a natural isomorphism

    $$Hom_{\Lambda^e}(Hom_{\Lambda^e}(Q,\Lambda^e),M)\cong
    Q\otimes_{\Lambda^e}M,$$ for each finitely generated projective
    bimodule $Q$.

    Note that the above isomorphisms of complexes are really
    isomorphisms of left dg modules over the dg algebra
    $B:=End_{\mathcal{C}_{dg}(\Lambda^e)}(P)$, whose homology algebra $H^*(B)$ is isomorphic to $\underline{HH}^*(\Lambda )$ (see the
paragraph following Proposition \ref{graded module Ext(M,N)}). Just
as  a sample, we do the
    last isomorphism. We view $P$ as a complex of left $\Lambda^e$-module  The complex  $P^*=Hom_{\Lambda^e}(P,\Lambda^e)$
    is a dg right $B$-module with  multiplication given by  $f\beta
    :=f\circ\beta$, for all homogeneous elements $f\in P^*$ and $\beta\in B$ .
    It has also a structure  of right $\Lambda^e$-module given by $f(a\otimes b^o):p\rightsquigarrow f(p)(a\otimes
    b^o)$, for all homogeneous elements $p\in P$, $f\in P^*$ and all
    $a,b\in\Lambda$. One readily sees that we have an equality $(f\beta )(a\otimes b^o)=[f(a\otimes
    b^o)]\beta$. By looking now at a right $\Lambda^e$-module as a
    left $\Lambda^e$-module in the usual way, this means that $P^*$
    has a structure of dg $\Lambda^e-B-$bimodule. As a consequence,
    when the  $\Lambda$-bimodule $M$ is viewed as a left
    $\Lambda^e$-module, the
    the complex of $K$-vector spaces
    $Hom_{\Lambda^e}(P^*,M)$ is a left dg $B$-module, with
    multiplication given by the rule $(\beta\psi )(f)=\psi (f\beta )=\psi (f\circ\beta
    )$, for all homogeneous elements $\psi\in
    Hom_{\Lambda^e}(P^*,M)$, $\beta\in B$ and $f\in P^*$. It is
    routinary to see that the canonical isomorphism of complex of $K$-vector spaces

    $$\Psi:P\otimes_{\Lambda^e}M\stackrel{\cong}{\longrightarrow}Hom_{\Lambda^e}(P^*,M),$$
identified by the formula $\Psi_{p\otimes m}(f)= f(p)m$,
    preserves the left multiplication by elements of $B$.

    We then get isomorphisms of graded left $\underline{HH}^*(\Lambda
    )$-modules:

    $$\underline{HH}^*(\Lambda,M)= H^*(Hom_{\Lambda^e}(P,M))\cong
    H^*((P\otimes_{\Lambda^e}M)[-m])=H^{*-m}(P\otimes_{\Lambda^e}M)\cong$$

    $$\underline{HH}_{-*+m}(\Lambda,M)= \underline{HH}_{-(*-m)}(\Lambda,M)=
    \underline{HH}_{-*}(\Lambda,M)[-m]$$

    \item It follows from $1)$ and from the isomorphism of graded
    $\underline{HH}^*(\Lambda)$-module $D(\underline{HH}^*(\Lambda, D(M)))\cong
    \underline{HH}_{-*}(\Lambda,M)$ (see remark \ref{Stable duality}).

    \item Since $\Lambda$ is $m$-Calabi-Yau Frobenius we have $\Omega^{m+1}_{\Lambda^e}(\Lambda )\cong
    Hom_{\Lambda^e}(\Lambda, \Lambda^e)\cong _1\Lambda_{\eta^{-1}}$,
    from which we get $\Omega^{-m-1}_{\Lambda^e}(\Lambda)\cong _1\Lambda_{\eta}\cong
    D(\Lambda)$. The isomorphism of $2$ for $M=\Lambda$ gives then

    $$\underline{HH}^*(\Lambda)\cong D(\underline{HH}^*(\Lambda, \Omega_{\Lambda^e}^{-m-1}(\Lambda ))[m])=
    D(\underline{HH}^{*+m+1}(\Lambda)[m])= D(\underline{HH}^*(\Lambda)[2m+1])$$ The fact
    that $\underline{HH}^*(\Lambda)$ is graded Frobenius is a direct consequence
    of the isomorphism in $3)$.

\end{enumerate}

\end{proof}

\begin{defi} \rm A finite dimensional algebra $\Lambda$ is said to be
\emph{periodic of period $m>0$} if
$\Omega_{\Lambda^e}^m(\Lambda)\cong \Lambda$ in the category of
$\Lambda$-bimodules and $m$ is minimal with that property.

\end{defi}

It is well-known that any periodic algebra is self-injective (see
\cite{BBK} ). In case $R$ is a graded commutative ring and $f\in R$
is a homogeneous element which is not nilpotent, we will denote by
$R_{(f)}$ the localization of $R$ with respect to the multiplicative
subset $\{1,f,f^2,\dots\}$. It is a graded commutative ring where
$deg(\frac{g}{f^n})= deg(g) - n\cdot deg(f)$, for all homogeneous
elements $g\in R$ and all $n\geq 0$. If $M$ is a graded $R$-module
we will denote by $M_{(f)}$ the localization of $M$ at
$\{1,f,f^2,\dots\}$.

\begin{prop}\label{localization} Let $\Lambda$ be a periodic algebra of period $s$ and
let $h\in HH^s(\Lambda)$ be any element represented by an
isomorphism
$\Omega_{\Lambda^e}^s(\Lambda)\stackrel{\sim}{\longrightarrow}\Lambda$.
Suppose that $M$ is a $\Lambda$-bimodule. The following assertions
hold:

\begin{enumerate}

    \item $\underline{HH}^*(\Lambda,M)\cong
    \underline{HH}^*(\Lambda,M)[s]$ and $\underline{HH}_{-*}(\Lambda,M)\cong
    \underline{HH}_{-*}(\Lambda,M)[s]$ as graded
    $\underline{HH}^*(\Lambda)$-modules.

    \item $h$ is an element of $HH^*(\Lambda)$ which is not nilpotent and
    $\underline{HH}^*(\Lambda)$ is isomorphic, as a graded algebra, to  $HH^*(\Lambda)_{(h)}$.

    \item  $\underline{HH}^*(\Lambda,M)$ is isomorphic to $HH^*(\Lambda ,M)_{(h)}$ as a graded $\underline{HH}^*(\Lambda )$-module.

\end{enumerate}

\end{prop}

\begin{proof} We have already seen in the previous comments that
$\underline{HH}^*(\Lambda)$ is isomorphic to the graded algebra
$\oplus_{n\in
\mathbb{Z}}\underline{Hom}_{\Lambda^e}(\Omega_{\Lambda^e}^n(\Lambda,
\Lambda))$, where the multiplication of homogeneous elements on this
algebra is given by $g\cdot f= g\circ \Omega_{\Lambda^e}^n(f)$. If
now $\widehat{h}:
\Omega_{\Lambda^e}^s(\Lambda)\stackrel{\sim}{\longrightarrow}\Lambda$
is an isomorphism representing $h$, then
$\Omega_{\Lambda^e}^{-s}(\widehat{h}^{-1}):
\Omega_{\Lambda^e}^{-s}(\Lambda)\longrightarrow \Lambda$ represents
an element $h'\in \underline{HH}^{-s}(\Lambda)$. But then $h'\cdot
h=1$ since $h'\cdot h$ is represented by
$\Omega_{\Lambda^e}^{-s}(\widehat{h}^{-1})\circ
\Omega_{\Lambda^e}^{-s}(\widehat{h})=
\Omega_{\Lambda^e}^{-s}(\widehat{h}^{-1}\widehat{h})=
\Omega_{\Lambda^e}^{-s}(1_{\Omega_{\Lambda^e}^s(\Lambda)})=
1_{\Lambda}$.

The above paragraph shows that $h$ is invertible (of degree $s$) in
$\underline{HH}^*(\Lambda)$, from which it follows that
multiplication by $h$ gives an isomorphism
$Y\stackrel{\sim}{\longrightarrow}Y[s]$, for each graded
$\underline{HH}^*(\Lambda)$-module $Y$ (here we have used that, for
$Char(K)\neq2$, the period $s$ is even, cf. \cite{ES}  Theorem
2.3.47]).

Since the multiplication of homogeneous elements of degree $>0$ is
the same in $HH^*(\Lambda )$ and in $\underline{HH}^*(\Lambda )$ and
$h$ in invertible in this latter algebra it follows that $h$ is not
nilpotent in $HH^*(\Lambda)$. On the other hand, the universal
property of the module of quotients gives a unique morphism of
graded $HH^*(\Lambda )$-modules

\begin{center}
$\Phi:HH^*(\Lambda , M)_{(h)}\longrightarrow\underline{HH}^*(\Lambda
,M)$
\end{center}
which takes the fraction $\frac{\eta}{h^n}\rightsquigarrow
h'^n\eta$, where $h'$ is the inverse of $h$ in
$\underline{HH}^*(\Lambda )$. It is clear that the homogeneous
elements of degree $\geq 0$ are in the image of $\Phi$. On the other
hand, if $\xi\in\underline{HH}^{-j}(\Lambda )$, with $j>0$, then
there is a $k> 0$ such that $ks>j$. Fixing such a $k$, we have that
$\eta:=h^k\xi\in\underline{HH}^{ks-j}(\Lambda ,M)=HH^{ks-j}(\Lambda
,M)$ and, clearly, the equality $\Phi (\frac{\eta}{h^k})=\xi$ holds.
Therefore $\Phi$ is surjective. Moreover $Ker(\Phi )$ consists of
those fractions $\frac{\eta}{h^n}$ such that $h'^n\eta =0$ in
$\underline{HH}^*(\Lambda ,M)$. This is in turn equivalent to say
that $\eta=0$ in $\underline{HH}^*(\Lambda ,M)$ for $h'$ is
invertible in $\underline{HH}^*(\Lambda )$. That is,  $\eta$ is in
the kernel of the canonical map $\lambda_{\Lambda ,M}:HH^*(\Lambda
,M)\longrightarrow\underline{HH}^*(\Lambda ,M)$. Hence we get that
$\eta\in\mathcal{P}(\Lambda ,M)$, which implies that $h\eta =0$ in
$HH^*(\Lambda ,M)$. It follows that
$\frac{\eta}{h^n}=\frac{h\eta}{h^{n+1}}=0$ and so $\Phi$ is also
injective. Finally, in case $\Lambda =M$, the map $\Phi$ is a
homomorphism of graded algebras, and the proof is complete.

\end{proof}

Note that if $\Lambda$ is symmetric, then $\Lambda$ is periodic of
period $s$ exactly when it is $(s-1)$-Calabi-Yau Frobenius. We then
have

\begin{cor} If $\Lambda$ is a symmetric periodic algebra of period
$s$ and $M$ is a $\Lambda$-bimodule, then:

\begin{enumerate}

    \item The multiplicative structure of
    $\underline{HH}^*(\Lambda)$ is determined by that of
    $HH^*(\Lambda)$.

    \item The structures of $\underline{HH}^*(\Lambda,M)$ and $\underline{HH}_{-*}(\Lambda,M)$
    as graded $\underline{HH}^*(\Lambda)$-modules and the structure
    of $HH_{-*}(\Lambda,M)$ as graded $HH^*(\Lambda)$-module are
    determined by the structure of $HH^*(\Lambda,M)$ as graded
    $HH^*(\Lambda)$-module.

\end{enumerate}

\end{cor}

\begin{proof} Since $\Lambda$ is CY Frobenius, the two assertions
are a direct consequence of the last two propositions.

\end{proof}

\section{The algebra $\Lambda =P(\mathbb{L}_n)$}

\subsection{A dualizable basis for the algebra}
\label{subsection.dualizable basis}

 In the rest of the paper, unless
otherwise stated, $\Lambda:=P(\mathbb{L}_n)$ is the preprojective
algebra associated to the generalized Dynkin quiver $\mathbb{L}_n$.
Then its quiver $Q$ is

\vspace{0.5cm}

\def\VSep{6pt} 

\begin{center}
$$\epsilon \psarc{-}(1,0.1){0.5}{35}{330}\kern1.5cm 1\quad
\begin{array}{ccc}
a_1\\
\kern-1cm\psline[linecolor=black]{->}(1.3,0)\psline[linecolor=black]{<-}(0,0.2)(1.3,0.2)\\
\overline{a}_1
\end{array}\qquad 2\quad\begin{array}{ccc}
a_2\\
\kern-1cm\psline[linecolor=black]{->}(1.3,0)\psline[linecolor=black]{<-}(0,0.2)(1.3,0.2)\\
\overline{a}_2
\end{array}\qquad 3\quad\ldots\quad\begin{array}{ccc}
a_{n-1}\\
\kern-1cm\psline[linecolor=black]{->}(1.3,0)\psline[linecolor=black]{<-}(0,0.2)(1.3,0.2)\\
\overline{a}_{n-1}
\end{array}\quad n$$
\end{center}

\vspace{0.5cm}

In \cite{BES} the authors used the fact that $\Lambda$ is
self-injective to prove that $\Lambda$ is a periodic algebra. Note
that the path algebra $KQ$ admits an obvious involutive
anti-isomorphism $(-)^-:KQ\longrightarrow KQ$
($x\rightsquigarrow\bar{x}$) which fixes the vertices and the arrow
$\epsilon$ and maps $a_i\rightsquigarrow\bar{a}_i$ and
$\bar{a}_i\rightsquigarrow a_i$, for all $i=1,...,n-1$. It clearly
preserves the relations for $\Lambda$, and hence it induces another
involutive anti-isomorphism $(-)^-:\Lambda\longrightarrow\Lambda$.
We shall call it the \emph{canonical (involutive) antiautomorphism}
of $\Lambda$.

The next proposition shows that we can apply to $\Lambda$ the
results in the previous subsection. It also fixes the basis of
$\Lambda$ with which we shall work all through the paper.

\begin{prop} \label{prop.fixed basis to work}
Let $\Lambda=P(\mathbb{L}_n)$ be the preprojective algebra of type L
and put $B=\bigcup_{i,j}e_iBe_j$, where

\begin{enumerate}
\item[a)] $e_1Be_1= \{e_1, \epsilon, \epsilon^2, \dots,
\epsilon^{2n-1}\}$

\item[b)] $e_1Be_j= \{a_1\cdots a_{j-1}, \epsilon a_1\cdots a_{j-1},
\epsilon^2 a_1\cdots a_{j-1}, \dots, \epsilon^{2(n-j)+1}a_1\cdots
a_{j-1}\}$ in case $j\neq 1$

\item[c)] $e_iBe_j= \{a_i\cdots a_{j-1}, a_i\cdots a_j\bar{a}_j, \dots,
a_i\cdots a_{n-1}\bar{a}_{n-1}\cdots\bar{a}_{j}\}\bigcup $

$\hspace{1.3cm}\{\bar{a}_{i-1}\cdots \bar{a}_1\epsilon a_1\cdots
a_{j-1}, \bar{a}_{i-1}\cdots \bar{a}_1\epsilon^3 a_1\cdots a_{j-1},
\dots,$

$\hspace{1.3cm}(-1)^{s_{ij}}\bar{a}_{i-1}\cdots
\bar{a}_1\epsilon^{2(n-j)+1} a_1\cdots a_{j-1}\}$

 where $s_{ij}=0$ for $i\neq j$ and $s_{ii}=\frac{i(i-1)}{2}$,
whenever $1<i\leq j\leq n$ (here we convene that $a_i...a_{j-1}=e_i$
in case $i=j$).

\item[d)] $e_iBe_j=\{\bar{b}:$ $b\in e_jBe_i\}$ in case $i>j$,
\end{enumerate}

Then $B$ is a dualizable basis of $\Lambda$. In particular,
$\Lambda$ is a symmetric algebra.

\end{prop}
\begin{proof}

Note that $e_iBe_j$ contains, at most, one element of a given
degree. In order to see that $B$ is a basis of $\Lambda$ we just
need to see that all the paths in $e_iBe_j$ are nonzero and that
they generate $e_i\Lambda e_j$ as a $K$-vector space. Note that then
$dim(e_i\Lambda_k e_j)\leq 1$, for all $i,j\in Q_0$ and $k\geq 0$,
and Lemma \ref{lema:basis for graded-Schurian algebras} can be
applied.

Suppose that we have already proved that $B$ is a basis of
$\Lambda$.  We claim that the condition on the parallel paths holds.
Indeed $\Lambda$ is a graded algebra and, given $i,j\in Q_0$ and
$n\geq 0$ integer, there is at most one element in $e_iBe_j$ of
degree $n$. It follows that any path $p:i\rightarrow...\rightarrow
j$ which is not in $I$ satisfies that  $p-\lambda q\in I$, for some
$0\neq\lambda\in K$, where $q\in e_iBe_j$ is the only element in
$e_iBe_j$ of degree equal to $\text{length}(p)$. But   the shape of
the relations which generate $I$ implies that $\lambda =(-1)^s$, for
some integer $s$. Therefore, given two parallel paths $p$ and $q$ of
equal length which do not belong to $I$, one has that either $p-q$
or $p+q$ belong to $I$.

We now pass  to check that all the paths in $e_iBe_j$ are nonzero
and that they generate $e_i\Lambda e_j$ as a $K$-vector space.
Assume that $i, j \in Q_0$ are vertices such that $i\leq j$. The
antiautomorphism $(-)^-$ given before guarantees that once we have a
basis for $e_iBe_j$, the remaining cases, $e_jBe_i$, can be
described by adding bars to the monomials obtained for $e_iBe_j$.

Observe that for each vertex $i\neq 1$ we have, up to sign, a unique
cycle of minimum length, namely $a_i\bar{a}_i$. However for the
vertex $i=1$ we do not only have the cycle $a_1\bar{a}_1=
-\epsilon^2$ but also the loop $\epsilon$.

Let $0\neq b$ be a monomial of a fixed length starting at $i$ and
ending at $i+s$. The previous comment tell us that $b$ contains
either an even number or an odd number of arrows of type $\epsilon$.

In the first case, the equality $(\bar{a}_{i-1}a_{i-1})a_i \cdots
a_{i+s-1}= (-1)^s a_i \cdots a_{i+s-1}(a_{i+s}\bar{a}_{i+s})$ shows
that $b$ has at most $n-i$ non-bar letters and $n-(s+i)$ bar
letters. Thus we can set as a basis element the non-zero path $b=
a_i \cdots a_{i+s+j}\bar{a}_{i+s+j}\cdots \bar{a}_{i+s}$ ($j\leq
n-1-i-s$), that is, where all the bar letters are to the right.

On the contrary, if $b$ contains and odd number of $\epsilon$
arrows, we have that

$$b= (a_i \bar{a}_i)(\bar{a}_{i-1}\cdots
\bar{a}_1\epsilon^{2t-1}a_1\cdots a_{i+s-1})=
(-1)^i(\bar{a}_{i-1}\cdots \bar{a}_1\epsilon^{2t+1}a_1\cdots
a_{i+s-1})$$ which is, up to sign, equal to

$$\bar{a}_{i-1}\cdots \bar{a}_1\epsilon (a_1\cdots a_t
\bar{a}_t\cdots \bar{a}_1a_1\cdots a_{i+s-1})$$ But notice that the
arrows between brackets form a path with an even number of
$\epsilon$ arrows which is in time, up to sign, equal to $a_1 \cdots
a_{t+i+s-1}\bar{a}_{t+i+s-1}\cdots \bar{a}_{i+s}$ . Hence we can
conclude that $\bar{a}_{i-1}\cdots \bar{a}_1\epsilon^{2t+1}a_1\cdots
a_{i+s-1}$ is a non zero path if and only if $0 \leq t\leq n-
(s+i)$. Thus the sets given in the statement are in fact a basis of
$\Lambda$.

It remains to prove that $B$ is a dualizable basis. This task is
reduced to prove that $a^*a=w_{t(a)}$, for each $a\in Q_1$. We have
$\omega_{i(\epsilon )}= \epsilon^{2n-1}$, hence $\epsilon^*=
\epsilon^{2n-2}$ and we clearly have $\epsilon^*\epsilon=
\omega_{i(\epsilon )}$.

For $a_i$ (i=1, \dots, n-1) we have

$$a_i[\bar{a}_i\cdots \bar{a}_1\epsilon^{2(n-i-1)+1}a_1\cdots
a_{i-1}]=(-1)^i \bar{a}_{i-1}\cdots
\bar{a}_1\epsilon^{2(n-i)+1}a_1\cdots a_{i-1}=$$

$$(-1)^i(-1)^{\frac{i(i-1)}{2}}\omega_i=
(-1)^{\frac{i(i+1)}{2}}\omega_i=(-1)^{\frac{i(i+1)}{2}}\omega_{i(a_i)}.$$
Then $a_i^*=(-1)^{\frac{i(i+1)}{2}}\bar{a}_i\cdots
\bar{a}_1\epsilon^{2(n-i-1)+1}a_1\cdots a_{i-1}$ and therefore

$$a_i^*a_i= (-1)^{\frac{i(i+1)}{2}}\bar{a}_i\cdots
\bar{a}_1\epsilon^{2(n-i-1)+1}a_1\cdots a_{i-1}a_i=
\omega_{i+1}=\omega_{t(a_i)}$$

The argument is symmetric for the arrows $\bar{a_i}$ and therefore
the basis $B$ is dualizable.

\end{proof}

\begin{rem}\rm
If one modifies the basis $B$ of proposition \ref{prop.fixed basis
to work}, by putting $\omega_i=\bar{a}_{i-1}\cdots \bar{a}_1
\epsilon^{2(n-i)+1}$ $a_1\cdots a_{i-1}$, for all $i=1, \dots, n$,
then the resulting basis is no longer dualizable. Indeed the proof
of the lemma shows that $a_i^*=(-1)^i\bar{a}_i\bar{a}_{i-1}\cdots
\bar{a}_1 \epsilon^{2(n-i-1)+1}a_1\cdots a_{i-1}$ in the new
situation, and then $a_i^*a_i=(-1)^iw_{i+1}$.
\end{rem}

By \cite{BES}, we know that the third syzygy of $\Lambda$ as a
bimodule is isomorphic to $_1\Lambda_\tau$, for some $\tau\in
Aut(\Lambda )$ such that $\tau^2=id_\Lambda$. Our emphasis on
choosing a dualizable basis on $\Lambda$ comes from the fact that it
allows a very precise determination of $\tau$. Indeed, combining
results of \cite{BES} and \cite{ES3}, we know that if $B$ is a
dualizable basis, then the initial part of the minimal projective
resolution of $\Lambda$ as a bimodule is:

$$0\rightsquigarrow
N\stackrel{\iota}{\hookrightarrow}P\stackrel{R}{\longrightarrow}Q\stackrel{\delta}{\longrightarrow}P\stackrel{u}{\longrightarrow}\Lambda\rightarrow
0,$$ where $P=\oplus_{i\in Q_0}\Lambda e_i\otimes e_i\Lambda$,
$Q=\oplus_{a\in Q_1}\Lambda e_{i(a)}\otimes e_{t(a)}\Lambda$ and $N$
is the $\Lambda$-subbimodule of $P$ generated by the elements
$\xi_i=\sum_{x\in e_iB}(-1)^{deg(x)}x\otimes x^*$, where $B$ is any
given basis of $\Lambda$ consisting of paths and negative of paths
which contains the vertices, the arrows and a basis of $Soc(\Lambda
)$. Here
 $\iota$ is the inclusion, $u$ is the multiplication map and $R$
and $\delta$  are as in proposition \ref{projetive resolution}
below.

The following result was proved in \cite{BES}.

\begin{lema}[see \cite{BES}, Proposition 2.3] \label{lema.correction of BES proposition}
Let $B$ be a dualizable basis of $\Lambda$, let $N$ be the
$\Lambda$-bimodule mentioned above and let $\tau\in Aut(\Lambda )$
be the only automorphism of $\Lambda$ such that $\tau (e_i)=e_i$ and
$\tau (a)=-a$, for all $i\in Q_0$ and $a\in Q_1$. There is an
isomorphism of $\Lambda$-bimodules $\phi:
_1\Lambda_\tau\stackrel{\cong}{\longrightarrow}N$ mapping
$e_i\rightsquigarrow\xi_i$, for each $i\in Q_0$.
\end{lema}

\begin{rem} \label{rem.Eus-dualizable-basis} \rm
The dualizable basis hypothesis does not appear in the statement of
Proposition 2.3 in \cite{BES}. However, it is implicitly used in the
proof of \cite{BES}[Lemma 2.4].
 From our work with
examples it seems that, without that extra hypothesis, the element
$\sum_{x\in e_iB}(-1)^{deg(x)}x\otimes x^*$ need not be in $Ker(R)$.

The dualizable hypothesis seems to be implicitly used also in the
argument of \cite{Eu}[Section 7.1], where the corresponding result
(with the automorphism $\tau$ conveniently modified) is proved. In
both cases, the crucial point is to guarantee that if $x\in B$ is a
homogeneous element of the basis $B$ of degree $>0$, then, for any
arrow $a\in Q_1$, the element $ax^*$  (resp. $x^*a$) should again be
of the form $y^*$, for some $y\in B$, whenever the product is
nonzero. This follows immediately in case one has $a(ya)^*=y^*$ and
$(ay)^*a=y$, for all $y\in B$ and $a\in Q_1$. This is precisely the
statement of Lemma 2.4 in \cite{BES} and is implicit in the argument
of \cite{Eu}[Section 7.1].

Essentially by the proof of our Lemma \ref{lema:basis for
graded-Schurian algebras}, we see that the mentioned crucial point
is tantamount to require that $B$ is a dualizable basis and that
$(-,-)$ is its associated Nakayama form. If, as in the spirit of
\cite{Eu}[Section 6.3],  one  has from the beginning a symmetric
Nakayama form $(-,-)$ such that $(e_i,e_i)=0$, for all $i\in Q_0$,
and finds a basis $B$ consisting of homogeneous elements which
contains the vertices and has the property that the dual elements
$\{w_i:=e_i^*:$ $i\in Q_0\}$ (in $B^*$) belong to $B\cap
Soc(\Lambda)$, then one readily sees that $B$ is dualizable and
$(-,-)$ is its associated Nakayama form.
\end{rem}

In the rest of the paper, the basis $B$ will be always that of
proposition \ref{prop.fixed basis to work}. The following properties
can be derived in a routinary way. We leave the verifications to the
reader.

\begin{cor}\label{Basis properties} Let $i, j\in Q_0$ be vertices. The following holds:

\begin{enumerate}

    \item The set of possible degrees of
    the elements in $e_iBe_j$ is

    $$\{j-i, j-i + 2, j -i + 4, \dots j-i + 2(n-j)= 2n - (i+j)\} \bigcup$$

    $$\{j+i -1, j+i +1, j+i+3, \dots j+i + 2(n-max(i,j))- 1\}$$

    \item If $\bar{a}_{i-1}\cdots \bar{a}_1\epsilon^{2k}a_1\cdots
    a_{j-1}$ is a nonzero element of $\Lambda$, then $k\leq
    n-i-j+1$.

    \item $a_1\cdots a_{j-1}\bar{a}_{j-1}\cdots \bar{a}_1= (-1)^{\frac{j(j-1)}{2}}\epsilon^{2(j-1)}$
    for $j=2, \dots, n$.

    \item $a_1\cdots a_{j-1}\bar{a}_{j-1}= (-1)^{j-1}\epsilon^2a_1\cdots a_{j-2}$

    \item
    $\bar{a}_ia_i...a_j=(-1)^{j-i+1}a_{i+1}...a_{j+1}\bar{a}_{j+1}$
    whenever $i\leq j< n$ (convening that $a_{n}=0$).

    \item $\text{dim}(Hom_{\Lambda^e}(P, \Lambda))= \sum_{i=1}^n \text{dim}(e_i\Lambda e_i)= \sum_{i=1}^n
    [2(n-i)+2]= n^2 + n$

    \item $\text{dim}(Hom_{\Lambda^e}(Q, \Lambda))= \text{dim}(e_1\Lambda e_1) + 2\sum_{i=1}^{n-1}(e_i\Lambda e_{i+1})
    = 2n + 2\sum_{i=1}^{n-1}[2(n-i-1) + 1]= 2n^2$

    \item The Cartan matrix of $\Lambda$ is given by:

        \vspace{0.5cm}

        $$\mathcal{C}_{P(\mathbb{L}_{n})}= \left( \begin{tabular}{c|cccc}

            2n & 2(n-1) & 2(n-2) & $\cdots$ & 2 \\ \hline

            2(n-1) & & & & \\

            $\vdots$ & & $\mathcal{C}_{P(\mathbb{L}_{n-1})}$ & & \\

            2 & & & &

       \end{tabular}\right)$$

       where

        $$\mathcal{C}_{P(\mathbb{L}_{2})}= \left( \begin{tabular}{c c}

            4 & 2\\

            2 & 2\\

        \end{tabular}\right)$$
        Its determinant is $det (\mathcal{C}_{P(\mathbb{L}_{n})})=2^n$ (see remark 3.3 in \cite{HZ}).

\end{enumerate}

\end{cor}

\subsection{The minimal projective resolution of $\Lambda$}

We are now ready to give all the modules and maps of the minimal
projective resolution $\Lambda =P(\mathbb{L}_n)$ as a bimodule.

\begin{prop}\label{projetive resolution} Let $\Lambda
=P(\mathbb{L}_n)$ be the preprojective algebra of type
$\mathbb{L}_n$, let $B$ be the dualizable basis of proposition
\ref{prop.fixed basis to work} and let $\tau\in Aut(\Lambda )$ the
algebra automorphism that fixes the vertices and satisfies that
$\tau (a)=-a$, for all $a\in Q_1$. The chain complex $\dots
P^{-2}\stackrel{d^{-2}}{\longrightarrow}
P^{-1}\stackrel{d^{-1}}{\longrightarrow}P^0\stackrel{u}{\longrightarrow}\Lambda\longrightarrow
0$  identified by the following properties is a minimal projective
resolution of $\Lambda$ as a bimodule:

\begin{enumerate}[a)]

    \item $P^{-n}= Q:= \bigoplus_{a\in Q_1}\Lambda e_{i(a)}\otimes
    e_{t(a)}\Lambda$ if $n\equiv 1 \text{(mod 3)}$ and $P^{-n}=P:= \bigoplus_{i\in Q_0} \Lambda e_i\otimes
     e_i\Lambda$ otherwise.

     \item $u$ is the multiplication map, $d^m= (d^n)_{\tau}$
     whenever $m-n=\pm 3$ and the initial differentials $d^{-1}=:
     \delta$, $d^{-2}=: R$ and $d^{-3}=: k$ are the only homomorphisms of
     $\Lambda$-bimodules satisfying:

     \begin{enumerate}[i)]

        \item $\delta(e_{i(a)}\otimes e_{t(a)})= a\otimes e_{t(a)}- e_{i(a)}\otimes a$

        \item $R(e_i\otimes e_i)= \sum_{a\in Q_1  i(a)=i}e_{i(a)}\otimes \bar{a}+ a\otimes e_{i(a)}$

        \item $k(e_i\otimes e_i)= \sum_{x\in e_iB}(-1)^{deg(x)}x\otimes x^*$

     \end{enumerate}
    for all $a\in Q_1$ and $i\in Q_0$.

\end{enumerate}

\end{prop}

\begin{proof} By Lemma \ref{lema.correction of BES proposition} (see \cite{BES}, Prop. 2.3) we have an exact sequence
of $\Lambda$-bimodules:

$$0\longrightarrow _1\Lambda_{\tau}\stackrel{j}{\longrightarrow} P\stackrel{R}{\longrightarrow}
Q\stackrel{\delta}{\longrightarrow}P\stackrel{u}{\longrightarrow}\Lambda
\longrightarrow 0,$$ where the map $j$ satisfies that $j(e_i)=
\sum_{x\in e_iB}(-1)^{deg(x)}x\otimes x^*$ for each $i\in Q_0$.

Applying the self-equivalence $F_{\tau}:
_{\Lambda}\text{Mod}_{\Lambda}\longrightarrow
_{\Lambda}\text{Mod}_{\Lambda}$, which acts as the identity on
morphisms, and bearing in mind that $\tau^2= 1_{\Lambda}$, we get an
exact sequence

$$0\longrightarrow \Lambda \stackrel{j}{\longrightarrow} _1P_{\tau}\stackrel{R}{\longrightarrow}
_1Q_{\tau}\stackrel{\delta}{\longrightarrow}_1P_{\delta}\stackrel{u}{\longrightarrow}
_1\Lambda_{\tau}\longrightarrow 0$$

By Lemma \ref{lema.equivalencia en bimodulos proyectivos},  we then
get an exact sequence of $\Lambda$-bimodules

$$0\longrightarrow \Lambda\stackrel{\tilde{j}}{\longrightarrow}P \stackrel{R_{\tau}}{\longrightarrow}
Q \stackrel{\delta_{\tau}}{\longrightarrow}P
\stackrel{\tilde{u}}{\longrightarrow}
_1\Lambda_{\tau}\longrightarrow 0,$$ where, if $\psi
:G_\tau\stackrel{\cong}{\longrightarrow}F_\tau$ denotes the natural
isomorphism of lemma \ref{lema.equivalencia en bimodulos
proyectivos}, then  $\tilde{u}= u\circ\psi_P: a\otimes b
\rightsquigarrow a\tau(b)$ and $\tilde{j}= \psi_P^{-1}\circ j$ which
takes $e_i \rightsquigarrow -\sum_{x\in e_iB}x\otimes x^*$.

The composition $P \stackrel{\tilde{u}}{\longrightarrow}
_1\Lambda_{\tau}\stackrel{j}{\longrightarrow} P$ takes $e_i\otimes
e_i \rightsquigarrow \sum_{x\in e_iB}(-1)^{deg(x)}x\otimes x^*$ and,
hence, coincides with the morphims $k$ given in the statement.
Finally, the composition $P\stackrel{u}{\longrightarrow}\Lambda
\stackrel{\tilde{j}}{\longrightarrow}P$ takes $e_i\otimes
e_i\rightsquigarrow -\sum_{x\in e_iB}x\otimes x^*= \sum_{x\in
e_iB}(-1)^{deg(x)}x\otimes \tau(x^*)= k_{\tau}(e_i\otimes e_i)$.
Therefore $\tilde{j}\circ u= k_{\tau}$. The rest of the proof is
clear.
\end{proof}

\subsection{A cochain complex which gives the Hochschild cohomology}

 Recall that if $f: \oplus_{s=1}^{m}\Lambda e_{i_s}\otimes
e_{j_s}\Lambda \longrightarrow \oplus_{t=1}^p \Lambda e_{k_t}\otimes
e_{l_t}\Lambda$ $(i_s, j_s, k_t, l_t \in Q_0)$ is a morphism of
$\Lambda$-bimodules, an application of the contravariant functor
$Hom_{\Lambda^e}(_,\Lambda):
_{\Lambda}\text{Mod}_{\Lambda}\longrightarrow _k\text{Mod}$ gives a
$K$-linear map

$$f^*: Hom_{\Lambda^e}(\oplus_{t=1}^p \Lambda e_{k_t}\otimes e_{l_t},
\Lambda)\longrightarrow Hom_{\Lambda^e}(\oplus_{s=1}^m\Lambda
e_{i_s}\otimes e_{j_s}\Lambda, \Lambda).$$ Due to the isomorphism of
$K$-vector spaces $Hom_{\Lambda^e}(\Lambda e_i\otimes e_j \Lambda,
\Lambda)\cong e_i\Lambda e_j$, for all $i,j \in Q_0$, we get an
induced map, still denoted the same $f^*: \oplus_{t=1}^p
e_{k_t}\Lambda e_{l_t}\longrightarrow \oplus_{s=1}^{m}e_{i_s}\Lambda
e_{j_s}$. As usual we will also denote by $J=J(\Lambda)$ the
Jacobson radical of $\Lambda$. With this terminology, we get:

\begin{prop} \label{cohomology complex} For each $n\geq 0$, $HH^n(\Lambda)$ is the $n$-th
cohomology space of the complex

$$V^{\bullet}: \cdots 0 \longrightarrow \oplus_{i\in Q_0}e_i\Lambda e_i\stackrel{\delta^*}{\longrightarrow}
\oplus_{a\in Q_1}e_{i(a)}\Lambda
e_{t(a)}\stackrel{R^*}{\longrightarrow}\oplus_{i\in Q_0}e_i\Lambda
e_i \stackrel{k^*}{\longrightarrow}\oplus_{i\in Q_0}e_i\Lambda e_i$$

$$\stackrel{\delta_{\tau}^*}{\longrightarrow} \oplus_{a\in
Q_1}e_{i(a)}\Lambda e_{t(a)}
\stackrel{R_{\tau}^*}{\longrightarrow}\oplus_{i\in Q_0}e_i\Lambda
e_i \stackrel{k_{\tau}^*}{\longrightarrow}\oplus_{i\in
Q_0}e_i\Lambda e_i \stackrel{\delta^*}{\longrightarrow}\oplus_{a\in
Q_1}e_{i(a)}\Lambda e_{t(a)}\cdots$$

\noindent where $V^0= \sum_{i\in Q_0}e_i\Lambda e_i$ and $V^n=0$
$\forall n<0$. Moreover, viewing $\oplus_{i\in Q_0}e_i\Lambda e_i$
and $\oplus_{a\in Q_1} e_{i(a)}\Lambda e_{t(a)}$ as subspaces of
$\Lambda$, the differentials of $V^{\bullet}$ act as follows for
each oriented cycle $c$ at $i$ and each path $p: i(a)\rightarrow
\cdots \rightarrow t(a)$ :

\begin{enumerate}[a)]

    \item $\delta^*(c)= a_{i-1}c - c\bar{a}_{i-1} + \bar{a}_ic - ca_i$

    \item $R^*(p)= p\bar{a}+ \bar{a}p$

    \item $k^*(c)= 0$  (i.e. $k^*$ is the zero map)

    \item $\delta_{\tau}^*(c)= a_{i-1}c + c\bar{a}_{i-1} + \bar{a}_ic + ca_i$

    \item $R_{\tau}^*(p)= p\bar{a} - \bar{a}p$

    \item $k_{\tau}^*(c)=0$ if $c\in e_iJe_i$, and $k_{\tau}^*(e_i)=- \sum_{j\in Q_0}\text{dim}(e_i\Lambda e_j)\omega_j$

\end{enumerate}

where we convene that $a_0= \bar{a}_0= \epsilon$ and
$a_n=\bar{a}_n=0$

\end{prop}

\begin{proof} $HH^n(\Lambda)$ is the $n$-th cohomology space of the
complex obtained by applying $Hom_{\Lambda^e}(_, \Lambda)$ to the
minimal projective resolution of $\Lambda$ as bimodule. The
$K$-vector spaces of that complexes are precisely those of
$V^{\bullet}$ and the only nontrivial part is the explicit
definition of its differentials.

We have two canonical isomorphisms of $k$-vector spaces:

$$\oplus_{j\in Q_0}e_j\Lambda e_j\stackrel{\sim}{\longrightarrow}
Hom_{\Lambda^e}(\oplus_{j\in Q_0}\Lambda e_j\otimes e_j\Lambda,
\Lambda)$$

$$\oplus_{a\in Q_1}e_{i(a)}\Lambda e_{t(a)}\stackrel{\sim}{\longrightarrow}
Hom_{\Lambda^e}(\oplus_{a\in Q_1}\Lambda e_{i(a)}\otimes
e_{t(a)}\Lambda, \Lambda)$$The first one matches a nonzero oriented
cycle $c$ at $i$ with the morphism of $\Lambda$-bimodules
$\oplus_{j\in Q_0}\Lambda e_j\otimes
e_j\Lambda\stackrel{\tilde{c}}{\longrightarrow} \Lambda$ taking
$e_j\otimes e_j\rightsquigarrow \delta_{ij}c$, where $\delta$ is the
Kronecker symbol. Similarly a nonzero path $p: i(a)\rightarrow
\cdots \rightarrow t(a)$ is matched by the second isomorphism with
the morphism of $\Lambda$-bimodules $\oplus_{b\in Q_1}\Lambda
e_{i(b)}\otimes e_{t(b)}\Lambda\longrightarrow \Lambda$ taking
$e_{i(b)}\otimes e_{t(b)}\rightsquigarrow\delta_{ab}p$. With these
matches in mind the task of checking that the explicit definition of
the differentials is the giving one is routinary and mainly left to
the reader. We just do some samples:

    a) $\delta^*(c)$ is the element of $\oplus_{b\in Q_1} e_{i(b)}\Lambda e_{t(b)}$
    matched with $\tilde{c}\circ \delta\in Hom_{\Lambda^e}(\oplus_{b\in Q_1}
    \Lambda e_{i(b)}\otimes e_{t(b)}\Lambda, \Lambda)$. Then

    $$\delta^*(c)= \sum_{b\in Q_1}(\tilde{c}\circ \delta)(e_{i(b)}\otimes e_{t(b)})=
    \sum_{b\in Q_1}\tilde{c}(b\otimes e_{t(b)}-e_{i(b)}\otimes b)=$$

    $$\sum_{b\in Q_1}[b\tilde{c}(e_{t(b)}\otimes e_{t(b)})-\tilde{c}(e_{t(b)}\otimes e_{i(b)})b]=
    \sum_{b \in Q_1, t(b)=i}bc \hspace{0.5cm- \hspace{0.4cm}}\sum_{b\in Q_1, i(b)=i}cb=$$

    $$a_{i-1}c + \bar{a}_ic - ca_i - c\bar{a}_{i-1}$$

    c) $k^*(c)$ is the element of $\oplus_{j\in Q_0}e_j\Lambda
    e_j$ matched with $\tilde{c}\circ k\in Hom_{\Lambda^e}(\oplus_{j\in Q_0}\Lambda e_j\otimes e_j\Lambda,
    \Lambda)$. Then

    $$k^*(c)= \sum_{j\in Q_0}(\tilde{c}\circ k)(e_j\otimes e_j)= \sum_{j\in Q_0}\tilde{c}
    (\sum_{x\in e_jB}(-1)^{deg(x)}x\otimes x^*)= \sum_{j\in Q_0}\sum_{x\in e_jBe_i}(-1)^{deg(x)}xcx^*$$ But $xcx^*=0$ in case $deg(c)>0$ because $xx^*=\omega_j$ is an
    element in the socle. In case $c=e_j$ we have $k^*(e_j)= \sum_{j\in Q_0}\sum_{x\in
    e_jBe_i}(-1)^{deg(x)}xx^*$. Bearing in mind that $xx^*=
    \omega_j$ for each $x\in e_jBe_i$ and that the number of
    elements in $e_jBe_i$ with even degree is the same as the number
    of those with odd degree, we conclude that also $k^*(e_i)=0$.
    Since $k^*$ vanishes on all nonzero oriented cycles it follows
    that $k^*=0$.

    f) Arguing similarly with $k_{\tau}^*$ we get that

    $k_{\tau}^*(c)=0$ if $deg(c)>0$ and

    $k_{\tau}^*(e_i)= -\sum_{j\in Q_0}\sum_{x\in e_jBe_i}xx^*=
    -\sum_{j\in Q_0}\text{dim}(e_j\Lambda e_i)\omega_j$

\end{proof}

\begin{rem} \label{rem.Eus-complex-V} \rm
With the adequate change of presentation of the algebra, the complex
$V^\bullet$ should correspond to the sequence of morphisms  in
\cite{Eu}[Section 7.4], although the there defined differentials
seem not to make it into a complex.
\end{rem}

\begin{cor} $\Lambda$ is a symmetric periodic algebra of period $6$
and $\mathcal{P}(\Lambda, \Lambda)= Soc(\Lambda)$ when we view the
isomorphism $HH^0(\Lambda)\cong Z(\Lambda)$ as an identification.

\end{cor}

\begin{proof} By \ref{prop.fixed basis to work}, we know that
$\Lambda$ is symmetric, and by \ref{projetive resolution} $\Lambda$
is periodic of period 6.

To see that the isomorphism $HH^0(\Lambda)\cong Z(\Lambda)$
identifies $\mathcal{P}(\Lambda, \Lambda)$ with $Soc(\Lambda)=
Soc(Z(\Lambda))$, note that from \ref{projetive resolution} it
follows that a minimal complete resolution of $\Lambda$ is given by

$$\cdots P^{-2}\longrightarrow P^{-1}\stackrel{d^{-1}}{\longrightarrow}P^0
\stackrel{d^0}{\longrightarrow}P^1\stackrel{d^1}{\longrightarrow}P^2
\longrightarrow \cdots$$ where $P^n= Q= \oplus_{a\in Q_1} \Lambda
e_{i(a)}\otimes e_{t(a)}\Lambda$, when $n\equiv -1 (mod 3)$, and
$P^n=\oplus_{i\in Q_0}\Lambda e_i \otimes e_i \Lambda$ otherwise,
and the arrows are given by $d^m=(d^n)_{\tau}$ whenever $m\equiv n
(mod 3)$ and $d^{-1}= \delta$, $d^{-2}=R$ and $d^{-3}=k$. It follows
that $\underline{HH}^*(\Lambda)$ is the cohomology of the complex

$$\cdots V^{-2}\stackrel{R_{\tau}^*}{\longrightarrow} V^{-1}\stackrel{k_{\tau}^*}{\longrightarrow}
V^0 \stackrel{\delta^*}{\longrightarrow}V^1
\stackrel{R^*}{\longrightarrow} V^2\longrightarrow \cdots$$ In
particular, we have $\underline{HH}^0(\lambda)=
\frac{Ker(\delta^*)}{Im(k_{\tau}^*)}$. But $Ker(\delta^*)=
HH^0(\Lambda)= Z(\Lambda)$ while $Im(k_{\tau}^*)= Soc(\Lambda)$
since the Cartan matrix of $\Lambda$ is invertible. Note that the
isomorphism $Z(\Lambda)\cong End_{\Lambda^e}(\Lambda)$ identifies
$Im(k_{\tau}^*)$ with $\mathcal{P}(\Lambda, \Lambda)$.

\end{proof}

\begin{cor} \label{structure of homology as graded module}
There are isomorphisms of graded
$\underline{HH}^*(\Lambda)$-modules:

$$\underline{HH}^*(\Lambda)\cong \underline{HH}^*(\Lambda)[6]$$

$$\underline{HH}_{-*}(\Lambda)\cong D(\underline{HH}^*(\Lambda))
\cong \underline{HH}^*(\Lambda)[5]$$ and isomorphisms of graded
$HH^*(\Lambda)$-modules $HH_{-*}(\Lambda)\cong D(HH^*(\Lambda))$.
Moreover $\underline{HH}^{*}(\Lambda)$ is a graded Frobenius
algebra.

\end{cor}

\begin{proof} $\underline{HH}^*(\Lambda)\cong
\underline{HH}^*(\Lambda)[6]$ since $\Lambda$ is periodic of period
6. On the other hand, $\Lambda$ is $5$-CY Frobenius and, by
\ref{Eu-Schedler}, we have

$$D(\underline{HH}^*(\Lambda))\cong \underline{HH}^*(\Lambda)[11]$$

$$\underline{HH}_{-*}(\Lambda)\cong \underline{HH}^*(\Lambda)[5]$$
Then the isomorphisms in the statement follow. The graded Frobenius
condition of $\underline{HH}^*(\Lambda)$ follows from Theorem
\ref{Eu-Schedler}.

On the other hand, we have an isomorphism $HH_{-*}(\Lambda)\cong
D(HH^*(\Lambda, D(\Lambda)))\cong D(HH^*(\Lambda))$ due to remark
\ref{Stable duality} and the fact that $D(\Lambda)\cong \Lambda$.

\end{proof}

\section{The Hochschild cohomology spaces }

In  the rest of the paper we assume that $Char(K)\neq 2$.

In this section we will use the complex $V^\bullet$ of proposition
\ref{cohomology complex} to calculate the dimension and an
appropriate basis of each space $HH^i(\Lambda )$. In the proof of
the following lemma and in the rest of the paper, the matrix of a
lineaar map is always written by columns.

\begin{lema} \label{Image of RStar and RtauStar}
The equality $Im(R^*)=\oplus_{i\in Q_0}e_iJe_i$ holds and
$Im(R_\tau^*)$ is a subspace of codimension $n$ in $\oplus_{i\in
Q_0}e_iJe_i$. In particular, we have:

$$\text{dim}(Im(R^*))= n^2$$

$$\text{dim}(Im (R_{\tau}^*))= n^2-n.$$

\end{lema}

\begin{proof} We put $V= \oplus_{a\in Q_1}e_{i(a)}\Lambda e_{t(a)}$
and $W=\oplus_{i\in Q_0}e_iJe_i$ for simplicity and view $R^*$ and
$R_{\tau}^*$ as $K$-linear maps $V\longrightarrow W$. For each
$0\leq k < 2n$ we denote by $V_k$ (resp. $W_k$) the vector subspace
consisting of the elements os degree $k$. Since both $R^*$ and
$R_{\tau}^*$ are graded maps of degree 1 we have induced $K$-linear
maps

$$R^*, R_{\tau}^*: V_{k-1}\longrightarrow W_k$$ for $k=1, \dots, 2n-1$.

It is important now to notice that the canonical antiisomorphism of
$\Lambda$, $x\rightsquigarrow \bar{x}$, is the identity on $W$.
Moreover, we have equalities $R^*(\bar{p})= \overline{R^*(p)}$ and
$R_{\tau}^*(\bar{p})= -\overline{R_{\tau}^*(p)}$. We then get
$R^*(\bar{p})= R^*(p)$ and $R_{\tau}^*(\bar{p})= - R_{\tau}^*(p)$.
This tells us that the images of the maps $R^*, R_{\tau}^*:
V\longrightarrow W$ are the same as those of their restrictions to
$V^{+}= V\bigcap (\oplus_{j=0}^{n-1}e_{i(a_j)}\Lambda e_{t(a_j)})$
(convening that $a_0=\epsilon$). Those images are in turn the direct
sum of the images of the induced maps

$$R^*, R_{\tau}^*: V_{k-1}^+ \longrightarrow W_k \hspace{1.5cm} (k=1, \dots 2n-1)$$ and thus the ones we shall calculate.

Let us denote by $b_i^t$ the only element in $e_iBe_i$ of degree
$t$.

We start by considering the case when $k=2m$ is even $(1\leq m \leq
n-1)$. In that situation, a basis of $W_{2m}$ is given by
$\{b_1^{2m}, b_2^{2m}, \dots, b_{n-m}^{2m}\}$ while a basis of
$V_{2m-1}^+$ is $\{v_{\epsilon}, v_{a_1}, \dots, v_{a_{n-m}}\}$
 where $v_{\epsilon}= \epsilon^{2m-1}$ and $v_{a_i}= a_i\cdots a_{i+m-1}\bar{a}_{i+m-1}
 \cdots \bar{a}_{i+1}$ for $i=1, \dots , n-m$. In particular $\text{dim}(V_{2m-1}^+)=
 n-m+1$ and $\text{dim}(W_{2m})= n-m$. Direct computation, using remark \ref{Basis properties},  shows that

 \begin{enumerate}[i)]

    \item $R^*(v_{\epsilon})= 2b_1^{2m}$, \hspace{0.5cm} $R_{\tau}^*(v_{\epsilon})=0$

    \item $R^*(v_{a_1})= (-1)^{\frac{(m+1)m}{2}}b_1^{2m} + (-1)^m b_2^{2m}$

        $R_{\tau}^*(v_{a_1})= (-1)^{\frac{(m+1)m}{2}}b_1^{2m} + (-1)^{m+1}b_2^{2m}$

    \item $R^*(v_{a_i})= b_i^{2m} + (-1)^{m}b_{i+1}^{2m}$

        $R_{\tau}^*(v_{a_i})= b_i^{2m} + (-1)^{m+1}b_{i+1}^{2m}$

    (convening that $b_j^{2m}=0$ if $j>n-m$)

 \end{enumerate} Then in the matrices of $R^*$ and $R_{\tau}^*$ with respect to the
 given bases of $V_{2m-1}^+$ and $W_{2m}$, which are both of size $(n-m) \times
 (n-m+1)$, the columns from the $2^{nd}$ to the $(n-m+1)-th$ are
 linearly independent. We then get that the maps $R^*, R_{\tau}^*: V_{2m-1}\longrightarrow W_{2m}$
 are both surjective for each $m=1, \dots n-1$.

 We now deal with the case when $k=2m-1$ is odd, in which case a
 basis of $W_{2m-1}$ is $\{b_1^{2m-1}, \dots b_m^{2m-1}\}$. On the
 other hand, a basis of $V_{2m-2}^+$ is given by $\{v_{\epsilon}', v_{a_1}', \dots
 v_{a_{m-1}}'\}$,
 where $v_{\epsilon}'= \epsilon^{2m-2}$ and $v_{a_i}'= \bar{a}_{i-1}\cdots \bar{a}_1
 \epsilon^{2(m-i)-1}a_1\cdots a_i$ for $i=1, \dots , m-1$. Direct
 calculation, using again remark \ref{Basis properties}, shows the following:

 \begin{enumerate}[i)]

    \item $R^*(v_{\epsilon}')= 2b_1^{2m-1}$, \hspace{0.5cm} $R_{\tau}^*(v_{\epsilon}')=0$

    \item  If $m\neq n$ then

    $R^*(v_{a_i}')= (-1)^ib_i^{2m-1} + b_{i+1}^{2m-1}$

    $R_\tau^*(v_{a_i}')= (-1)^ib_i^{2m-1} - b_{i+1}^{2m-1}$

    \item If $m=n$ then

    $R^*(v_{a_i}')= (-1)^{\frac{(i+1)i}{2}}b_i^{2n-1} +
    (-1)^{\frac{(i+1)i}{2}}b_{i+1}^{2n-1}=(-1)^{\frac{(i+1)i}{2}}(w_i+w_{i+1})$

    $R_{\tau}^*(v_{a_i}')= (-1)^{\frac{(i+1)i}{2}}b_i^{2n-1} -
    (-1)^{\frac{(i+1)i}{2}}b_{i+1}^{2n-1}=(-1)^{\frac{(i+1)i}{2}}(w_i-w_{i+1})$

 \end{enumerate}Therefore, the square matrices of $R^*$ and $R_{\tau}^*$ with
 respect to the given bases of $V_{2m-2}^+$ and $W_{2m-1}$ are upper
 triangular. In the case of $R^*$ all its diagonal entries are
 nonzero while in the case of $R_{\tau}^*$ only the entry $(1,1)$ is
 zero. It follows:

 \begin{enumerate}[a)]

    \item The map $R^*: V_{2m-2}\longrightarrow W_{2m-1}$ is
    surjective for all $m=1, \dots n$.

    \item The image of the map $R_{\tau}^*: V_{2m-2}\longrightarrow
    W_{2m-1}$ has codimension $1$ in $W_{2m-2}$ for all $m=1, \dots,
    n$.

 \end{enumerate}

 The final conclusion is that the map $R^*: V\longrightarrow W$ is
 surjective while the image of $R_{\tau}^*: V\longrightarrow W$ has
 codimension exactly the number of odd numbers in $\{1,2, \dots
 2n-1\}$. That is $\text{dim}(W)- \text{dim}(Im(R_{\tau}^*))=n$.

\end{proof}

\begin{rem}\label{w_j - w_j+1 belongs to Im(RTauStar)}\rm The proof of
lemma \ref{Image of RStar and RtauStar} gives that if $\omega_j$ is
viewed as an element of $Ker(k_{\tau}^*)$ $\forall j\in Q_0$, then
$\omega_j - \omega_{j+1}\in Im(R_{\tau}^*)$ $\forall j=1,2,\dots,
n-1$.

\end{rem}

\begin{lema}\label{Center of lambda} The center of $\Lambda$ is isomorphic to $\frac{K[x_0, x_1, \dots
x_n]}{I}$, where $I$ is the ideal of $K[x_0, x_1, \dots x_n]$
generated by $x_0^n$ and all the products $x_ix_j$ with
$(i,j)=(0,0)$. In particular, $\text{dim}(HH^0(\Lambda))=2n$.

\end{lema}

\begin{proof} It is well-known that $Z(\Lambda)\subseteq \oplus_{i\in Q_0}e_i\Lambda
e_i$, that $rad(Z(\Lambda))= Z(\Lambda)\bigcap rad(\Lambda)$ and
$\frac{Z(\Lambda)}{rad(\Lambda)}= K\cdot1= K(e_1+ \cdots e_n)$.
Since $\Lambda$ is graded one readily sees that the grading on
$\Lambda$ gives by restriction a grading on $Z(\Lambda)$.

We claim that if $z\in Z(\Lambda)_{2m-1}$ is an element of odd
degree $2m-1$, then $m=n$ and $z$ is a linear combination of the
socle elements $\omega_1, \dots, \omega_n$. Indeed we have $z=
\sum_{i=1}^r \lambda_i b_i^{2m-1}$, with $\lambda_r\neq 0$, for some
integer $1\leq r\leq m$.  If $r<n$ then $\lambda_rb_r^{2m-1}a_r=
za_r=a_rz=0$, and hence $0= b_r^{2m-1}a_r= \bar{a}_{r-1}\cdots
\bar{a}_1\epsilon^{2(m-r)+1}a_1\cdots a_{r-1}a_r$. This only happens
when $m=n$, in which case $b_r^{2m-1}= b_r^{2n-1}= \omega _r$. On
the other hand, if $r=n$ then $n=m$ and we are done also in this
case.

The previous paragraph shows that $Z(\Lambda)_{odd}:=
\oplus_{m>0}Z(\Lambda)_{2m-1}= \sum_{i\in Q_0}K\omega_i=
Soc(\Lambda)$ since $\omega_i\in Z(\Lambda)_{2n-1}$ for each $i\in
Q_0$. We now want to identify
$Z(\Lambda)_{even}^+:=\oplus_{m>0}Z(\Lambda)_{2m}$. One easily
checks that $x_0= \sum_{i=1}^{n-1}(-1)^i a_i\bar{a}_i= b_1^2 +
\sum_{i=2}^{n-1}(-1)^ib_i^2$ is an element of $Z(\Lambda)_2$.
Moreover $(b_i^2)^m\neq 0$ if and only if $1\leq i\leq n-m$ and
$m<n$. In this case we necessarily have an equality $(b_i^2)^m=
(-1)^{t_i}b_i^{2m}$, for some integer exponent $t_i$. In particular
$x_0^m\neq 0$ and $x_0^m= \sum_{i=0}^{n-m}\lambda_ib_i^{2m}$, with
scalars $\lambda_i$ all nonzero. We claim that if $0\neq z\in
Z(\Lambda)_{2m}$ and we write it as a $K$-linear combination
$z=\sum_{i=1}^{n-m}\mu_ib_i^{2m}$, then $\mu_i\neq 0$ for all $i=1,
\dots , n-m$. Suppose that it is not the case. We first prove that
if $\mu_j=0$ then $\mu_i=0$ for each $i\leq j$. for that we can
assume $j>1$ and then we have

$$0= \mu_ja_{j-1}b_j^{2m}= a_{j-1}z= za_{j-1}= \mu_{j-1}b_{j-1}^{2m}a_{j-1}$$ But $b_{j-1}^{2m}a_{j-1}\neq 0$ since $j\leq n-m \leq n-1$ and so
$j-1< n-m$. It follows that $\mu_{j-1}=0$ and, by iterating the
process, that $\mu_i=0$ $\forall i\leq j$.

We can then write $z= \sum_{i=r}^{n-m} \mu_i b_i^{2m}$ for some
$1\leq r\leq n-m$ and some $mu_i\neq 0$ $\forall i=r, \dots, n-m$.
We prove that $r=1$ and our claim will be settled. Indeed, if
$r>1$ then we have

$$\mu_ra_{r-1}b_r^{2m}= a_{r-1}z= za_{r-1}=0$$ which implies that $\mu_r=0$ since $a_{r-1}b_r^{2m}\neq 0$. This is
a contradiction.

Once we know that if $z\in Z(\Lambda)_{2m}\backslash 0\}$ and
$z=\sum_{i=1}^{n-m}\mu_ib_i^{2m}$ then $\mu_i\neq 0$ $\forall i=1,
\dots n-m$, we conclude that any such $z$ is a scalar multiple of
$x_0^m$. Then $Z(\Lambda)_{2m}= Kx_0^m$, for each $m>0$.

Putting now $x_i=\omega_i$ $\forall i=1, \dots, n$ we clearly have
that $x_0, x_1, \dots x_n$ generate $Z(\Lambda)$ as an algebra and
they are subject to the relations $x_0^n=0$ and $x_ix_j=0$ for
$(i, j)\neq (0,0)$.

\end{proof}

We are now ready to prove the main result of this section:

\begin{teor}\label{dimension of HH^i}

Let us assume that $Char(K)\neq 2$ and let $\Lambda
=P(\mathbb{L}_n)$ be the preprojective algebra of type
$\mathbb{L}_n$. Then  $\text{dim}(HH^0(\Lambda))=
\text{dim}(HH_0(\Lambda))= 2n$ and $\text{dim}(HH^i(\Lambda))=
\text{dim}(HH_i(\Lambda))= n$ for all $i>0$.

\end{teor}

\begin{proof} By the isomorphism $HH_{-*}(\Lambda)\cong
D(HH^*(\Lambda))$ (see Remark \ref{Stable duality}), it is enough to
calculate the dimensions of the Hochschild cohomology spaces.

On the other hand, by Corollary \ref{structure of homology as graded
module}, we have and isomorphism $\underline{HH}^*(\Lambda ))\cong
\underline{HH}^*(\Lambda ))[6]$. We then  get isomorphisms of
$K$-vector spaces

\begin{center}
$HH^{6k}(\Lambda )\cong \underline{HH}^{0}(\Lambda
)=\frac{HH^0(\Lambda )}{\mathcal{P}(\Lambda ,\Lambda
)}=\frac{Z(\Lambda )}{Soc(\Lambda )}$

$HH^{6k+i}(\Lambda )\cong HH^i(\Lambda )$,
\end{center}
for all $k>0$ and $i=1,2,3,4,5$.

By the same corollary, we have an isomorphism
$D(\underline{HH}^*(\Lambda ))\cong \underline{HH}^*(\Lambda )[5]$,
which gives isomorphisms of $K$-vector spaces:

\begin{center}
$D(\underline{HH}^0(\Lambda ))\cong HH^5(\Lambda)$

$D(HH^1(\Lambda ))\cong HH^4(\Lambda )$

$D(HH^2(\Lambda ))\cong HH^3(\Lambda )$.
\end{center}

Bearing in mind Lemma 4.3, the proof is  reduced to check that

\begin{center}
$\text{dim}(\frac{Z(\Lambda )}{Soc(\Lambda
)})=\text{dim}(HH^1(\Lambda ))=\text{dim}(HH^2(\Lambda ))=n$.
\end{center}

That $\text{dim}(\frac{Z(\Lambda )}{Soc(\Lambda )})=n$ follows
directly from Lemma \ref{Center of lambda} and its proof.  Moreover,
we have two exact sequences

$$0\longrightarrow Ker(R^*)\hookrightarrow \oplus_{a\in Q_1}e_{i(a)}\Lambda e_{t(a)}
\stackrel{R^*}{\longrightarrow}\oplus_{i\in
Q_0}e_iJe_i\longrightarrow 0$$

$$0\longrightarrow Z(\Lambda)\hookrightarrow \oplus_{i\in Q_=}e_i\Lambda e_i
\longrightarrow Im(\delta^*)\longrightarrow 0$$

From the first one we get $\text{dim}(Ker(R^*))= 2n^2 - n^2= n^2$
using Lemma \ref{Image of RStar and RtauStar} and Corollary
\ref{Basis properties}. From the second sequence we get
$\text{dim}(Im(\delta^*))= (n^2 + n) - 2n= n^2 - n$ using lemma
\ref{Center of lambda}. It follows that
$\text{dim}(HH^1(\Lambda))=n$.

We also have that $HH^2(\Lambda)\cong Coker(R^*)$ since $k^*= 0$.
But $Im(R^*)= \oplus_{i\in Q_0}e_iJe_i$ by  lemma \ref{Image of
RStar and RtauStar}. It follows that
$\text{dim}(HH^2(\Lambda))=\text{dim}( \oplus_{i\in
Q_0}\frac{e_i\Lambda e_i}{e_iJe_i})=n$.

\end{proof}

Once we have computed the dimensions of the Hochschild (co)homology
spaces of $\Lambda$, we can do the same for its cyclic homology
spaces in characteristic zero, denoted by $HC_i(\Lambda )$ following
the notation used in \cite{Lo}. We start by recalling the following
fact about graded algebras.

\begin{prop} \label{prop.cyclic homology}
Suppose $Char(K)=0$ and let $A= \oplus_{i\geq 0}A_i$ be a positively
graded algebra such that $A_0$ is a semisimple algebra. The
following assertions hold:

\begin{enumerate}

    \item As $K$-vector spaces $HC_i(A_0)\cong \left\{\begin{array}{cc}

    0\hspace{0.25cm} \text{if i is odd}\\

    A_0 \hspace{0.25cm} \text{if i is even}\\

    \end{array}\right.$

    \item Connes' boundary map $B$ induces an exact sequence

    $$0\longrightarrow A_0 \longrightarrow HH_0(A)\stackrel{B}{\longrightarrow} HH_1(A)
    \stackrel{B}{\longrightarrow}HH_2(A)\longrightarrow \cdots$$
    such that the image of $B: HH_i(A)\longrightarrow HH_{i+1}(A)$
    is isomorphic to $\frac{HC_i(A)}{HC_i(A_0)}$, for all $n\geq 0$.

\end{enumerate}

\end{prop}

\begin{proof} Assertion $1$ is well-known, and is a direct
consequence of Connes' periodicity exact sequence (\cite{Lo},
Theorem 2.2.1) and the fact that $HH_i(A_0)=0$, for all $i>0$.

On the other hand, by \cite{Lo}, Theorem 4.1.13, we know that
Connes' periodicity exact sequence gives exact sequences:

$$0\longrightarrow \frac{HC_{i-1}(A)}{HC_{i-1}(A_0)}\stackrel{B}{\longrightarrow}
\frac{HH_i(A)}{HH_i(A_0)}\stackrel{I}{\longrightarrow}\frac{HC_i(A)}{HC_i(A_0)}
\longrightarrow 0$$ for all $i\geq 0$. Since $HH_i(A_0)=0$, for
$i>0$, we get an induced $K$-linear map $B\circ I:
HH_i(A)\longrightarrow HH_{i+1}(A)$ such that $Im(B\circ I)= Im
(B)\cong \frac{HC_i(A)}{HC_i(A_0)}.$

\end{proof}

\begin{cor} \label{cor:cyclic homology spaces}
If $\Lambda= P(\mathbb{L}_n)$ is the preprojective algebra of type
$\mathbb{L}_n$, then

$$dim HC_i(\Lambda)=\left\{\begin{array}{cc}

    0\hspace{0.25cm} \text{if i is odd} \\

    2n\hspace{0.25cm} \text{if i is even}  \\

    \end{array}\right.$$

\end{cor}

\begin{proof} Put $B^i:=
Im(HH_i(\Lambda)\stackrel{B}{\longrightarrow}HH_{i+1}(\Lambda))$
where $B$ is Connes' map. From the previous theorem, we have

$$dim(B^0)=dimHH_0(\Lambda) - dim(KQ_0)= 2n-n= n$$ and

$$dim(B^i)= dim HH_i(\Lambda)- dim(B^{i-1})= n-dim(B^{i-1})$$, for
all $i>0$.

It follows that $dim(B^i)=n$, when $i$ is even and zero otherwise.

Then we have $$dimHC_i(\Lambda)-
dimHC_i(KQ_0)=\left\{\begin{array}{cc}

    n\hspace{0.25cm} \text{if i is even} \\

    0\hspace{0.25cm} \text{if is odd} \\

    \end{array}\right.$$ From this the result follows using the
    foregoing proposition.

\end{proof}

\begin{rem} \rm
In \cite{Eu}[Section 7.5] the author calculates the reduced cyclic
homology spaces $\overline{HC}_i(\Lambda )$ using Connes' sequence
(see Proposition \ref{prop.cyclic homology}(2)) and, as a byproduct,
he also calculates the absolute cyclic homology spaces. However, he
states that the equality $HC_i(\Lambda )=\overline{HC}_i(\Lambda )$
holds, for all $i>0$. This is not true since
$\overline{HC}_i(\Lambda )=\frac{HC_i(\Lambda )}{HC_i(\Lambda_0)}$,
for all $i>0$. Therefore the description of the $HC_i(\Lambda )$ in
\cite{Eu}[p. 22] is not correct.
\end{rem}

\begin{rem} \rm Due to the fact that $\Lambda$ is a $\Lambda^e$-
$Z(\Lambda)$-bimodule, for each $\Lambda$-bimodule $M$, the
$K$-vector space $Hom_{\Lambda^e}(M,\Lambda)$ inherites a structure
of $Z(\Lambda)$-module. In particular, via the isomorphisms,

$$\oplus_{i\in Q_0} e_i\Lambda e_i\stackrel{\sim}{\longrightarrow} Hom_{\Lambda^e}(P, \Lambda)$$

$$\oplus_{a\in Q_1}e_{i(a)}\Lambda e_{t(a)}\stackrel{\sim}{\longrightarrow}Hom_{\Lambda^e}(Q, \Lambda)$$both $\oplus_{i\in Q_0}e_i\Lambda e_i$ and $\oplus_{a\in
Q_1}e_{i(a)}\Lambda e_{t(a)}$ have a structure of
$Z(\Lambda)$-modules. It is routinary to see that these structures
are given by the multiplication in $\Lambda$ and that the
differentials of the complex $V^{\bullet}$ in Proposition
\ref{cohomology complex} are all morphisms of $Z(\Lambda)$-modules.

\end{rem}

\begin{lema}\label{HH^i as $Z(Lambda)-module$}We view $Soc(\Lambda)$ as an ideal of $Z(\Lambda)$. The
following assertions hold.

\begin{enumerate}[1)]

    \item $Soc(\Lambda)HH^j(\Lambda)=0$ for all $j>0$.

    \item $HH^j(\Lambda)$ is a semisimple $Z(\Lambda)$-module for
    all $j\equiv 2,3$ $(\text{mod}6)$

    \item $HH^j(\Lambda)$ is isomorphic to
    $\frac{Z(\Lambda)}{Soc(\Lambda)}$ as a $Z(\Lambda)$-module for
    all $j>0$, $j\not\equiv 2, 3$ $(\text{mod}6)$

\end{enumerate}

\end{lema}

\begin{proof}

\begin{enumerate}[1)]

    \item Is a direct consequence of the fact that $\mathcal{P}(\Lambda, \Lambda)=
    Z(\Lambda)$ and $HH^j(\Lambda)\cong \underline{Hom}_{\Lambda^e}(\Omega_{\Lambda^e}^j
    (\Lambda), \Lambda)$ for all $j>0$.

    \item If $x_0= \sum_{i=0}^{n-1}(-1)^ia_i\bar{a}_i$ as in lemma
    \ref{Center of lambda}, then $x_0HH^j(\Lambda)= x_0 \cdot (\oplus \frac{e_i\Lambda
    e_i}{e_iJe_i})=0$ when $j\equiv 2$ $(mod6)$ and $x_0 HH^j(\Lambda)= x_0\cdot
    Soc(\Lambda)=0$ when $j\equiv 3$ $(mod 6)$.

    \item We clearly have an isomorphism $HH^j(\Lambda)\cong
    HH^{j+6}(\Lambda)$ for all $j>0$, so we only need to prove the
    claim for $j=1, 4, 5, 6$.

    For $j=6$, we take $h= 1+ Im(k_{\tau}^*)\in \frac{Ker(\delta^*)}{Im(k_{\tau}^*)}=
    \frac{Z(\Lambda)}{Soc(\Lambda)}$ and one clearly has that $Z(\Lambda)h= \frac{Z(\Lambda)}{Soc(\Lambda)}= HH^6(\Lambda)$

    For $j=1$ we take the element $\widehat{y}= \sum_{a\in Q_1}a\in \oplus_{a\in Q_1}e_{i(a)}\Lambda
    e_{t(a)}$. One routinary sees that $R^*(\widehat{y})=0$. We then
    get an element $y= \widehat{y} + Im(\delta^*)\in HH^1(\Lambda)=
    \frac{Ker(R^*)}{Im(\delta^*)}$.

    We now take the induced morphism of $Z(\Lambda)$-modules

    $$\frac{K[x_0]}{(x_0)^n}\cong \frac{Z(\Lambda)}{Soc(\Lambda)}\longrightarrow Z(\Lambda)y$$

    \hspace{6.9cm} $\bar{x}$ \hspace{0.5cm}$\rightsquigarrow$ \hspace{0.4cm} $xy$

Its kernel is an ideal of $\frac{K[x_0]}{(x_0^n)}$, then it is of
the form $\frac{(x_0^k)}{(x_0^n)}$, for some $k\leq n$.

We claim that if $k<n$ then $x_0^ky\neq 0$. That will imply that
$\frac{K[x_0]}{(x_0^n)}\cong Z(\Lambda)y$ so that $Z(\Lambda)y=
HH^1(\Lambda)$ by a dimension argument.

Suppose that $k<n$ and $yx_0^k=0$. Then $\widehat{y}x_0^k\in
Im(\delta^*)$. Since $\delta^*$ is a graded map of degree $1$ (with
respect to length degree) we will have an element $x=
\sum_{i=1}^{n-k}\mu_ib_i^{2k}$ of length-degree $2k$ in
$\oplus_{i\in Q_0}e_i\Lambda e_i$ such that $\delta^*(x)=
\widehat{y}x_0^k$. Since this element belongs to $\oplus_{a\in
Q_1}e_{i(a)}\Lambda e_{t(a)}$ we can look at its
$\epsilon$-component:

$$\delta^*(x)_{\epsilon}= \epsilon x - \epsilon x= \mu_1\epsilon
b_1^{2k} - \mu_1b_1^{2k}\epsilon=0$$

$$(\widehat{y}x_0^k)_{\epsilon}= \lambda_1 \epsilon
b_1^{2k},$$ where $\lambda_1$ is the coefficient of $b_i^{2k}$ in
the expression $x_0^k= \sum_{i=1}^{n-k}\lambda_i b_i^{2k}$. We know
from the proof of lemma \ref{Center of lambda} that $\lambda_1\neq
0$, which gives a contradiction since $\epsilon b_1^{2k}=
\epsilon^{2k+1}\neq 0$.

For $j=4$ we note that $R_{\tau}^*(e_1)= e_1\epsilon - \epsilon
e_1=0$ and that $\overline{\delta_{\tau}^*(c)}= \delta_{\tau}^*(c)$,
which implies that $\overline{\delta_{\tau}^*(x)}=
\delta_{\tau}^*(x)$ for $x\in \oplus_{i\in Q_0}e_i \Lambda e_i$. We
argue as in the previous paragraph and prove that if $x_0^k e_1\in
Im(\delta_{\tau}^*)$ and $k\leq n$ then $k=n$.

If $x_0^k e_1= \epsilon^{2k}\in Im(\delta_{\tau}^*)$ then there is
$1\leq r\leq k$ and $\mu_1, \dots, \mu_r\in K$, with $\mu_r\neq 0$,
such that $\delta_{\tau}^*(\sum_{i=1}^{r}\mu_i b_i^{2k-1})=
\epsilon^{2k}$. We look then at the $a_r$-component of both members
of the equality (i.e. at their iamge by applying the projection
$\oplus_{a\in Q_1}e_{i(a)}\Lambda e_{t(a)}\longrightarrow
e_{i(a_r)}\Lambda e_{t(a_r)}$). We then get $b_r^{2k-1}a_r=0$, which
is only possible in case $r=n$, and hence $k=n$.

As in the previous paragraph, we conclude that $HH^4(\Lambda)=
\frac{Ker(R_{\tau}^*)}{Im(\delta_{\tau}^*)}= Z(\Lambda)\gamma$,
where $\gamma:= e_1 + Im(\delta_{\tau}^*)$.

It only remains the case $j=5$. In this case we consider the element
$y\gamma\in HH^1(\Lambda)\cdot HH^4(\Lambda)\subseteq
HH^5(\Lambda)$. Via the isomorphism $\oplus_{a\in
Q_1}e_{i(a)}\Lambda
e_{t(a)}\stackrel{\sim}{\longrightarrow}Hom_{\Lambda^e}(\oplus_{a\in
Q_1}\Lambda e_{i(a)}\otimes e_{t(a)}, \Lambda)$, the ele\-ment $e_1$
is identified with the morphism of $\Lambda$-bimodules $\oplus_{a\in
Q_1}\Lambda e_{i(a)}\otimes
e_{t(a)}\Lambda\stackrel{\tilde{\gamma}}{\longrightarrow}\Lambda$
ma\-pping $e_{i(\epsilon)}\otimes e_{t(\epsilon)}\rightsquigarrow
e_1$ and $e_{i(a)}\otimes e_{t(a)}\rightsquigarrow 0$, for $a\neq
\epsilon$. It obviously lifts to the morphism $\oplus_{a\in
Q_1}\Lambda e_{i(a)}\otimes e_{t(a)}\Lambda
\stackrel{f}{\longrightarrow}\oplus_{i\in Q_0}\Lambda e_i\otimes
e_i\Lambda$ taking $e_{i(\epsilon)}\otimes
e_{t(\epsilon)}\rightsquigarrow e_1\otimes e_1$ and $e_{i(a)}\otimes
e_{t(a)}\rightsquigarrow 0$, for $a \neq\epsilon$. It is then
routinary to see that we have a commutative diagram

$$\begin{large}\xymatrix{\oplus_{i\in Q_0}\Lambda e_i \otimes e_i\Lambda
\ar[r]^{R_{\tau}} \ar[d]^{g} & \oplus_{a\in Q_1}\Lambda
e_{i(a)}\otimes e_{t(a)}\Lambda\ar[d]^{f}
\\ \oplus_{a\in Q_1}\Lambda e_{i(a)}\otimes e_{t(a)}\Lambda \ar[r]_{\delta} &
\oplus_{i\in Q_0} \Lambda e_i\otimes e_i\Lambda} \end{large}$$

where $g(e_1\otimes e_1)= e_{i(\epsilon)}\otimes e_{t(\epsilon)}$
and $g(e_i\otimes e_i)=0$ for $i\neq 1$.

On the other hand, via the isomorphism $\oplus_{a\in
Q_1}e_{i(a)}\Lambda
e_{t(a)}\stackrel{\sim}{\longrightarrow}Hom_{\Lambda^e}(\oplus_{a\in
Q_1}\Lambda e_{i(a)}\otimes e_{t(a)}\Lambda, \Lambda)$ the element
$\widehat{y}$ gets identified with the morphism of $\Lambda$-modules
$\oplus_{a\in Q_1}\Lambda e_{i(a)}\otimes e_{t(a)}\Lambda
\stackrel{\tilde{y}}{\longrightarrow}\Lambda$ such that
$\tilde{y}(e_{i(a)}\otimes e_{t(a)})=a$ $\forall a\in Q_1$. By
definition of the Yoneda product in $HH^*(\Lambda)$, the element
$y\gamma$ is represented by the morphism of $\Lambda$-bimodules
$\oplus_{i\in Q_0}\Lambda e_i\otimes
e_i\Lambda\longrightarrow\Lambda$ which takes $e_1\otimes e_1
\rightsquigarrow \epsilon$ and $e_i\otimes e_i \rightsquigarrow 0$,
for $i\neq 1$. Via the isomorphism $\oplus_{i\in Q_0}e_i \Lambda e_i
\stackrel{\sim}{\longrightarrow} Hom_{\Lambda^e}(\oplus_{i\in Q_0}
\Lambda e_i \otimes e_i \Lambda, \Lambda)$ $\tilde{y}\circ g$
corresponds to $\epsilon\in e_1\Lambda e_1\subseteq \oplus_{i\in
Q_0}e_i\Lambda e_i$.

We apply the argument already used for the cases $j=1$ and $j=4$ and
prove that if $x_0^k\epsilon \in Im(R_{\tau}^*)$ and $1\leq k\leq n$
then $k=n$.

That follows from the proof of lemma \ref{Image of RStar and
RtauStar}. Indeed, with the same notation, the matrix of the map
$R_{\tau}^*: V_{2m-2}^+ \longrightarrow W_{2m-1}$, with respect to
the bases of $V_{2m-2}^+$ and $W_{2m-1}$ given there is of the
form

$$\left(\begin{tabular}{ccccc}

            0 & $(-1)^{s_1}$ & 0 & $\ldots$ & 0 \\

            0 & -1 & $(-1)^{s_2}$ &  $\cdots$ & $\vdots$ \\

            0 & 0 & -1  & $$ &  0\\

            $\cdots$ & $\cdots$ & $\cdots$ & $\ddots$ & $(-1)^{s_m-1}$ \\

            0 & 0 & 0 & $\cdots$ & $(-1)^q$

       \end{tabular}\right)$$

which implies that no nonzero element of the form
$\left(\begin{tabular}{c} * \\ 0 \\ 0 \\ $\vdots$ \\ 0
\end{tabular}\right)$ can be in $Im(R_{\tau}^*)$. Therefore $x_0^k
\epsilon= \epsilon^{2k+1}\in Im(R_{\tau}^*)$ implies
$\epsilon^{2k+1}=0$ in $\Lambda$ and, hence, that $k=n$.

\end{enumerate}

\end{proof}

If
$$...P^{-i}\stackrel{d^{-i}}{\longrightarrow}P^{-i+1}\longrightarrow
...\longrightarrow
P^{-1}\stackrel{d^{-1}}{\longrightarrow}P^0\stackrel{u}{\longrightarrow}\Lambda\rightarrow
0$$ is the minimal projective resolution of $\Lambda$ (see
proposition \ref{projetive resolution}) then, by definition, we have
$HH^i(\Lambda
)=\frac{Ker((d^{-i-1})^*)}{Im((d^{-i})^*)}\subseteq\frac{(P^{-i})^*}{Im((d^{-i})^*)}$
for each $i>0$. Thus any element $\eta\in HH^i(\Lambda )$ is of the
form $\eta =\tilde{\eta}+Im((d^{-i})^*)$, for some $\tilde{\eta}\in
Hom_{\Lambda^e}(P^{-i},\Lambda )$ such that $\tilde{\eta}\circ
d^{-i-1}=0$. We will say that $\tilde{\eta}$ \emph{represents}
$\eta$ or that $\eta$ \emph{is represented by} $\tilde{\eta}$. The
following is a straightforward consequence of the results of this
section and their proofs.

\begin{prop}\label{K-basis for HH^*Lambda1} The following are bases
of the $HH^i(\Lambda)$, for $i=0,1,...,6$:

\begin{enumerate}

    \item For $HH^0(\Lambda)=Z(\Lambda )$: \hspace*{0.5cm} $\{x_0, x_0^2, \dots, x_0^{n-1}, x_1, \dots,
    x_n\}$, where $x_0= \sum_{i=1}^{n-1}(-1)^i a_i\bar{a}_i $ and $x_i=
    \omega_i$ for each $i=1, \dots n$.

    \item For $HH^1(\Lambda)=\frac{Ker(R^*)}{Im(\delta^*)}$: \hspace*{0.5cm} $\{y, x_0y, x_0^2 y, \dots
    x_0^{n-1}y\}$, where $y=\sum_{a\in Q_1}a+Im(\delta^*)$.

    The
    element $y$ is represented by the only morphism $\tilde{y}: Q\longrightarrow
    \Lambda$ such that  $\tilde{y}(e_{i(a)}\otimes e_{t(a)})=a$ for each $a\in
    Q_1$.

    \item For $HH^2(\Lambda )=\frac{Ker(k^*)}{Im(R^*)}$:
    \hspace*{0.5cm} $\{z_1, \dots, z_n\}$, where $z_k=e_k+Im(R^*)$
    for each $k\in Q_0$.

    The element $z_k$ is represented by the only morphism $\tilde{z_k}:P\longrightarrow
    \Lambda$ such that $\tilde{z}_k(e_i\otimes e_i)=\delta_{ik}e_k$.

    \item For $HH^3(\Lambda )=\frac{Ker(\delta_\tau^*)}{Im(k^*)}=Ker(\delta_\tau^*)$:
    \hspace*{0.5cm} $\{t_1, \dots, t_n\}$, where $t_k=w_k$ for each
    $k\in Q_0$.

    The element $t_k$ is represented by the only morphism $\tilde{t}_k: P\longrightarrow
    \Lambda$ such that $\tilde{t}_k(e_i\otimes e_i)=
    \delta_{ki}\omega_k$.

    \item For $HH^4(\Lambda )=\frac{Ker(R_\tau^*)}{Im(\delta_\tau^*)}$:
    \hspace*{0.5cm} $\{\gamma, x_0\gamma, \dots,
    x_0^{n-1}\gamma\}$, where $\gamma =e_1+Im(\delta_\tau^*)$.

    The element $\gamma$ is represented by the only morphism $\tilde{\gamma}: Q\longrightarrow
    \Lambda$ such that $\tilde{\gamma}(e_{i(a)}\otimes e_{t(a)})=\delta_{\epsilon
    a}e_1$, for each $a\in Q_1$.

    \item For $HH^5(\Lambda )=\frac{Ker(k_\tau^*)}{Im(R_\tau^*)}$:
    \hspace*{0.5cm} $\{y\gamma, x_0y\gamma, \dots
    x_0^{n-1}y\gamma\}$.

    \item For $HH^6(\Lambda )=\frac{Ker(\delta^*)}{Im(k_\tau^*)}$:
    \hspace*{0.5cm} $\{h, x_0h, \dots,
    x_0^{n-1}h\}$, where $h=1+Im(k_\tau^*)$.

    The element $h$ is represented by the multiplication map $\tilde{h}=u:\oplus_{i\in Q_0}\Lambda e_i\otimes
    e_i\Lambda\longrightarrow\Lambda$.

\end{enumerate}

\end{prop}

The bases of the $HH^i(\Lambda )$ given in the above proposition
will be called \emph{canonical bases}.

\begin{rem} \rm
In \cite{Eu} the author uses the length grading on $\Lambda$ and
looks at the minimal projective resolution of $\Lambda$ as one in
the category of graded  $\Lambda$-bimodules. With that in mind the
Hochschild homology and cohomology spaces  become graded vector
spaces. Then he calculates this graded structure in terms of three
seminal graded vector spaces $R$, $U$ and $K$ (see Theorems 4.0.13
and  4.0.14  in \cite{Eu}). In our terminology, $R=KQ_0$
(concentrated in degree $0$), $U=\frac{Z(\Lambda )}{Soc (\Lambda
)}[2]$ (with the length grading on $\frac{Z(\Lambda )}{Soc (\Lambda
)}$) and $K=HH^2(\Lambda )[2]$ (which is concentrated in degree $0$
since $HH^2(\Lambda )$ is concentrated in degree $-2$).

His strategy to prove the mentioned theorems  is based on the use of
Connes' exact sequence (see Proposition \ref{prop.cyclic
homology}(2)) and the fact that the differentials in this sequence
are graded maps. Then he is able to describe the graded structure of
each $HH_i(\Lambda )$ and, using dualities between the Hochschild
homology and cohomology graded spaces obtained in \cite{ES}, the
author also gets the graded structure of each $HH^i(\Lambda )$.

Due to the fact that the dimension of $R$, $U$ and $K$ is $n$, the
dimensions of the $HH_i(\Lambda )$ and the $HH^i(\Lambda )$ can be
read off from the mentioned theorems 4.0.13 and 4.0.14 of \cite{Eu},
even if they were not explicitly stated in a proposition or
corollary. After that and before calculating the ring structure of
$HH^*(\Lambda )$, Eu gives explicit bases of the $HH^i(\Lambda )$
using the correspondent of our  complex $V^\bullet$, which he
considers as a complex of graded vector spaces (see section 8 in
\cite{Eu}).

In our case, we have not used the graded condition of the minimal
projective resolution of $\Lambda$. Instead, we have  calculated the
dimensions of the $HH^i(\Lambda )$ by directly manipulating the
complex $V^\bullet$ and using the isomorphisms of Corollary
\ref{structure of homology as graded module}. The bases of the
$HH^i(\Lambda )$ have been obtained in the process of identifying
the structure of these spaces as $Z(\Lambda )$-modules, using the
process to calculate already some of the products in $HH^*(\Lambda
)$ (see the proof of Lemma \ref{HH^i as $Z(Lambda)-module$}).
\end{rem}

\section{The ring structure of $HH^*(\Lambda)$ and proof of the main \newline theorems}

We start by studying the map $\phi_y: HH^2(\Lambda)\longrightarrow
HH^3(\Lambda)$  given by $\phi_y(u)= yu$ for all $u\in
HH^2(\Lambda)$

\begin{lema}\label{HH*1xHH*2} If $C=(C_{kj})$ is the matrix of
$\phi_y$ with respect to the canonical bases of $HH^2(\Lambda)$ and
$HH^3(\Lambda)$, then the following conditions hold:

\begin{enumerate}[1)]

    \item $C$ is a symmetric integer matrix.

    \item $C_{jk}=(-1)^{k-j+1}(2j-1)(n-k+1)$ whenever $1\leq j\leq k\leq
    n$.

    \item $rank(C)=n$, when $Char(K)$ does not divide $2n+1$, and
    $rank(C)=1$, when $Char(K)$ divides $2n+1$.

\end{enumerate}

\end{lema}

\begin{proof} Let $x= \alpha_1\cdots \alpha_r$ ($r>0$) be any path
in $e_jKQ e_k$ which does not belong to the ideal $I$. We put

$$h_x= \alpha_1\cdots \alpha_{r-1}\otimes x^* + \alpha_1\cdots \alpha_{r-2}\otimes \alpha_r x^*
+\cdots + e_j\otimes \alpha_2\cdots \alpha_r x^*$$ which is an
element of $\oplus_{a\in Q_1}\Lambda e_{i(a)}\otimes
e_{t(a)}\Lambda$. In case $j=k$ we  put $h_{e_j}=0$ and if
$p_j=\bar{a}_{j-1}...\bar{a}_{1}\epsilon^{2(n-j)+1}a_1...a_{j-1}$ we
also define  $h_{w_j}=(-1)^{\frac{j(j-1)}{2}}h_{p_j}$ (recall that
$w_j=(-1)^{\frac{j(j-1)}{2}}p_j$). In this way we have defined $h_x$
for each $x\in e_jBe_k$ and for all $j,k\in Q_0$.

Direct calculation shows that $\delta(h_x)= x\otimes x^* -
e_j\otimes \omega_j$, and hence

$$\delta(\sum_{x\in e_jBe_k}(-1)^{deg(x)}h_x)= \sum_{x\in e_jBe_k}(-1)^{deg(x)}
(x\otimes x^*-e_j\otimes \omega_j)= \sum_{x\in
e_jBe_k}(-1)^{deg(x)}x\otimes x^*,$$ bearing in mind that in
$e_jBe_k$ there are exactly the same number of odd and even length
degree.

Now consider  $\tilde{z}_k:P\longrightarrow\Lambda$ as in
proposition \ref{K-basis for HH^*Lambda1}. It is clear that the
morphism of $\Lambda$-bimodules $\hat{z}_k:P\longrightarrow P$
determined by the rule $\hat{z}_k(e_i\otimes e_i)=
\delta_{ik}e_k\otimes e_k$ is a lifting of $\tilde{z}_k$ (i.e.
$\tilde{z}_k= u\circ \hat{z}_k$).

If now $f_k: P\longrightarrow Q$ is the morphism of
$\Lambda$-bimodules determined by the rule $f_k(e_j\otimes e_j)=
\sum_{x\in e_jBe_k}(-1)^{deg(x)}h_x$ then we have a commutative
diagram

$$\begin{large}\xymatrix{P\ar[r]^k \ar[d]_{f_k} & P\ar[d]^{\hat{z}_k}
\\ Q\ar[r]^{\delta} & P} \end{large}$$
and hence $yz_k$ is represented by the morphism

$$\tilde{y}\circ f_k: P\longrightarrow \Lambda\text{, }(e_j\otimes e_j\rightsquigarrow \sum_{x\in
e_jBe_k}(-1)^{deg(x)}deg(x)\omega_j).$$ That means that if we put
$C_{jk}= \sum_{x\in e_jBe_k}(-1)^{deg(x)}deg(x)$ for all $j, k\in
Q_0$, then we have $yz_k= \sum_{j\in Q_0} C_{jk}t_j$ (notation as in
proposition \ref{K-basis for HH^*Lambda1}). Therefore $C:=(C_{jk})$
is the matrix of $\phi_y: HH^2(\Lambda)\longrightarrow
HH^3(\Lambda)$ with respect to the canonicval bases of
$HH^2(\Lambda)$ and $HH^3(\Lambda)$.

That $C$ is a symmetric integer matrix is clear since the
anti-isomorphism $x\rightsquigarrow\bar{x}$ gives a bijection
between $e_jBe_k$ and $e_kBe_j$ which preserves the term
$(-1)^{deg(x)}deg(x)$. We then proceed to calculate the entries of
this matrix. To do that we should recall the possible degrees of
elements in $e_jBe_k$ (see Remark \ref{Basis properties}), for
$1\leq j\leq k\leq n$. There are two possibilities.

\begin{enumerate}[i)]

     \item $k\equiv j$ $(\text{mod}2)$: Then the sum of odd degrees is $\frac{[(k-j)+ (k-j)
     +2(n-k)](n-k+1)}{2}= (n-j)(n-k+1)$, while the sum of even
     degree is $\frac{[(k+j-1)+ (k+j-1)+2(n-k)](n-k+1)}{2}=
     (n+j-1)(n-k+1)$. Therefore we have $C_{jk}= (n-j)(n-k+1)- (n+j-1)(n-k+1)= (1-2j)(n-k+1)$

     \item $k\not\equiv j$ $(\text{mod}2)$: In this case $C_{jk}$ is the
     negative of the number above, i.e., $C_{jk}=(2j-1)(n-k+1)$.

\end{enumerate}

It finally remains to calculate $rank(C)$. We view each $n\times
n$ matrix as a $n$-tuple whose components are its rows. By
elementary row transformation one passes from $C=(C_1, \dots,
C_n)$ to
$$C'=(C_1, C_2+3C_1, \dots, C_j+(-1)^j(2j-1)C_1, \dots, C_n +
(-1)^n(2n-1)C_1)$$ so that $rank(C)= rank(C')$. We look at the
$j$-th row $C_j'= C_j + (-1)^j(2j-1)C_1$ of $C'$. It is
straightforward, and thus is left to the reader, to check that for
$j\leq k$, then $C_{jk}'=0$ and if $j>k$, $C_{jk}'=
(-1)^{j-k+1}(k-j)(2n+1)$.

Therefore, in case $Char(K)$ divides $2n+1$, all rows of $C'$ except
the first one are zero. On the other hand, we have $C_{1n}'= C_{1n}=
(-1)^{n-1+1}(2\cdot 1- 1)(n-n+1)=(-1)^n$. It follows that
$rank(C)=1$ in case $Char(K) / 2n+1$. In case $Char(K)$ does not
divide $2n+1$, if we apply the $n$-cycle $(1 \ n \ n-1 \cdots \ 2)$
to the rows of $C'$ we obtain a lower triangular matrix with
diagonal entries $C_{21}', C_{32}', \dots, C_{n,n-1}', C_{1n}'$.

We have $C_{k+1,k}'= (-1)^{(k+1)-k+1}(k-(k+1))(2n+1)= -(2n+1)$ for
$k=2, \dots, n$ and $C_{1n}'=(-1)^n$. It follows that $det(C)=
det(C')=(-1)^{2n-1}(2n+1)^{n-1}\neq 0$. Therefore $rank(C)=n$ in
this case.

\end{proof}

\begin{rem} \rm
Given a graph $\Gamma$ without
 double edges, its \emph{adjacency
matrix} $D=D_\Gamma$ is the symmetric matrix
$D=(d_{ij})_{i,j\in\Gamma_0}$ having $d_{ij}=1$, in case there is an
edge $i --- j$, and $d_{ij}=0$ otherwise. In particular, for the
graph $\mathbb{L}_n$, one  has $d_{11}=1$, $d_{i,i+1}=d_{i+1,i}=1$
for $i=1,...,n-1$, and $d_{ij}=0$ otherwise. Direct computation
shows that the matrix $C$ of Lemma \ref{HH*1xHH*2} satisfies the
equality \hspace*{0.5cm} $-C(2I_n+D)=(2n+1)I_n$, where $I_n$ is the
identity $n\times n$ matrix. Therefore, when $char (K)$ does not
divide $2n+1$, an alternative description of the matrix $C$ is
$C=-(2n+1)(2I_n+D)^{-1}$. Up to signs forced by the different
presentation of $\Lambda$ and the different choice of the
exceptional vertex of $\mathbb{L}_n$, the last equality is that of
\cite{Eu}[Proposition 9.3.1] (see also \cite{Eu}[Theorem 4.0.16]).

Taking into account also the case when $char(K)$ divides $2n+1$ is
fundamental for the difference of presentations in our two main
theorems  and is the part of our work where the arguments of
\cite{Eu} cannot be applied.
\end{rem}

\begin{lema} \label{HH^2 x HH^4} The equality

$$z_j\gamma= (-1)^{j}(n-j+1)x_0^{n-1}h$$ holds in the ring $HH^*(\Lambda)$, for each $j=1, 2, \dots, n$.

\end{lema}

\begin{proof} Let $\tilde{\gamma}$ be as in proposition \ref{K-basis for HH^*Lambda1}. It is clear that
the morphism of $\Lambda$-bimodules $\widehat{\gamma}: \oplus_{a\in
Q_1}\Lambda e_{i(a)}\otimes e_{t(a)}\Lambda \longrightarrow
\oplus_{i\in Q_0}\Lambda e_i\otimes e_i\Lambda$ which maps
$e_{i(a)}\otimes e_{t(a)}\rightsquigarrow \delta_{\epsilon
a}e_1\otimes e_1$ is a lifting of $\tilde{\gamma}$ (i.e. $u\circ
\widehat{\gamma}= \tilde{\gamma}$).

If we take now the morphism $\beta: \oplus_{i\in Q_0}\Lambda
e_i\otimes e_i \Lambda\longrightarrow \oplus_{a\in Q_1}\Lambda
e_{i(a}\otimes e_{t(a)}\Lambda$ determined by the rule
$\beta(e_i\otimes e_i)= \delta_{i1}e_{i(\epsilon)}\otimes
e_{t(\epsilon)}$, then direct computation shows that
$\widehat{\gamma}\circ R_{\tau}= \delta \circ \beta$. Our aim is now
to define a morphism $\alpha: \oplus_{i\in Q_0}\Lambda e_i \otimes
e_i \Lambda\longrightarrow \oplus_{i\in Q_0}\Lambda e_i \otimes e_i
\Lambda$ completing commutatively the following diagram

$$\xymatrix{P\ar[r]^{k_{\tau}} \ar[d]_{\alpha} & P \ar[r]^{R_{\tau}}
\ar[d]_{\beta} & Q \ar[d]_{\widehat{\gamma}} \\ P \ar[r]_{R} & Q
\ar[r]_{\delta} & P}$$

Once our goal is attained, the composition

$$\oplus_{i\in Q_0}\Lambda e_i \otimes e_i \Lambda
\stackrel{\alpha}{\longrightarrow} \oplus_{i\in Q_0}\Lambda e_i
\otimes e_i \Lambda \stackrel{\tilde{z}_j}{\longrightarrow}
\Lambda$$ will represent+ the element $z_j\gamma \in HH^6(\Lambda)$.

For each vertex $i\in Q_0$ and for each integer $r=1, 2, \dots,
n-i+1$, we put

$$y_r^i= \sum_{k=0}^{n-i-r+1}\bar{a}_{i-1}\cdots \bar{a}_1\epsilon^{2k}a_1
\cdots a_{r-1}\otimes \bar{a}_{r-1}\cdots \bar{a}_1
\epsilon^{2(n-i-r-k+1)}a_1\cdots a_{i-1},$$ which is an element of
$e_i\Lambda e_r\otimes e_r \Lambda e_i$. Given
$x=\bar{a}_{i-1}...\bar{a}_1\epsilon^k\in e_iBe_1$ (convening that
$\bar{a}_{i-1}...\bar{a}_1=e_1$ when $i=1$), we put $x^\dagger
=\epsilon^{2(n-i)+1-k}a_1...a_{i-1}$. Then we have

$$xx^\dagger
=\bar{a}_{i-1}...\bar{a}_1\epsilon^{2(n-i)+1}a_1...a_{i-1}=(-1)^{\frac{i(i-1)}{2}}w_i,$$for
each $i\in Q_0$. Note that this implies that $x^\dagger
=(-1)^{\frac{i(i-1)}{2}}x^*$.

We claim that if $(u_n)_{n\geq 1}$ is the sequence of integers
$0,0,1,1,0,0,1,1, \dots$, then the following equality holds

$$R(\sum_{r=1}^{n-i+1}(-1)^{u_r}y_r^i)=\sum_{x\in e_iBe_1}x\otimes
x^\dagger.$$ From this equality, by direct computation the following
will follow:
$$R((-1)^{1+\frac{i(i-1)}{2}}\sum_{r=1}^{n-i+1}(-1)^{u_r}y_r^i)= - \sum_{x\in e_iBe_1} x\otimes x^*.$$
As a consecquence,  the morphism of $\Lambda$-bimodules $\alpha:
\oplus_{i\in Q_0}\Lambda e_i \otimes e_i \Lambda\longrightarrow
\oplus_{i\in Q_0} \Lambda e_i \otimes e_i \Lambda$, determined by
the rule $\alpha(e_i\otimes e_i)= (-1)^{1+\frac{i(i-1)}{2}}
\sum_{r=1}^{n-i+1}(-1)^{u_r}y_r^i$, will satisfy the desired
equality $\beta \circ k_{\tau}= R\circ \alpha$, because $(\beta
\circ k_{\tau})(e_i\otimes e_i)= -\sum_{x\in e_iBe_1}x\otimes x^*$
for all $i\in Q_0$.

To settle our claim we shall prove by induction on $s=1, 2, \dots
n-i+1$ the equality

$$ \sum_{r=1}^s (-1)^{u_r} R(y_r^i)- \sum_{x\in e_iBe_1}x\otimes x^\dagger=$$
$$(-1)^{u_s}\sum_{k=0}^{n-i-s+1}
\bar{a}_{i-1} \cdots \bar{a}_1 \epsilon^{2k}a_1 \cdots
a_{s-1}(a_s\otimes e_s + e_s\otimes \bar{a}_s)\bar{a}_{s-1}\cdots
\bar{a}_1 \epsilon^{2(n-i-k-s+1)}a_1\cdots a_{i-1}$$

The equality, when taken for $s= n-i+1$, will give:

$$\sum_{r=1}^{n-i+1}(-1)^{u_r} R(y_r^i)- \sum_{x\in e_iBe_1}x\otimes
x^\dagger=(-1)^{u_{n-i+1}} (\bar{a}_{i-1}\cdots \bar{a}_1
\epsilon^{2\cdot 0}a_1\cdots a_{n-i}a_{n-i+1}\otimes$$

$$\bar{a}_{n-i}\cdots \bar{a}_1\epsilon^{2\cdot 0}a_1\cdots
a_{i-1} + \bar{a}_{i-1}\cdots \bar{a}_1 \epsilon^{2\cdot 0}a_1\cdots
a_{n-i}\otimes \bar{a}_{n-i+1}\bar{a}_{n-i}\cdots \bar{a}_1
\epsilon^{2\cdot 0}a_1\cdots a_{i-1})$$ But the second member of
this equality is zero (see remark \ref{Basis properties}(2)) so that
our claim will be settled.

For $s=1$, we have

$$R(y_1^i)- \sum_{x\in e_iBe_1}x\otimes
x^\dagger=$$
$$\sum_{k=0}^{n-i} \bar{a}_{i-1}
\cdots \bar{a}_1 \epsilon^{2k}(\epsilon\otimes e_1 + e_1\otimes
\epsilon + a_1\otimes e_1 + e_1 \otimes
\bar{a}_1)\epsilon^{2(n-i-k)}a_1\cdots a_{i-1}-$$

$$\sum_{t=0}^{2(n-i)+1} \bar{a}_{i-1} \cdots \bar{a}_1
\epsilon^{t}\otimes \epsilon^{2(n-i)+1 -t)}a_1\cdots a_{i-1})=$$

$$ \sum_{k=0}^{n-i}
\bar{a}_{i-1} \cdots \bar{a}_1 \epsilon^{2k}(a_1\otimes e_1 +
e_1\otimes \bar{a}_1) \epsilon^{2(n-i-k)}a_1\cdots a_{i-1}$$ and the
desired equality is true in this case.

If now $s>1$ and the equality holds for $s-1$ then we get

$$\sum_{r=1}^{s}(-1)^{u_r} R(y_r^i)- \sum_{x\in e_iBe_1}x\otimes
x^\dagger=$$

$$(-1)^{u_s} R(y_s^i)+ (-1)^{u_{s-1}}
\sum_{k=0}^{n-i-s+2} \bar{a}_{i-1} \cdots \bar{a}_1 \epsilon^{2k}a_1
\cdots a_{s-2}(a_{s-1}\otimes e_{s-1} +$$

$$e_{s-1}\otimes \bar{a}_{s-1})\bar{a}_{s-2}\cdots \bar{a}_1
\epsilon^{2(n-i-k-s+2)}a_1\cdots a_{i-1})=$$

$$(-1)^{u_s}\sum_{k=0}^{n-i-s+1}
\bar{a}_{i-1} \cdots \bar{a}_1 \epsilon^{2k}a_1 \cdots
a_{s-1}(a_{s}\otimes e_{s} + e_{s}\otimes \bar{a}_{s} +$$

$$\bar{a}_{s-1}\otimes e_s + e_s\otimes a_{s-1})\bar{a}_{s-1}\cdots
\bar{a}_1 \epsilon^{2(n-i-k-s+1)}a_1\cdots a_{i-1} +$$

$$(-1)^{u_{s-1}}\sum_{k=0}^{n-i-s+2}
\bar{a}_{i-1} \cdots \bar{a}_1 \epsilon^{2k}a_1 \cdots
a_{s-2}(a_{s-1}\otimes e_{s-1} + e_{s-1}\otimes
\bar{a}_{s-1})\bar{a}_{s-2}\cdots \bar{a}_1
\epsilon^{2(n-i-k-s+2)}a_1\cdots a_{i-1}$$ Using the equality $(4)$
in remark \ref{Basis properties} with $j=s$ and the one obtained
from it by applying the canonical anti-isomorphism $(\bar{-})$, we
can rewrite the second member of the last expression as

$$(-1)^{u_s}\sum_{k=0}^{n-i-s+1}
\bar{a}_{i-1} \cdots \bar{a}_1 \epsilon^{2k}a_1 \cdots
a_{s-1}(a_{s}\otimes e_{s} + e_{s}\otimes
\bar{a}_{s})\bar{a}_{s-1}\cdots \bar{a}_1
\epsilon^{2(n-i-k-s+1)}a_1\cdots a_{i-1} +$$

$$(-1)^{u_s} (-1)^{s-1}\sum_{k=0}^{n-i-s+1}
(\bar{a}_{i-1} \cdots \bar{a}_1 \epsilon^{2(k+1)}a_1 \cdots
a_{s-2}\otimes \bar{a}_{s-1}\cdots \bar{a}_1
\epsilon^{2(n-i-k-s+1)}a_1\cdots a_{i-1} +$$

$$\bar{a}_{i-1} \cdots \bar{a}_1 \epsilon^{2(k)}a_1 \cdots
a_{s-1}\otimes \bar{a}_{s-2}\cdots \bar{a}_1
\epsilon^{2(n-i-k-s+2)}a_1\cdots a_{i-1}) +$$

$$(-1)^{u_{s-1}}\sum_{k=0}^{n-i-s+2}
\bar{a}_{i-1} \cdots \bar{a}_1 \epsilon^{2k}a_1 \cdots
a_{s-2}(a_{s-1}\otimes e_{s-1} + e_{s-1}\otimes
\bar{a}_{s-1})\bar{a}_{s-2}\cdots \bar{a}_1
\epsilon^{2(n-i-k-s+2)}a_1\cdots a_{i-1}$$ Except two of them, the
subsummands of the second and third summands of this expression
cancel by pairs due to the fact that $u_s + s - 1 - u_{s-1}$ is
always an odd integer. The two subsummands which do not cancel are

$$(-1)^{u_{s-1}}
\bar{a}_{i-1} \cdots \bar{a}_1 a_1 \cdots a_{s-2}\otimes
\bar{a}_{s-1}\bar{a}_{s-2}\cdots \bar{a}_1
\epsilon^{2(n-i-s+2)}a_1\cdots a_{i-1}$$ and

$$(-1)^{u_{s-1}}
\bar{a}_{i-1} \cdots \bar{a}_1 \epsilon^{2(n-i-s+2)}a_1 \cdots
a_{s-2}a_{s-1}\otimes \bar{a}_{s-2}\cdots \bar{a}_1 a_1\cdots
a_{i-1}$$ But these two subsummands are zero due to $(2)$ in remark
\ref{Basis properties} and, hence, our claim is settled.

We are finally ready to end the proof. The element $z_j\gamma$ is
represented by $\tilde{z}_j \circ \alpha$, where $\alpha:
\oplus_{i\in Q_0} \Lambda e_i\otimes e_i\Lambda\longrightarrow
\oplus_{i\in Q_0}\Lambda e_i \otimes e_i\Lambda$ is the morphism
determined by the rule $\alpha(e_i\otimes e_i)=
(-1)^{1+\frac{i(i-1)}{2}}\sum_{r=1}^{n-i+1}(-1)^{u_r}y_r^i$, for
each $i\in Q_0$. We get:

$$(\tilde{z}_j\circ \alpha)(e_i\otimes e_i)= (-1)^{1+\frac{i(i-1)}{2}} \sum_{r=1}^{n-i+1}
(-1)^{u_r}\tilde{z}_j (y_r^i).$$ This expression is zero in case $j>
n-i+1$, while it is equal to
$(-1)^{1+\frac{i(i-1)}{2}}(-1)^{u_j}\tilde{z}_j (y_j^i)$ otherwise.
But, using remark \ref{Basis properties}(3), we have an equality

$$(-1)^{1+\frac{i(i-1)}{2}}(-1)^{u_j}\tilde{z}_j (y_j^i)$$
$$=(-1)^{1+\frac{i(i-1)}{2}}(-1)^{u_{j}}\sum_{k=0}^{n-i-j+1}
\bar{a}_{i-1} \cdots \bar{a}_1 \epsilon^{2k}a_1 \cdots
a_{j-1}\bar{a}_{j-1}\cdots \bar{a}_1
\epsilon^{2(n-i-j-k+1)}a_1\cdots a_{i-1})=$$

$$(-1)^{1+\frac{i(i-1)}{2}}(-1)^{u_j} (-1)^{\frac{j(j-1)}{2}}(n-i-j+2)
\bar{a}_{i-1} \cdots \bar{a}_1 \epsilon^{2(n-i)}a_1 \cdots
a_{i-1}$$in case $1\leq j\leq n-i+1$. But this is zero unless $i=1$
(see $(2)$ in remark \ref{Basis properties}).

In case $i=1$, we have $(\tilde{z}_j \circ \alpha)(e_1\otimes e_1)=
(-1)^{1+ u_j  +\frac{j(j-1)}{2}} (n-j+1)\epsilon^{2(n-1)}$.

It is left as an exercise to see that $1+ u_j +
\frac{j(j-1)}{2}\equiv j$ $(\text{mod}2)$ and the equality
$z_j\gamma= (-1)^{j}(n-j+1)x_0^{n-1}h$ follows since $x_0^{n-1}h$ is
the element
$\epsilon^{2n-2}+Im(k_\tau^*)\in\frac{Ker(\delta^*)}{Im(k_\tau^*)}=HH^6(\Lambda)$.

\end{proof}

In order to know the multiplicative structure of $HH^*(\Lambda)$,
the role of $HH^3(\Lambda)$ is fundamental. The following is a
great help.

\begin{lema}\label{products involving HH^3} The following assertions
hold:

\begin{enumerate}[1)]

    \item $HH^3(\Lambda)\cdot HH^{2m+1}(\Lambda)=0$, for all
    integers $m\geq 0$.

    \item $z_kt_j= \delta_{jk}x_0^{n-1}y\gamma$, for all $i, j\in
    Q_0$.

    \item $t_j\gamma= \delta_{1j}x_0^{n-1}yh$, for all $j\in Q_0$.

\end{enumerate}

\end{lema}

\begin{proof} From the proof of lemma \ref{HH^i as
$Z(Lambda)-module$} we know that multiplication by $h$ gives an
isomorphism of $Z(\Lambda)$-modules $\phi_h:
HH^k(\Lambda)\stackrel{\sim}{\longrightarrow}HH^{k+6}(\Lambda)$, for
all $k>0$, and multiplication by $y$ yields another one $\phi_y:
HH^4(\Lambda)\stackrel{\sim}{\longrightarrow}HH^5(\Lambda)$. In
particular, we get $HH^{k+6}(\Lambda)= HH^k(\Lambda)\cdot
HH^6(\Lambda)$, for all $k>0$, and $HH^1(\Lambda)\cdot HH^4(\Lambda
)=HH^5(\Lambda)$. From these considerations we deduce that in order
to prove assertion $1$, we just need to check that the products
$HH^1(\Lambda) \cdot HH^3(\Lambda)$ and $HH^3(\Lambda) \cdot
HH^3(\Lambda)$ are both zero.

We will now prove the three assertions of the lemma by considering
the following commutative diagram, for each $j\in Q_0$:

$$\begin{large}\xymatrix{Q \ar[r]^{\delta} \ar[d]^{l_j} & P \ar[r]^{k_{\tau}} \ar[d]^{h_j}
 & P \ar[r]^{R_{\tau}} \ar[d]^{g_j} & Q \ar[r]^{\delta_{\tau}} \ar[d]^{f_j} & P \ar[d]^{\widehat{t}_j} \\
 Q \ar[r]_{\delta_{\tau}} & P \ar[r]_k & P \ar[r]_R & Q \ar[r]_{\delta} & P }\end{large}$$ where the vertical arrows are the only morphisms of
$\Lambda$-bimodules satisfying the following properties:

\begin{enumerate}[a)]

    \item $\widehat{t}_j(e_i\otimes e_i)= \frac{1}{2} \delta_{ij}(\omega_j\otimes e_j + e_j\otimes \omega_j)$

    \item In case $(j,a)\neq (1, \epsilon)$ we have

    $$f_j(e_{i(a)}\otimes e_{t(a)})=\left\{\begin{array}{ccc}

    0\hspace{0.25cm} \text{if} \hspace{0.25cm} j\not \in \{i(a),t(a)\}\\

    \\

    \frac{1}{2}e_{i(a)}\otimes \omega_{t(a)}\hspace{0.25cm} \text{if} \hspace{0.25cm}j=t(a) \\

    \\

    -\frac{1}{2}\omega_{i(a)}\otimes e_{t(a)}\hspace{0.25cm} \text{if}\hspace{0.25cm} j=i(a)\\

    \end{array}\right.$$

    and, for $(j,a)=(1, \epsilon)$, we have:

    $$f_1(e_{i(\epsilon)}\otimes e_{t(\epsilon)})= \frac{1}{2} (e_{i(\epsilon)}\otimes
    \omega_{t(\epsilon)} - \omega_{i(\epsilon)}\otimes e_{t(\epsilon)})$$

    \item $g_j= \widehat{t}_j$

    \item $h_j(e_i\otimes e_i)= \frac{1}{2}\delta_{ij}(e_j\otimes \omega_j - \omega_j\otimes e_j)$

    \item In case $(j,a)\neq (1, \epsilon)$ we have

    $$l_j(e_{i(a)}\otimes e_{t(a)})=\left\{\begin{array}{ccc}

    0\hspace{0.25cm} \text{if} \hspace{0.25cm} j\not \in \{i(a),t(a)\}\\

    \\

    \frac{1}{2}e_{i(a)}\otimes \omega_{t(a)}\hspace{0.25cm} \text{if} \hspace{0.25cm}j=t(a) \\

    \\

    \frac{1}{2}\omega_{i(a)}\otimes e_{t(a)}\hspace{0.25cm} \text{if}\hspace{0.25cm} j=i(a)\\

    \end{array}\right.$$

    and, in case $(j,a)=(1, \epsilon)$, we have:

    $$l_1(e_{i(\epsilon)}\otimes e_{t(\epsilon)})= \frac{1}{2}(e_{i(\epsilon )}\otimes
    \omega_{t(\epsilon)} + \omega_{i(\epsilon)}\otimes e_{t(\epsilon)})$$

    \end{enumerate}

    Although tedious, checking that these morphisms make commutative
    the last diagram is easy and routinary. It is left to the
    reader.

    We know have:

    \vspace{0.5cm}

    \noindent 1) .i) $HH^1(\Lambda) \cdot HH^3(\Lambda)=0$

    \vspace{0.25cm}

        We already know that $\{y, x_0y, \dots, x_0^{n-1}y\}$ is a
        basis of $HH^1(\Lambda)$. So we only need to check that

        $yt_j=0$, for all $j\in Q_0$. But $yt_j$ is represented by
        the composition

        $$Q \stackrel{f_j}{\longrightarrow}Q\stackrel{\tilde{y}}{\longrightarrow}
        \Lambda .$$

        Due to the fact that $a\omega_{t(a)}=0=\omega_{i(a)}a$
        $\forall a\in Q_1$, one gets that $(\tilde{y}\circ f_j)(e_{i(a)}\otimes e_{t(a)})=0$

        $\forall a\in Q_1$, and hence $\tilde{y}\circ f_j=0$

        \begin{enumerate}[1)]

        \item .ii) $HH^3(\Lambda)\cdot HH^3(\Lambda)=0$

        The element $t_kt_j$ of $HH^6(\Lambda)$ is represented by
        the composition

        $$P\stackrel{h_j}{\longrightarrow}P \stackrel{\tilde{t}_k}{\longrightarrow}\Lambda$$

        We have $(\tilde{t}_k \circ h_j)(e_i\otimes e_i)= \tilde{t}_k (
        \frac{1}{2}\delta_{ij}(e_j\otimes \omega_j - \omega_j\otimes
        e_j))$. This is clearly zero for all $i,j,k \in Q_0$.

    \item $z_kt_j= \delta_{jk}x_0^{n-1}y\gamma$

    The element $z_kt_j\in HH^5(\Lambda)$ is represented by the
    composition

    $$P\stackrel{g_j}{\longrightarrow}P \stackrel{\tilde{z}_k}{\longrightarrow} \Lambda$$

    Due to the fact that $g_j(e_i\otimes e_i)=0$ for $i\neq j$ and
    $g_j(\Lambda e_j\otimes e_j\Lambda)\subseteq \Lambda e_j\otimes
    e_j\Lambda$, we readily see that $\tilde{z}_k\circ g_j=0$
    when $j\neq k$. Moreover, in case $j=k$, we have

    $$(\tilde{z}_j \circ g_j)(e_i\otimes e_i)= \frac{1}{2}\delta_{ij}(\omega_j + \omega_j)=
    \delta_{ij}\omega_j$$

    From remark \ref{w_j - w_j+1 belongs to Im(RTauStar)} we know
    that, when we view  $\omega_j$ as an element of $Ker(k_{\tau}^*)=
    Soc(\Lambda)$, we have $\omega_j + Im(R_{\tau}^*)= \omega_{j+1} +
    Im(R_{\tau}^*)$ for all $j=1, 2, \dots, n-1$. But the description of $y\gamma$ in the proof of
    lemma \ref{HH^i as $Z(Lambda)-module$} implies that $\omega_1 + Im(R_{\tau}^*)=
    \epsilon^{2n-1} + Im(R_{\tau}^*)$ is precisely the element $x_0^{n-1}y\gamma\in HH^5(\Lambda)$

    \item $\gamma t_j= \delta_{1j}x_0^{n-1}yh$
Graded commutativity gives that $t_j\gamma =\gamma t_j$ and the
element $\gamma t_j\in HH^7(\Lambda)$ is represented by the
    composition

    $$Q\stackrel{l_j}{\longrightarrow} Q \stackrel{\tilde{\gamma}}{\longrightarrow} \Lambda .$$

    Note that $l_j(e_{i(a)}\otimes e_{t(a)})\in \Lambda e_{i(a)}\otimes
    e_{t(a)}\Lambda$, from which we deduce that $\tilde{\gamma}\circ l_j=0$
    for $j\in Q_0\slash \{1\}$. And for $j=1$ we have

    $$(\tilde{\gamma}\circ l_1)(e_{i(a)}\otimes e_{t(a)})=\left\{\begin{array}{ccc}

    0\hspace{0.25cm} \text{if} \hspace{0.25cm} a\neq \epsilon\\

    \\

    \frac{1}{2}(\omega_1 + \omega_1)= \omega_1= \epsilon^{2n-1}\hspace{0.25cm} \text{if} \hspace{0.25cm}a=\epsilon \\

    \end{array}\right.$$ But, due to the identification $HH^1(\Lambda)= HH^7(\Lambda)$,
    which is just multplication by $h$, and the proof of lemma \ref{HH^i as
    $Z(Lambda)-module$}, we know that $x_0^{n-1}yh$ is precisely the
    element $\epsilon^{2n-1} + Im(\delta^*)=
    \frac{Ker(R^*)}{Im(\delta^*)}$. Therefore we get $t_j\gamma =\gamma t_j=
    \delta_{1j}x_0^{n-1}yh$, for all $j\in Q_0$.

\end{enumerate}

\end{proof}

\begin{lema} \label{HH^4 x HH^4}$\gamma^2= z_1h$

\end{lema}

\begin{proof}

We have the following commutative diagram of $\Lambda$-bimodules:

$$\begin{large}\xymatrix{P \ar[r]^{R} \ar[d]^{\beta} & Q \ar[r]^{\delta} \ar[d]^{\widehat{\gamma}}
 & P \ar[r]^{k_{\tau}} \ar[d]^{\alpha} & P \ar[r]^{R_{\tau}} \ar[d]^{\beta} & Q \ar[d]^{\widehat{\gamma}} \\
 Q \ar[r]_{\delta_{\tau}} & P \ar[r]_k & P \ar[r]_R & Q \ar[r]_{\delta} & P }\end{large},$$
where $\alpha$, $\beta$ and $\hat{\gamma}$ are the morphisms defined
in the proof of Lemma \ref{HH^2 x HH^4}. Indeed the commutativity of
the two right squares was proved in that lemma. On the other hand,
we have $\hat{\gamma}_\tau =\hat{\gamma}$ and $\beta_\tau =\beta$
(see Lemma \ref{lema.equivalencia en bimodulos proyectivos}) since
$\tau$ is fixes the vertices. The commutativity of the left most
square is then obtained from the commutativity of the right most
square by applying the equivalence $G_\tau :_\Lambda
Proj_\Lambda\stackrel{\cong}{\longrightarrow}_\Lambda Proj_\Lambda$.

It remains to prove the commutativity of the second square from left
to right. We need to prove that $(\alpha\circ\delta
)(e_{i(a)}\otimes e_{t(a)})=0$, when $a\neq\epsilon$, and that
$(\alpha\circ\delta )(e_{i(\epsilon)}\otimes e_{t(\epsilon
)})=\sum_{x\in e_1B}(-1)^{deg(x)}x\otimes x^*$. For the first
equality, we do the case $a=a_i$, with $i=1,...,n-1$,   the case
$a=\bar{a}_i$ being symmetric. We have

$$(\alpha\circ\delta )(e_{i(a_i)}\otimes e_{t(a_i)})=a_i\alpha
(e_{i+1}\otimes e_{i+1})-\alpha (e_i\otimes e_i)a_i=$$

$$=a_i[(-1)^{1+\frac{(i+1)i}{2}}\sum_{r=1}^{n-i}(-1)^{u_r}y_r^{i+1}]-[(-1)^{1+\frac{i(i-1)}{2}}\sum_{r=1}^{n-i+1}(-1)^{u_r}y_r^{i}]a_i=$$

$$=(-1)^{1+\frac{(i+1)i}{2}}\sum_{r=1}^{n-i}(-1)^{u_r}a_iy_r^{i+1}+(-1)^{\frac{i(i-1)}{2}}\sum_{r=1}^{n-i+1}(-1)^{u_r}y_r^{i}a_i.\hspace*{1cm}(\bigstar )$$
Direct calculation shows that

$$a_iy_r^{i+1}=\sum_{k=0}^{n-i-r}a_i\bar{a}_i...\bar{a}_1\epsilon^{2k}a_1...a_{r-1}\otimes\bar{a}_{r-1}...\bar{a}_1\epsilon^{2(n-i-r-k)}a_1...a_i.$$
Using remark \ref{Basis properties}(4), we then have

$$a_iy_r^{i+1}=(-1)^i \sum_{k=0}^{n-i-r}\bar{a}_{i-1}...\bar{a}_1\epsilon^{2(k+1)}a_1...a_{r-1}\otimes\bar{a}_{r-1}...\bar{a}_1\epsilon^{2(n-i-r-k)}a_1...a_i=$$

$$(-1)^i \sum_{t=1}^{n-i-r+1}\bar{a}_{i-1}...\bar{a}_1\epsilon^{2t}a_1...a_{r-1}\otimes\bar{a}_{r-1}...\bar{a}_1\epsilon^{2(n-i-r-t+1)}a_1...a_i=$$

$$(-1)^i[y_r^ia_i-\bar{a}_{i-1}...\bar{a}_1\epsilon^{2\cdot 0}a_1...a_{r-1}\otimes\bar{a}_{r-1}...\bar{a}_1\epsilon^{2(n-i-r+1)}a_1...a_i ]=(-1)^iy_r^ia_i,$$
where the last equality is due to remark \ref{Basis properties}(2).

On the other hand, for $r=n-i+1$, we have the equality

$$y_{n-i+1}^ia_i =\bar{a}_{i-1}...\bar{a}_1a_1...a_{n-i}\otimes\bar{a}_{n-i}...\bar{a}_1a_1...a_{i-1}a_i,$$
which is zero due again to remark \ref{Basis properties}(2). All
these considerations allow us to rewrite the equality $(\bigstar )$
given above as

$$(\alpha\circ\delta )(e_{i(a_i)}\otimes e_{t(a_i)})=\sum_{r=1}^{n-i}[(-1)^{u(r,i)}+(-1)^{v(r,i)}]y_r^ia_i,$$
where $u(r,i)=1+\frac{(i+1)i}{2}+i+u_r$ and
$v(r,i)=\frac{i(i-1)}{2}+u_r$. Since $u(r,i)-v(r,i)=2i+1$ we
conclude that $(\alpha\circ\delta )(e_{i(a_i)}\otimes e_{t(a_i)})=0$
as desired.

We next calculate $(\alpha\circ\delta )(e_{i(\epsilon )}\otimes
e_{t(\epsilon )})$. By definition, we have
$$\alpha (e_1\otimes
e_1)=(-1)^{1+\frac{1(1-1)}{2}}\sum_{r=1}^{n}(-1)^{u_r}y_r^1=\sum_{r=1}^{n}(-1)^{u_{r}+1}y_r^1$$Then
we have

$$(\alpha\circ\delta )(e_{i(\epsilon )}\otimes
e_{t(\epsilon )})=\epsilon\alpha (e_1\otimes e_1)-\alpha (e_1\otimes
e_1)\epsilon =\sum_{r=1}^{n}(-1)^{u_r+1}(\epsilon
y_r^1-y_r^1\epsilon ).$$ But, with the terminology of the proof of
lemma \ref{HH^2 x HH^4},  we also have the equality

$$\epsilon y_r^1-y_r^1\epsilon =\sum_{k=0}^{n-r}\epsilon^{2k+1}a_1...a_{r-1}\otimes\bar{a}_{r-1}...\bar{a}_1\epsilon^{2(n-r-k)}-
\sum_{k=0}^{n-r}\epsilon^{2k}a_1...a_{r-1}\otimes\bar{a}_{r-1}...\bar{a}_1\epsilon^{2(n-r-k)+1}=$$

$$\sum_{x\in e_rBe_1\text{, }deg(x)\equiv r-1}x^\dagger\otimes x-\sum_{x\in e_rBe_1\text{, }deg(x)\equiv r}x^\dagger\otimes
x=$$

$$(-1)^{\frac{r(r-1)}{2}}(\sum_{x\in e_rBe_1\text{,
}deg(x)\equiv r-1}x^*\otimes x-\sum_{x\in e_rBe_1\text{,
}deg(x)\equiv r}x^*\otimes x),$$ where $\equiv$ means 'congruent
modulo $2$'. Due to the fact that our basis $B$ is dualizable and
that $deg(x)+deg(x^*)\equiv 1$, for all $x\in B$,  we deduce that

$$\epsilon y_r^1-y_r^1\epsilon = (-1)^{\frac{r(r-1)}{2}}(\sum_{y\in e_1Be_r\text{, }deg(y)\equiv
r}y\otimes y^*-\sum_{y\in e_1Be_r\text{, }deg(y)\equiv r-1}y\otimes
y^*),$$ and hence

$$(-1)^{u_r+1}(\epsilon y_r^1-y_r^1\epsilon )= (-1)^{\frac{r(r-1)}{2}+u_r+1}(\sum_{y\in e_1Be_r\text{, }deg(y)\equiv
r}y\otimes y^*-\sum_{y\in e_1Be_r\text{, }deg(y)\equiv r-1}y\otimes
y^*),$$ for each $r=1,...,n$. Since $\frac{r(r-1)}{2}+u_r+1\equiv r$
it follows that the coefficient of $y\otimes y^*$ in the last
expression is precisely $(-1)^{deg(y)}$, for each $y\in e_1Be_r$.
From that the desired equality $(\alpha\circ\delta )(e_{i(\epsilon
)}\otimes e_{t(\epsilon )})=\sum_{y\in e_1B}(-1)^{deg(y)}y\otimes
y^*$ follows.

Once the commutativity of the initial diagram is proved, we get that
$\gamma^2$ is represented by the morphism of $\Lambda$-bimodules
$P\stackrel{\beta}{\longrightarrow}P\stackrel{\tilde{\gamma}}{\longrightarrow}\Lambda$,
which satisfies that $(\tilde{\gamma}\circ\beta )(e_i\otimes
e_i)=\delta_{i1}e_1$, for each $i\in Q_0$. This morphism represents
the element $e_1+Im(R^*)\in\frac{Ker(k_*)}{Im(R^*)}=HH^2(\Lambda
)=HH^8(\Lambda )$, which is precisely the element $z_1h$. Therefore
we have $\gamma^2=z_1h$ in $HH^*(\Lambda )$.

\end{proof}

We are now ready to prove  the two main theorems of the paper.

\vspace{1cm}

PROOF OF THE THEOREMS \ref{HH^* with Char not 2 not dividing 2n+1}
and \ref{HH^* with char 2n+1}: With the notation used until now, we
know from Proposition \ref{K-basis for HH^*Lambda1}  that the set
$\{x_0, x_1, \dots, x_n, y, z_1, \dots, z_n, t_1, \dots t_{n},
\gamma, h\}$ generates $HH^*(\Lambda)$ as an algebra. If $Char(K)$
does not divide $2n+1$, then lemma \ref{HH*1xHH*2} that each $t_i$
is a $K$-linear combination of the $yz_j$, which allows us to drop
all the $t_i$ from the set of generators. On the other hand, if
$Char(K)$ divides $2n+1$, then lemma \ref{HH*1xHH*2} tells us that
the $K$-subspace of $HH^3(\Lambda)$ generated by $\{yz_1, \dots
yz_n\}$ is $1$-dimensional. In the proof of this lemma it is
actually proved that $yz_j + (-1)^j(2j-1)yz_1=0$ since this is the
relationship between the rows (=columns) of the matrix $C$. In
addition, looking at the first row of this matrix, we see that
$yz_1=-nt_1+ (n-1)t_2 + \cdots (-1)^nt_n$, and hence $\{t_1, \dots,
t_{n-1}, yz_1\}$ is a basis of $HH^3(\Lambda)$. So, in case
$Char(K)$ divides $2n+1$, we can drop $t_n$ from the set of
generators. This proves that the set of generators given in Theorems
\ref{HH^* with Char not 2 not dividing 2n+1} and \ref{HH^* with char
2n+1} is correct.

From the fact that $Soc(\Lambda)HH^j(\Lambda)=0$ $\forall j>0$ one
readily obtains the relations in $i)$. From lemmas \ref{Center of
lambda}, \ref{HH^i as $Z(Lambda)-module$} and \ref{products
involving HH^3} we obtain all the relations in $ii)$ (in both cases)
except $y^2=0$. To see that this equality also holds, consider the
morphism of $\Lambda$-bimodules $\tilde{y}: \oplus_{a\in Q_1}\Lambda
e_{i(a)}\otimes e_{t(a)}\Lambda \longrightarrow \Lambda$ defined by
$\tilde{y}(e_{i(a)}\otimes e_{t(a)})=a$ $\forall a\in Q_1$ that
represents $y\in HH^1(\Lambda)$. If $\widehat{y}: \oplus_{a\in
Q_1}\Lambda e_{i(a)}\otimes e_{t(a)}\Lambda\longrightarrow \Lambda
e_i\otimes e_i \Lambda$ is the only morphism of $\Lambda$-bimodules
such that $\widehat{y}(e_{i(a)}\otimes e_{t(a)})=
\frac{1}{2}(a\otimes e_{t(a)} + e_{i(a)}\otimes a)$, for all $a\in
Q_1$, then we have $u\circ \widehat{y}= \tilde{y}$, where $u:
\oplus_{i\in Q_0}\Lambda e_i\otimes e_i\Lambda\longrightarrow
\Lambda$ is the multiplication map. Note now that we have a
commutative diagram in the category of $\Lambda$-bimodules:

$$\begin{large}\xymatrix{P\ar[r]^R \ar[d]_{\eta} & Q \ar[d]^{\widehat{y}} \\
Q \ar[r]_{\delta} & P}\end{large}$$ where $\eta$ is the only
morphism of $\Lambda$-bimodules such that $\eta(e_i\otimes e_i)=
\sum_{a\in Q_1, i(a)=1}e_{i(a)}\otimes \bar{a}$.

The definition of Yoneda product says that $y^2$ is represented by
the composition

$$\oplus_{i\in Q_0}\Lambda e_i\otimes e_i\lambda\stackrel{\eta}{\longrightarrow}
\oplus_{a\in Q_1}\Lambda e_{i(a)}\otimes e_{t(a)} \Lambda
\stackrel{\tilde{y}}{\longrightarrow}\Lambda$$ which takes
$e_i\otimes e_i\rightsquigarrow \sum_{a\in Q_1, i(a)=i}a\bar{a}=0$.

For the relations in $iii)$ note that the proof of lemma \ref{HH^i
as $Z(Lambda)-module$} gives an isomorphism $\phi_y: HH^4(\Lambda)
\stackrel{\sim}{\longrightarrow} HH^5(\Lambda)$ $(\xi
\rightsquigarrow y\xi)$. What we shall prove is the equality

$$z_j(yz_k)= (-1)^{k-j+1}(2j-1)(n-k+1)x_0^{n-1}y\gamma ,$$
from which the desired equality will follow.

Indeed, by lemma \ref{HH*1xHH*2}, we have that $yz_k= \sum_{l=1}^n
c_{lk}t_l$ and hence $z_j(yz_k)= \sum_{l=1}^n c_{lk}z_jt_l$. From
lemma \ref{products involving HH^3}(2) we then get

$$z_j(yz_k)= c_{jk}z_jt_j= (-1)^{k-j+1}(2j-1)(n-k+1)x_0^{n-1} y\gamma$$

Note that lemma \ref{products involving HH^3}(2) also proves the
relations in $vii)$ for Theorem \ref{HH^* with char 2n+1}. We claim
that it also gives the relations $iv)$. Indeed multiplication by $h$
gives an isomorphism of $Z(\Lambda)$-modules
$HH^1(\Lambda)\stackrel{\sim}{\longrightarrow}HH^7(\Lambda)$. In
particular $yh$ is the canonical generator of $HH^7(\Lambda)$ as
$Z(\Lambda)$-module, which implies that multiplication by $y$ gives
an isomorphism of $Z(\Lambda)$-modules $HH^6(\Lambda)
\stackrel{\sim}{\longrightarrow}HH^7(\Lambda)$. Similar to the
previous paragraph, we shall prove the equality

$$\gamma(yz_j)= (-1)^j(n-j+1)x_0^{n-1}yh,$$
from which the relations in $iv)$ will follow.

Again we have $yz_j= \sum_{l=1}^n c_{lj}t_l$ and so

$$\gamma(yz_j)= \sum_{l=1}^n c_{lj}\gamma t_l = c_{1j}x_0^{n-1}yh=
(-1)^j(n-j+1)x_0^{n-1}yh,$$ using the relations in viii) which have
already been proved in lemma \ref{products involving HH^3}(3).

Finally, lemma \ref{HH^4 x HH^4} gives the relations in $v)$ while,
when $Char(K)$ divides $2n+1$, the equality $yz_j +
(-1)^j(2j-1)yz_1=0$ mentioned in the first paragraph of this proof
gives the relations in $vi)$ of Theorem \ref{HH^* with char 2n+1}.

The previous paragraphs show that there is a surjective homomorphism
of graded algebras from the algebra given by the mentioned
generators and relations to the algebra $HH^*(\Lambda)$. By looking
at dimensions in each degree, it is not difficult to see that it is
actually an isomorphism.

\vspace*{0.3cm}
\begin{rem} \rm
In \cite{Eu}[Section 9] the graded ring structure of $HH^*(\Lambda
)$ was calculated taking $\mathbf{C}$ as ground field. However,  the
arguments and calculations appear to be valid whenever $char(K)\neq
2$ and $char(K)$ does not divide $2n+1$. Then, with the suitable
changes derived from the different presentations of the algebra, our
Theorem \ref{HH^* with Char not 2 not dividing 2n+1} could be
derived from Eu's work.

Eu's methods use  sometimes direct calculation of the products
$HH^i(\Lambda )\cdot HH^j(\Lambda )$, sometimes  the graded
condition of the minimal projective resolution of $\Lambda$ (see
9.2) and sometimes the matrix Hilbert series $H_\Lambda (t)$ (see
Definition 2.5.2) together with the equality
 $H_\Lambda (t)=(1+t^{2n+1})((1+t^2)I_n-Dt)^{-1}$ proved in
 \cite{MOV}, where $D$ is the adjacency matrix of $\mathbb{L}_n$ (see the proof of  Lemma 9.3.3 and  section 6.2 in \cite{Eu}).

 We have not used these tools in our paper and have directly
 calculated all products $HH^i(\Lambda )\cdot HH^j(\Lambda )$
 working with the bases of Proposition \ref{K-basis for HH^*Lambda1}
 which already included in themselves some products.
\end{rem}

\begin{cor} \label{cor:presentation:stable cohomology ring}

Let us fix the presentation of $HH^*(\Lambda)$ given by Theorems
\ref{HH^* with Char not 2 not dividing 2n+1} or \ref{HH^* with char
2n+1}. A presentation of $\underline{HH}^*(\Lambda )$ is obtaining
from it by doing the following:
\begin{enumerate}
\item Replace the generators
 $x_1,...,x_n$ by a new generator $h'$ of degree $-6$ \item Replace
 the relations i) in the list by a new relation $hh'=1$.
 \item Leave the remaining generators and relations unchanged.
 \end{enumerate}
\end{cor}

\begin{proof}
It is immediately seen that the graded commutative algebra given by
the just described generators and relations is isomorphic to
$HH^*(\Lambda )_{(h)}$, whence isomorphic to
$\underline{HH}^*(\Lambda )$ (see Proposition \ref{localization}).
\end{proof}

\section{Acknowledgements}

\hspace{0.4cm} The first part of this work was done while I was at
University of Oxford as a Ph. D. student under the supervision of
Karin Erdmann. I would like to thank her for guidance and kindness.

\vspace{0.3cm} I am indebted to my advisor Manuel Saorín for his
suggestions and contributions which improved and led to a better
presentation of this paper.

\end{document}